\documentclass[11pt, reqno]{amsart}

\usepackage{amsmath, amsthm, amssymb}
\usepackage{enumitem}
\usepackage{pdflscape}
\usepackage{caption}
\usepackage{bm}

\usepackage{ifpdf}
\ifpdf
\usepackage[pdftex]{graphicx}
\else
\usepackage[dvips]{graphicx}
\fi
\usepackage{tikz}
 	 \usetikzlibrary{arrows,backgrounds}
\usepackage[all]{xy}

\usepackage{multicol}

\usepackage{tocvsec2}

\usepackage{bbm}

\input xy
\xyoption{all}

\usepackage[pdftex,plainpages=false,hypertexnames=false,pdfpagelabels]{hyperref}
\newcommand{\arxiv}[1]{\href{http://arxiv.org/abs/#1}{\tt arXiv:\nolinkurl{#1}}}
\newcommand{\arXiv}[1]{\href{http://arxiv.org/abs/#1}{\tt arXiv:\nolinkurl{#1}}}
\newcommand{\doi}[1]{\href{http://dx.doi.org/#1}{{\tt DOI:#1}}}

\newcommand{\googlebooks}[1]{(preview at \href{http://books.google.com/books?id=#1}{google books})}

\usepackage{xcolor}
\definecolor{dark-red}{rgb}{0.7,0.25,0.25}
\definecolor{dark-blue}{rgb}{0.15,0.15,0.55}
\definecolor{medium-blue}{rgb}{0,0,.8}
\definecolor{DarkGreen}{RGB}{0,150,0}
\definecolor{rho}{named}{red}
\hypersetup{
   colorlinks, linkcolor={purple},
   citecolor={medium-blue}, urlcolor={medium-blue}
}

\usepackage{longtable}
\usepackage{fullpage}
\numberwithin{equation}{section}

\theoremstyle{plain}
\newtheorem{thm}[equation]{Theorem}
\newtheorem*{thm*}{Theorem}
\newtheorem{thmalpha}{Theorem}

\newtheorem*{cor*}{Corollary}

\newtheorem*{conj*}{Conjecture}
\newtheorem{lem}[equation]{Lemma}
\newtheorem{prop}[equation]{Proposition}

\newtheorem*{quest*}{Question}
\newtheorem*{claim*}{Claim}

\theoremstyle{definition}
\newtheorem{defn}[equation]{Definition}
\newtheorem{fact}[equation]{Fact}

\newtheorem{construction}[equation]{Construction}
\newtheorem{alg}[equation]{Algorithm}

\newtheorem{ex}[equation]{Example}

\newtheorem{sub-ex}[equation]{Sub-Example}
\newtheorem{rem}[equation]{Remark}
\newtheorem*{rem*}{Remark}

\DeclareMathOperator{\Ad}{Ad}

\DeclareMathOperator{\coev}{coev}

\DeclareMathOperator{\ev}{ev}

\DeclareMathOperator{\Hom}{Hom}

\DeclareMathOperator{\id}{id}

\DeclareMathOperator{\mate}{mate}
\DeclareMathOperator{\Forget}{Forget}

\newcommand{\comment}[1]{}

\newcommand{\be}{\begin{enumerate}[label=(\arabic*)]}
\newcommand{\ee}{\end{enumerate}}

\def\semicolon{;}
\def\applytolist#1{
    \expandafter\def\csname multi#1\endcsname##1{
        \def\multiack{##1}\ifx\multiack\semicolon
            \def\next{\relax}
        \else
            \csname #1\endcsname{##1}
            \def\next{\csname multi#1\endcsname}
        \fi
        \next}
    \csname multi#1\endcsname}

\def\calc#1{\expandafter\def\csname c#1\endcsname{{\mathcal #1}}}
\applytolist{calc}QWERTYUIOPLKJHGFDSAZXCVBNM;
\def\bbc#1{\expandafter\def\csname bb#1\endcsname{{\mathbb #1}}}
\applytolist{bbc}QWERTYUIOPLKJHGFDSAZXCVBNM;
\def\bfc#1{\expandafter\def\csname bf#1\endcsname{{\mathbf #1}}}
\applytolist{bfc}QWERTYUIOPLKJHGFDSAZXCVBNM;
\def\sfc#1{\expandafter\def\csname s#1\endcsname{{\sf #1}}}
\applytolist{sfc}QWERTYUIOPLKJHGFDSAZXCVBNM;
\def\fc#1{\expandafter\def\csname f#1\endcsname{{\mathfrak #1}}}
\applytolist{fc}QWERTYUIOPLKJHGFDSAZXCVBNM;

\newcommand{\Hilb}{\mathsf{Hilb}}
\newcommand{\Vect}{\mathsf{Vect}}
\newcommand{\sVect}{\mathsf{sVect}}
\newcommand{\vcat}{\mathsf{VCat}}

\newcommand{\vmod}{\mathsf{VMod}}

\newcommand{\vmoncat}{\mathsf{VMonCat}}
\newcommand{\vmodtens}{\mathsf{VModMon}}

\newcommand{\dagucat}{\mathsf{UCat_\dagger}}
\newcommand{\dagumod}{\mathsf{UMod_\dagger}}
\newcommand{\dagumoncat}{\mathsf{UMonCat_\dagger}}
\newcommand{\dagumodmon}{\mathsf{UModMon_\dagger}}
\newcommand{\dagvcat}{\mathsf{VCat_\dagger}}
\newcommand{\dagvmod}{\mathsf{VMod_\dagger}}
\newcommand{\dagvmoncat}{\mathsf{VMonCat_\dagger}}
\newcommand{\dagvmodmon}{\mathsf{VModMon_\dagger}}

\newcommand{\noshow}[1]{}
\renewcommand{\MR}[1]{}

\usetikzlibrary{shapes}
\usetikzlibrary{cd}
\usetikzlibrary{backgrounds}
\usetikzlibrary{decorations,decorations.pathreplacing,decorations.markings}
\usetikzlibrary{fit,calc,through}
\usetikzlibrary{external}
\usetikzlibrary{arrows}
\tikzstyle{knot}=[preaction={super thick, white, draw}]
\tikzset{smallstring/.style={thick,scale=0.75,every node/.style={transform shape}}}
\tikzset{vertex/.style = {shape=circle,draw,fill=black,inner sep=0pt,minimum size=5pt}}
\tikzset{edge/.style = {->,> = latex', bend right}}
\tikzset{
	super thick/.style={line width=3pt}
}
\tikzset{
    quadruple/.style args={[#1] in [#2] in [#3] in [#4]}{
        #1,preaction={preaction={preaction={draw,#4},draw,#3}, draw,#2}
    }
}
\tikzstyle{shaded}=[fill=red!10!blue!20!gray!30!white]
\tikzstyle{unshaded}=[fill=white]
\tikzstyle{empty box}=[circle, draw, thick, fill=white, opaque, inner sep=2mm]
\tikzstyle{annular}=[scale=.7, inner sep=1mm, baseline]
\tikzstyle{rectangular}=[scale=.75, inner sep=1mm, baseline=-.1cm]
\tikzstyle{mid>}=[decoration={markings, mark=at position 0.5 with {\arrow{>}}}, postaction={decorate}]
\tikzstyle{mid<}=[decoration={markings, mark=at position 0.5 with {\arrow{<}}}, postaction={decorate}]
\tikzstyle{over}=[double, draw=white, super thick, double=]
\tikzstyle{box} = [rectangle,draw,rounded corners=5pt,very thick]

\newcommand{\roundNbox}[6]{
	\draw[rounded corners=5pt, very thick, #1] ($#2+(-#3,-#3)+(-#4,0)$) rectangle ($#2+(#3,#3)+(#5,0)$);
	\coordinate (ZZa) at ($#2+(-#4,0)$);
	\coordinate (ZZb) at ($#2+(#5,0)$);
	\node at ($1/2*(ZZa)+1/2*(ZZb)$) {#6};
}

  \newcommand{\tikzmath}[2][]
     {\vcenter{\hbox{\begin{tikzpicture}[#1]#2
                     \end{tikzpicture}}}
     }

\newcommand{\modulebox}[4]{%bottom left (x,y), width, height, label
  \draw (#1) --
  ++(0,#3) {[rounded corners=3] --
    ++(#2,0) --
    ++(0,-#3)} --
  cycle {};
  \node at ($(#1)+(#2/2, #3/2)$) {#4};
  }

\begin{document}
\title{Unitary braided-enriched monoidal categories}
\author{Zachary Dell, Peter Huston, and David Penneys}
\date{\today}
\begin{abstract}
Braided-enriched monoidal categories were introduced in work of Morrison-Penneys, where they were characterized using braided central functors.
Recent work of Kong-Yuan-Zhang-Zheng and Dell extended this characterization to an equivalence of 2-categories.
Since their introduction, braided-enriched fusion categories have been used to describe certain phenomena in topologically ordered systems in theoretical condensed matter physics.
While these systems are unitary, there was previously no general notion of unitarity for enriched categories in the literature.
We supply the notion of unitarity for enriched categories and braided-enriched monoidal categories and extend the above 2-equivalence to the unitary setting.
%This is the submitted version of \arxiv{???}.
\end{abstract}
\maketitle

%%%%%%%%%%%%%%%%%%%%%%%%%%%%%%%%%%%%%%%%%%%%%%%%%%%%%%%%%%%%%%%
%%%%%%%%%%%%%%%%%%%%%%%%%%%%%%%%%%%%%%%%%%%%%%%%%%%%%%%%%%%%%%%
%%%%%%%%%%%%%%%%%%%%%%%%%%%%%%%%%%%%%%%%%%%%%%%%%%%%%%%%%%%%%%%
\section{Introduction}

The article \cite{MR3961709} introduced the notion of braided-enriched monoidal category and showed they are equivalent to rigid module monoidal categories $\sV \to \sZ(\sA)$ in the spirit of \cite{MR3578212}.
This characterization was recently extended to an equivalence of 2-categories by \cite{2104.03121,2104.07747}.
The article \cite{2104.03121} extends this 2-equivalence even further, allowing for changes in the enriching braided monoidal category $\sV$.

Braided-enriched monoidal categories are part of a larger story of \emph{enriched quantum symmetry}.
Here, by \emph{quantum symmetry}, we mean that while classical symmetry is described by the notion of a group, quantum mathematical objects naturally live in higher categories, and so their symmetries are better described by tensor categories.
The notion of enriched quantum symmetry describes tensor category symmetries of 
fermionic systems (enriched over $\sVect$) \cite{MR3623246,1709.01941,2109.10911,PhysRevB.105.155126,2109.11039},
and tensor category symmetries of anomalous quantum systems where the anomaly is described by the braided-enriched category $\sV$.
Indeed, 
%monoidal supercategories have been studied in \cite{},
%and 
the recent article \cite{2208.14018} uses braided-enriched unitary fusion categories to study generalized symmetries (domain walls) of anomalous (2+1)D topological orders.

Other uses of braided-enriched monoidal categories include a unified treatment of gapped and gapless topological domain walls between topologically ordered phases of matter \cite{1905.04924,1912.01760}, and an exploration of certain unitary fusion categories arising from subfactor theory, namely the $\Ad(E_8)$ and $\Ad(4442)$ fusion categories \cite{MR3314808,MR4079744}.
Other notions related to braided-enriched monoidal categories include planar para algebras \cite{MR3623255} and  anchored planar algebras \cite{1607.06041}.

While many of these applications are unitary, especially those to theoretical condensed matter physics, the notion of \emph{unitarity} for enriched quantum symmetries is overlooked in many of these applications, with the exception of \cite{MR3623255}.
In this article, we define the notion of unitarity for braided-enriched monoidal categories and we prove a characterization theorem in the spirit of \cite{MR3961709,2104.07747}.
We leave out an investigation of the change of enrichment as it would take us too far afield.

As in \cite{1809.09782}, we begin one categorical level down with the notion of dagger category $\cA$ enriched in an \emph{involutive} monoidal category $(\sV,\overline{\,\cdot\,})$ \cite{MR2861112}.
We are inspired by $\rm H^*$-categories and 2-Hilbert spaces from \cite{MR1448713}, and we define an \emph{enriched dagger}\footnote{After completing this manuscript, we became aware of a seminar talk of Egger also defining an enriched dagger structure on an involutive $\sV$-category; his convention is inverse to ours \cite[42:40]{EggerEnrichedDagger}.
}
$$
\kappa_{a\to b} : \overline{\cA(b\to a)} \to \cA(a\to b)
$$
satisfying certain properties analogous to those in an $\rm H^*$-category (see \S\ref{sec:EnrichedDaggerCats} below).

To prove a characterization theorem, we focus on the case that $\sV=\sU$ is a \emph{unitary monoidal category}, i.e., a rigid dagger monoidal category equipped with a \emph{unitary dual functor} in the sense of \cite{MR4133163}.
In this setting, $\sU$ comes equipped with a canonical \emph{bi-involutive structure} $(\dag,\overline{\,\cdot\,}, r, \nu,\varphi)$ \cite{MR3663592};
$\sU$ is equipped with three involutions $\dag,*, \overline{\,\cdot\,}$ which all commute, and composing two in either order produces the third.

\begin{thmalpha}
  \label{thm:EnrichedDaggerEquivalence}
  Let $\sU$ be a unitary monoidal category.
  There is a strict locally isomorphic 2-equivalence of strict 2-categories
  \[
  \left\{\,
  \parbox{4.9cm}{\rm Tensored dagger $\sU$-categories}
  \,\right\}
  \,\,\cong\,\,
  \left\{\,
  \parbox{6.3cm}{\rm Tensored dagger $\sU$-module categories}
  \,\right\}.
  \]
\end{thmalpha}

Here, a $\sU$-category $\cA$ is called \emph{tensored} if every $\sU$-representable functor $\cA(a\to -): \cA\to \cU$ admits a left $\sU$-adjoint \cite[\S1.11]{MR2177301}, where $\cU$ is the \emph{self-enrichment} of $\sU$ (see Examples \ref{ex:InvolutiveSelfEnrichment} and \ref{ex:DaggerSelfEnrichment}).
A right (strong) $\sU$-module category $\sM$ is called \emph{tensored} if each functor $m\lhd - : \sU \to \sM$ admits a right adjoint.

This 2-equivalence is called \emph{locally isomorphic} \cite{2104.03121} as the hom functors on hom categories are all \emph{isomorphisms} of categories.
Since the 2-category of tensored dagger $\sU$-module categories is naturally a dagger 2-category, this allows us to transport a dagger structure to the 2-category of tensored dagger $\sU$-categories along the strict locally isomorphic 2-equivalence.

%When $\sU$ is a unitary tensor category (a semisimple rigid $\rm C^*$ tensor category with simple unit object), the unitary $\sU$-enriched categories are those whose underlying categories on the right hand side are tensored $\sU$-module $\rm C^*/W^*$-categories.

After establishing this theorem, we move to the braided-enriched setting, where our enriched dagger $\kappa$ on $\cA$ must also be compatible with the $\sV$-tensor product $-\otimes_\cA-$ (see \S\ref{sec:BraidedEnrichedDaggerStuff} below).
We prove our main characterization theorem when $\sV=\sU$ is a braided unitary monoidal category, i.e., a unitary monoidal category equipped with a unitary braiding.

\begin{thmalpha}
  \label{thm:EnrichedMonoidalDaggerEquivalence}
  Let $\sU$ be a braided unitary monoidal category.
  There is a strict locally isomorphic 2-equivalence of strict 2-categories
  \[
  \left\{\,
  \parbox{3.8cm}{\rm Tensored rigid dagger $\sU$-monoidal categories}
  \,\right\}
  \,\,\cong\,\,
  \left\{\,
  \parbox{4.8cm}{\rm Tensored rigid dagger $\sU$-module monoidal categories }
  \,\right\}.
  \]
\end{thmalpha}

Here, a rigid $\sU$-monoidal category $\cA$ is called \emph{tensored} if the underlying $\sU$-category of $\cA$ is tensored, equivalently, the $\sU$-representable functor $\cA(1_\cA\to -): \cA\to \cU$ admits a left $\sU$-adjoint.
A $\sU$-module monoidal category \cite{MR3578212} is a pair $(\sA,\sF^Z)$ consisting of a rigid dagger monoidal category $\sA$ together with a braided dagger module functor $\sF^Z: \sA\to Z^\dag(\sA)$, the dagger Drinfeld center of $\sA$ \cite[Def.~6.1]{MR1966525}.
A rigid dagger $\sU$-module monoidal category $(\sA,\sF^Z)$ is called \emph{tensored} if $\sF:=\sF^Z\circ \Forget_Z$ admits a right adjoint, where $\Forget_Z: Z(\sA)\to \sA$ is the forgetful functor.

In future applications, we hope to use the notion of unitary enrichment to study notions of unitarity for higher categories and a unitary version of higher idempotent completion \cite{1812.11933,1905.09566} building on \cite{MR4419534,MR4369356,2108.09872}.
We also plan on investigating unitarity for anchored planar algebras building on \cite{MR3578212,1607.06041}.

\subsection*{Acknowledgments}
The authors would like to thank Andr\'e Henriques, David Reutter, and Jan Steinebrunner for helpful conversations.
The authors were supported by NSF DMS grants 1654159 and 2154389.

%%%%%%%%%%%%%%%%%%%%%%%%%%%%%%%%%%%%%%%%%%%%%%%
%%%%%%%%%%%%%%%%%%%%%%%%%%%%%%%%%%%%%%%%%%%%%%%
%%%%%%%%%%%%%%%%%%%%%%%%%%%%%%%%%%%%%%%%%%%%%%%
\section{Background}

In this article, composition of maps is written in the reverse order; i.e., if $f: a\to b$ and $g: b\to c$, then we write $f\circ g: a\to c$.

%%%%%%%%%%%%%%%%%%%%%%%%%%%%%%%%%%%%%%%%%%%%%%%
\subsection{Enriched categories}
Suppose $\sV$ is a closed monoidal category.
We suppress tensor symbols, associators, and unitors whenever possible.

\begin{defn}[\cite{MR2177301}]
  A $\sV$-category $\cA$ consists of a collection of objects, a hom object $\cA(a \to b)$ in $\sV$ for each pair of objects $a, b$ in $\cA$, and a distinguished identity morphism $j_a \in \sV(1_\sV \to \cA(a \to a))$.
  Further, there is a distinguished composition morphism $- \circ_\cA - \in \sV (\cA(a \to b) \cA(b\to c)\to \cA(a\to c))$ for each triple of objects $a, b, c$ in $\cA$.
  This data must satisfy the following identity and associativity axioms:
  \begin{align*}
    \begin{tikzpicture}[smallstring, baseline=30]
      \node (top) at (1,3) {$\cA ( a \to b )$};
      \node (bot) at (2,0) {$\cA ( a \to b )$};
      \node[box] (j) at (0,1) {$j_a$};
      \node[box] (circ) at (1,2) {$- \circ_{\cA} -$};
      \draw (j) to[in=-90,out=90] (circ.225);
      \draw (bot) to[in=-90,out=90] (circ.-45);
      \draw (circ) to[in=-90,out=90] (top);
    \end{tikzpicture}
    =
    \begin{tikzpicture}[smallstring, baseline=30]
      \node(bot) at (0,0) {$\cA ( a \to b )$};
      \node (top) at (0,3) {$\cA ( a \to b )$};
      \draw (top) to (bot);
    \end{tikzpicture}
    =
    \begin{tikzpicture}[smallstring, baseline=30]
      \node (top) at (1,3) {$\cA ( a \to b )$};
      \node (bot) at (0,0) {$\cA ( a \to b )$};
      \node[box] (j) at (2,1) {$j_b$};
      \node[box] (circ) at (1,2) {$- \circ_{\cA} -$};
      \draw (j) to[in=-90,out=90] (circ.-45);
      \draw (bot) to[in=-90,out=90] (circ.225);
      \draw (circ) to[in=-90,out=90] (top);
    \end{tikzpicture}
  \end{align*}
  and
  \begin{align*}
    \begin{tikzpicture}[smallstring, baseline=30]
      \node (ab) at (0,0) {$\cA ( a \to b )$};
      \node (bc) at (2,0) {$\cA ( b \to c )$};
      \node (cd) at (4,0) {$\cA ( c \to d )$};
      \node (top) at (2,3) {$\cA ( a \to d )$};
      \node[box] (circ1) at (1,1) {$- \circ_{\cA} -$};
      \node[box] (circ2) at (2,2) {$- \circ_{\cA} -$};
      \draw (ab) to[in=-90,out=90] (circ1.225);
      \draw (bc) to[in=-90,out=90] (circ1.-45);
      \draw (circ1) to[in=-90,out=90] (circ2.225);
      \draw (cd) to[in=-90,out=90] (circ2.-45);
      \draw (circ2) to[in=-90,out=90] (top);
    \end{tikzpicture}
    =
    \begin{tikzpicture}[smallstring, baseline=30]
      \node (ab) at (0,0) {$\cA ( a \to b )$};
      \node (bc) at (2,0) {$\cA ( b \to c )$};
      \node (cd) at (4,0) {$\cA ( c \to d )$};
      \node (top) at (2,3) {$\cA ( a \to d )$};
      \node[box] (circ1) at (3,1) {$- \circ_{\cA} -$};
      \node[box] (circ2) at (2,2) {$- \circ_{\cA} -$};
      \draw (ab) to[in=-90,out=90] (circ2.225);
      \draw (bc) to[in=-90,out=90] (circ1.225);
      \draw (circ1) to[in=-90,out=90] (circ2.-45);
      \draw (cd) to[in=-90,out=90] (circ1.-45);
      \draw (circ2) to[in=-90,out=90] (top);
    \end{tikzpicture}.
  \end{align*}
\end{defn}

\begin{ex}
  \label{ex:SelfEnrichment}
  When $\sV$ is a rigid monoidal category (or more generally closed monoidal), we can form the self-enrichment $\cV$ whose hom objects are $\cV(u\to v):= u^* v$ and composition and unit morphisms are given by
  $$
  -\circ_\cV-
  :=
  \tikzmath{
    \draw[thick] (-.3,-.6) node[below, yshift=-.05cm]{$\scriptstyle v$} arc (180:0:.3cm) node[below, yshift=.05cm]{$\scriptstyle v^*$}; 
    \draw[thick] (-.6,-.6) node[below, yshift=.05cm]{$\scriptstyle u^*$} to[in=-90,out=90] (-.3,.3) node[above]{$\scriptstyle u^*$};
    \draw[thick] (.6,-.6) node[below, yshift=-.05cm]{$\scriptstyle w$} to[in=-90,out=90] (.3,.3) node[above]{$\scriptstyle w$};
  }
  =
  \id_{u^*} \ev_v \id_w
  \qquad\qquad\text{and}\qquad\qquad
  j_v:= 
  \tikzmath{
    \draw[thick] (-.3,0) node[above]{$\scriptstyle v^*$} arc (-180:0:.3cm) node[above]{$\scriptstyle v$}; 
  }
  =
  \coev_v.
  $$
  Implicit here are our conventions for duals in $\sV$.
  We make this explicit in Definition \ref{defn:Rigid} below.
\end{ex}

\begin{defn}
  A $\sV$-functor $\cF$ between $\sV$-categories $\cA$ and $\cB$ is a function on objects, and an assignment of a morphism $\cF_{a \to b} \in \sV ( \cA (a \to b) \to \cB (\cF(a) \to \cF(b)) )$ satisfying functoriality and unitality axioms:
  \begin{align*}
    \begin{tikzpicture}[smallstring, baseline=30]
      \node (ab) at (0,0) {$\cA ( a \to b )$};
      \node (bc) at (2,0) {$\cA ( b \to c )$};
      \node (top) at (1,3) {$\cB ( \cF ( a ) \to \cF ( c ) )$};
      \node[box] (circ) at (1,1) {$- \circ_{\cA} -$};
      \node[box] (r) at (1,2) {$\cF_{a \to c}$};
      \draw (ab) to[in=-90,out=90] (circ.225);
      \draw (bc) to[in=-90,out=90] (circ.-45);
      \draw (circ) to[in=-90,out=90] (r);
      \draw (r) to[in=-90,out=90] (top);
    \end{tikzpicture}
    =
    \begin{tikzpicture}[smallstring, baseline=30]
      \node (ab) at (0,0) {$\cA ( a \to b )$};
      \node (bc) at (2,0) {$\cA ( b \to c )$};
      \node (top) at (1,3) {$\cB ( \cF ( a ) \to \cF ( b ) )$};
      \node[box] (rab) at (0,1) {$\cF_{a \to b}$};
      \node[box] (rbc) at (2,1) {$\cF_{b \to c}$};
      \node[box] (circ) at (1,2) {$- \circ_{\cB} -$};
      \draw (ab) to[in=-90,out=90] (rab);
      \draw (bc) to[in=-90,out=90] (rbc);
      \draw (rab) to[in=-90,out=90] (circ.225);
      \draw (rbc) to[in=-90,out=90] (circ.-45);
      \draw (circ) to[in=-90,out=90] (top);
    \end{tikzpicture}
    \quad \text{and} \quad
    \begin{tikzpicture}[smallstring, baseline=15]
      \node (top) at (0,2) {$\cB ( \cF ( a ) \to \cF ( a ) )$};
      \node[box] (j) at (0,0) {$j_a^{\cA}$};
      \node[box] (r) at (0,1) {$\cF_{a \to a}$};
      \draw (j) to (r);
      \draw (r) to (top);
    \end{tikzpicture}
    =
    \begin{tikzpicture}[smallstring, baseline=5]
      \node (top) at (0,1) {$\cB ( \cF ( a ) \to \cF ( a ) )$};
      \node[box] (j) at (0,0) {$j_{\cF ( a )}^{\cB}$};
      \draw (j) to (top);
    \end{tikzpicture}
  \end{align*}
\end{defn}
\begin{defn}
  Given objects $a, b$ in a $\sV$-category $\cA$, a \emph{$1_\sV$-graded morphism} from $a$ to $b$ is a morphism in $\sV ( 1_\sV \to \cA(a \to b))$. 
  A \emph{$\sV$-natural transformation} between $\sV$-functors $\cF, \cG : \cA \to \cB$ is a collection of $1_\sV$-graded morphisms $\lambda_a \in \sV(1_\sV \to \cB(\cF(a) \to \cG(a)))$ that is natural, i.e.,
  \begin{equation}
    \begin{tikzpicture}[smallstring, baseline=35]
      \node (ab) at (2,0) {$\cA ( a \to b )$};
      \node (top) at (1,3) {$\cB ( \cF ( a ) \to \cG ( b ) )$};
      \node[box] (lambda) at (0,1) {$\lambda_a$};
      \node[box] (g) at (2,1) {$\cG_{a \to b}$};
      \node[box] (circ) at (1,2) {$- \circ_{\cB} -$};
      \draw (ab) to[in=-90,out=90] (g);
      \draw (lambda) to[in=-90,out=90] (circ.225);
      \draw (g) to[in=-90,out=90] (circ.-45);
      \draw (circ) to[in=-90,out=90] (top);
    \end{tikzpicture}
    =
    \begin{tikzpicture}[smallstring, baseline=35]
      \node (ab) at (0,0) {$\cA ( a \to b )$};
      \node (top) at (1,3) {$\cB ( \cF ( a ) \to \cG ( b ) )$};
      \node[box] (lambda) at (2,1) {$\lambda_b$};
      \node[box] (f) at (0,1) {$\cF_{a \to b}$};
      \node[box] (circ) at (1,2) {$- \circ_{\cB} -$};
      \draw (ab) to[in=-90,out=90] (f);
      \draw (f) to[in=-90,out=90] (circ.225);
      \draw (lambda) to[in=-90,out=90] (circ.-45);
      \draw (circ) to[in=-90,out=90] (top);
    \end{tikzpicture}.
  \end{equation}
\end{defn}

%%%%%%%%%%%%%%%%%%%%%%%%%%%%%%%%%%%%%%%%%%%%%%%
\subsection{Mates, the underlying category, and enriched adjunctions}

\begin{defn}
  Suppose $\sA,\sB$ are categories and $\sL: \sA \to \sB, \sR: \sB \to \sA$ are functors such that $\sL\dashv \sR$ via the adjunction
  \begin{equation}
    \label{eq:IntroAdjunction}
    \sB(\sL(a) \to b) \cong \sA(a\to \sR(b)).
  \end{equation}
  We say that $f\in \sB(\sL(a)\to b)$ and $g\in \sA(a\to \sR(b))$ are \emph{mates} if they map to each other under the isomorphism \eqref{eq:IntroAdjunction}.
  We denote this by $f=\mate(g)$ and $g=\mate(f)$.
  We have the following two important identities for mates, which hold by naturality of the adjunction \eqref{eq:IntroAdjunction}:
  \begin{enumerate}[label=(mate\arabic*)]
  \item
    \label{mate:L-composite}
    Given $f_1 : a\to \sR(b_1)$ and $f_2: \sR(b_1)\to \sR(b_2)$, 
    $\mate(f_1\circ f_2) = \sL(f_1)\circ \mate(f_2)$.
  \item
    \label{mate:R-composite}
    Given $g_1: \sL(a_1)\to \sL(a_2)$ and $g_2: \sL(a_2)\to b$,
    $\mate(g_1\circ g_2)=\mate(g_1)\circ \sR(g_2)$.
  \end{enumerate}
  The \emph{unit} of the adjunction is the natural isomorphism $\eta: \id_{\sA}\Rightarrow \sL \circ \sR$ given by $\eta_a:=\mate(\id_{\sL(a)})$.
  The \emph{counit} of the adjunction is the natural isomorphism $\epsilon: \id_{\sB}\Rightarrow \sR\circ\sL$ given by $\epsilon_b := \mate(\id_{\sR(b)})$.
  Observe that \ref{mate:L-composite} and \ref{mate:R-composite} give the identities
  $\sL(\eta_a)\circ \epsilon_{\sL(a)} = \id_{\sL(a)}$
  and
  $\eta_{\sR(b)} \circ \sR(\epsilon_b) = \id_{\sR(b)}$
  respectively.
\end{defn}

\begin{defn}
  Given a $\sV$-category $\cA$, the \emph{underlying category} $\cA^\sV$ has the same objects as $\cA$, and $\cA^\sV(a\to b) := \sV(1\to \cA(a\to b))$.
  The identity of $\cA^\sV(a\to a)$ is the identity element $j_a$, and composition of morphisms is given by $f\circ_{\cA^\sV} g:= (f g)\circ (-\circ_\cA -)$.
  Whenever possible, we use the the sans-serif font $\sA$ as an abbreviation for $\cA^\sV$ for notational simplicity.

  Given a $\sV$-functor $\cF: \cA\to \cB$, the underlying functor $\cF^\sV: \cA^\sV \to \cB^\sV$ is given on $f\in \cA^\sV(a\to b)=\sV(1\to \cA(a\to b))$ by $\cF^\sV(f) := f\circ \cF_{a\to b}$.
  Again, we use the sans-serif font $\sF$ to denote $\cF^\sV$ whenever possible.
\end{defn}

\begin{defn}
\label{defn:v-adjoint}
  Suppose $\sV$ is a monoidal category, $\cA,\cB$ are $\sV$-categories, and $\cL: \cA \to \cB$ is a $\sV$-functor.
  A \emph{right $\cV$-adjoint} $\cL\dashv_\sV \cR$ consists of a $\sV$-functor $\cR: \cB \to \cA$ 
  equipped with a family of isomorphisms $\Psi_{a, d} \in\sV(\cB(\cL(a)\to d)\to \cA(a\to \cR(d)))$ for $a\in \cA$ and $d\in \cB$,
  such that for all $a,b\in \cA$ and all $c,d\in \cB$, the following two diagrams commute:
  \begin{equation}
    \label{eq:V-Adjoint1}
    \begin{tikzcd}%[column sep=tiny]
      \cA(a\to b) \cA(b \to \cR(d))
      \ar[d, "\cL_{a\to b}\Psi_{b, d}^{-1}"]
      \ar[rr, "-\circ_\cA-"]
      &&
      \cA(a\to \cR(d))
      \ar[d, "\Psi_{a, d}^{-1}"]
      \\
      \cB(\cL(a)\to \cL(b))\cB(\cL(b) \to d)
      \ar[rr, "-\circ_\cB-"]
      &&
      \cB(\cL(a)\to d)
    \end{tikzcd}
  \end{equation}
  \begin{equation}
    \label{eq:V-Adjoint2}
    \begin{tikzcd}%[column sep=tiny]
      \cB(\cL(a) \to c)\cB(c\to d)
      \ar[d, "\Psi_{a, c}\cR_{c\to d}"]
      \ar[rr, "-\circ_\cB-"]
      &&
      \cB(\cL(a) \to d)
      \ar[d, "\Psi_{a, d}"]
      \\
      \cA(a \to \cR(c))\cA(\cR(c)\to \cR(d))
      \ar[rr, "-\circ_\cA-"]
      &&
      \cA(a\to \cR(d))
    \end{tikzcd}
  \end{equation}
  By \cite[\S1.11]{MR2177301}, the underlying functor $\sR:\sB \to \sA$ is an ordinary right adjoint to $\sL: \sA \to \sB$.
\end{defn}

\begin{rem}[{\cite[Rem.~2.17]{1809.09782}}]
  Without a braiding on $\sV$, we are not able to form product or opposite $\sV$-categories.
  Thus Definition \ref{defn:v-adjoint} is slightly different from the definition given in \cite[\S1.11]{MR2177301};
  these structures are needed to talk about a $\sV$-natural isomorphism $\cB(\cL(a)\to d) \cong \cA(a\to \cR(d))$.
  A consequence is that standard results for enriched adjunctions, including the enriched Yoneda Lemma, may not necessarily hold under our definitions.
  However, we use only the ordinary Yoneda Lemma and no results from the theory of enriched adjunctions;
  \cite{1809.09782} uses the definition directly and proves the conditions of the definition directly, without appealing to other results.
\end{rem}

\begin{defn}
  Suppose $\cA$ is a $\sV$-category.
  For each $a\in \sA$, we get the representable functor $\sA(a\to -): \sA \to \sV$ given by $b\mapsto \cA(a\to b)$ and for $f: b\to c$,
  $$
  \sA(a\to f):=
  \tikzmath[smallstring]{
    \node (ab) at (0,0) {$\cA (a \to b)$};
    \node (top) at (.5,3) {$\cA (a\to c)$};
    \node[box] (f) at (1,1) {$f$};
    \node[box] (circ) at (.5,2) {$- \circ_{\cA} -$};
    \draw (ab) to[in=-90,out=90] (circ.-135);
    \draw (f) to[in=-90,out=90] node[right]{$\scriptstyle \cA(b\to c)$} (circ.-45);
    \draw (circ) to[in=-90,out=90] (top);
  }.
  $$
  When $\sV$ is rigid monoidal (or more generally closed monoidal) so that we may define the self-enrichment $\cV$, this functor can be promoted to a $\sV$-functor $\cA(a\to -): \cA\to \cV$ by 
  $$
  \cA(a\to -)_{b\to c} :=
  \tikzmath[smallstring]{
    \node (ab) at (-1.5,3) {$\cA (a \to b)^*$};
    \node (ac) at (.5,3) {$\cA (a\to c)$};
    \node (bc) at (0.75,.5) {$\cA(b\to c)$};
    \node[box] (circ) at (.5,2) {$- \circ_{\cA} -$};
    \draw (bc) to (bc |- circ.-90);
    \draw (ab) to  (ab |- circ.-90);
    \draw (ab |- circ.-90) to[in=-90,out=-90]  (0.25,0 |- circ.-90);
    \draw (circ) to[in=-90,out=90] (ac);
  }.
  $$
\end{defn}

\begin{defn}
  A $\sV$-category $\cA$ is called \emph{weakly tensored} if every representable functor $\sA(a\to -): \sA\to \sV$ admits a left adjoint denoted $a\lhd - : \sV \to \sA$.
  We call $\cA$ \emph{tensored} if every $\sV$-representable functor $\cA(a\to -): \cA \to \cV$ admits a left $\sV$-adjoint, which we still denote $a\lhd - : \cV \to \cA$.
\end{defn}

%%%%%%%%%%%%%%%%%%%%%%%%%%%%%%%%%%%%%%%%%%%%%%%
\subsection{Module categories and characterization of tensored enriched categories}

Suppose $\sV$ is a monoidal category.
As before, we suppress tensor products, associators, and unitors whenever possible.

\begin{defn}
  A (right) $\sV$-module is a category $\sM$ equipped with a functor $-\lhd - : \sM\times\sV\to \sM$ 
  together with families of natural isomorphisms 
  $\{\alpha_{m,u,v}: m\lhd u v \to m\lhd u\lhd v\}_{m\in \sM, u,v\in \sV}$
  and
  $\{\rho_m : m\lhd 1\to m\}_{m\in \sM}$
  which satisfy an obvious pentagon associativity axiom and triangle unitality axiom.
  We refer the reader to \cite[\S7.1]{MR3242743} for the diagrams.
  If the $\alpha$ are not necessarily isomorphisms but still satisfy the coherences, we call $\sM$ a \emph{strongly unital oplax} $\sV$-module.

  A $\sV$-module functor $(\sF,\mu): \sM\to \sN$ consists of a functor $\sF: \sM \to \sN$ and \emph{modulator} natural isomorphisms $\{\omega_{m,v} : \sF(m)\lhd v \to \sF(m\lhd v)\}_{m\in\sM, v\in \sV}$ which satisfy the obvious associativity and unitality conditions.
  Again, we refer the reader to \cite[\S7.2]{MR3242743} for the diagrams.

  A $\sV$-module natural transformation $\theta: \sF\Rightarrow \sG$ is a natural transformation such that for each $m\in \sM$ and $v\in \sV$, the following diagram commutes:
  $$
  \begin{tikzcd}
    \sF(m)\lhd v
    \arrow[d, swap, "\theta_m\lhd \id_v"]
    \arrow[r, "\omega_{m,v}^\sF"]
    &
    \sF(m\lhd v)
    \arrow[d, "\theta_{m\lhd v}"]
    \\
    \sG(m)\lhd v
    \arrow[r, "\omega_{m,v}^\sG"]
    &
    \sG(m\lhd v).
  \end{tikzcd}
  $$
\end{defn}

\begin{construction}[$\sV$-category to $\sV$-module]
  \label{const:CatToMod}
  Starting with a weakly tensored $\sV$-category $\cA$,
  the category $\sA$ has a canonical strongly unital $\sV$-module structure
  $-\lhd - : \sV \times \sA \to \sA$ where
  $\id_a \lhd g$ is the functor $a\lhd -$ applied to $g:u \to v$.
  Taking mates under the adjunction $\sA(a\lhd u \to a\lhd v) \cong \sV(u \to \cA(a\to a\lhd v))$, naturality of $\eta$ gives the identity
  \begin{equation}
    \label{eq:MateOf1lhdg}
    \begin{tikzpicture}[smallstring, baseline=25]
      \node (bot) at (0,0) {$u$};
      \node[box] (g) at (2,1) {$\id_a \lhd g$};
      \node[box] (eta) at (0,1) {$\eta_{a, u}$};
      \node[box] (circ) at (1,2) {$- \circ -$};
      \node (top) at (1,3) {$\cA( a \to a \lhd v )$};
      \draw (bot) to (eta);
      \draw (g) to[in=-90,out=90] (circ.-45);
      \draw (eta) to[in=-90,out=90] (circ.225);
      \draw (circ) to (top);
    \end{tikzpicture}
    =
    \begin{tikzpicture}[smallstring, baseline=25]
      \node (bot) at (0,0) {$u$};
      \node[box] (g) at (0,1) {$g$};
      \node[box] (eta) at (0,2) {$\eta_{a, v}$};
      \node (top) at (0,3) {$\cA( a \to a \lhd v )$};
      \draw (bot) to (g);
      \draw (g) to[in=-90,out=90] (eta);
      \draw (eta) to (top);
    \end{tikzpicture}
    \,.
  \end{equation}
  The map $f\lhd \id_v$ is defined as the mate of 
  $$
  u 
  \xrightarrow{f\eta}
  \cA(a\to b)\cA(b\to b\lhd u)
  \xrightarrow{-\circ_\cA -}
  \cA(a\to b\lhd u)
  $$
  under the adjunction $\sA(a\lhd u \to b\lhd u) \cong \sV(u \to \cA(a\to b\lhd u))$.
  This yields the identity
  \begin{equation}
    \label{eq:DefineLhd}
    \begin{tikzpicture}[smallstring, baseline=25]
      \node (bot) at (0,0) {$u$};
      \node[box] (f) at (2,1) {$f \lhd \id_u$};
      \node[box] (eta) at (0,1) {$\eta_{a, u}$};
      \node[box] (circ) at (1,2) {$- \circ -$};
      \node (top) at (1,3) {$\cA( a \to b \lhd u )$};
      \draw (bot) to (eta);
      \draw (f) to[in=-90,out=90] (circ.-45);
      \draw (eta) to[in=-90,out=90] (circ.225);
      \draw (circ) to (top);
    \end{tikzpicture}
    =
    \begin{tikzpicture}[smallstring, baseline=25]
      \node (bot) at (2,0) {$u$};
      \node[box] (f) at (0,1) {$f$};
      \node[box] (eta) at (2,1) {$\eta_{b, u}$};
      \node[box] (circ) at (1,2) {$- \circ -$};
      \node (top) at (1,3) {$\cA( a \to b \lhd u )$};
      \draw (bot) to (eta);
      \draw (f) to[in=-90,out=90] (circ.225);
      \draw (eta) to[in=-90,out=90] (circ.-45);
      \draw (circ) to (top);
    \end{tikzpicture}
    \,.
  \end{equation}
  The exchange relation for $- \lhd -$ is now visibly apparent combining \eqref{eq:MateOf1lhdg} and \eqref{eq:DefineLhd}.
  The functor $-\lhd -$ comes equipped with a strongly unital \emph{oplaxitor} $\alpha_{a, u,v} \in \sA(a \lhd u v \to a \lhd u \lhd v)$ given by the mate of the map
  $$
  uv 
  \xrightarrow{\eta_{a,u}\eta_{a \lhd u, v}} 
  \cA(a\to a \lhd u) \cA(a \lhd u \to a \lhd u \lhd v)
  \xrightarrow{-\circ_\cA-}
  \cA(a\to a \lhd u \lhd v)
  $$
  under the adjunction
  $
  \sA(a \lhd u v \to a \lhd u \lhd v)
  \cong 
  \sV(uv \to \cA(a\to a \lhd u \lhd v))$.
  The isomorphism $\rho_a\in \sA(a\lhd 1_\sV \to a)$ is the mate of $j_a$ under the adjunction
  $\cA(a\lhd 1_\sV \to a) \cong \sV(1_\sV \to \cA(a\to a))$
  with inverse
  $\eta_{a,1_\sV}\in \sV(1_\sV \to \cA(a\to a\lhd 1_\sV))=\sA(a\to a\lhd 1_\sV)$.
  By \cite{MR649797} or \cite[\S4]{1809.09782}, the map $\alpha_{u,v,a}$ is an isomorphism if and only if the left adjoint $a \lhd -$ of $\sA(a\to -)$ can be promoted to a left $\cV$-adjoint of $\cA(a\to -)$.
\end{construction}

\begin{rem}[{\cite{MR3961709}}]
  \label{rem:mate of circ}
  The mate of the composition map $- \circ_\cA - \in \sV(\cA(a \to b) \cA(b \to c) \to \cA(a \to c))$ under Adjunction \ref{eq:CatToModAdjunction} is
  \begin{equation}
    \begin{tikzpicture}[smallstring, baseline=25]
      \node (a) at (0,0) {\small $a$};
      \node (bot) at (2,0) {\small $\cA(a \to b) \cA(b \to c)$};
      \node (c) at (0,4) {\small $c$};
      \draw[double] (bot) to (bot |- 0,0.5);
      \modulebox{0,0.5}{3.5}{0.75}{$\alpha_{a, \cA(a \to b), \cA(b \to c)}$}
      \modulebox{0,1.5}{1.5}{0.75}{$\epsilon_{a \to b}$}
      \modulebox{0,2.5}{3.5}{0.75}{$\epsilon_{b \to c}$}
      \draw[blue] (a) to (c);
      \draw (1,1.25) to (1,1.5);
      \draw (2,1.25) to (2,2.5);
    \end{tikzpicture}
    \,.
  \end{equation}
  The proof is omitted as it is simpler in our setting than the $\sV$-monoidal setting of \cite{MR3961709}.
\end{rem}

\begin{construction}[$\sV$-module to $\sV$-category]
  \label{const:ModToCat}
  Conversely, given a strongly unital oplax $\sV$-module $\sA$ where each $c \lhd -: \sV \to \sA$ admits a right adjoint,
  we define a $\sV$-category $\cA$ by defining the hom objects via the adjunction
  \begin{align} \label{adj:main}
    \sA(a \lhd v \to b) 
    \cong
    \sV(v \to \cA(a\to b)),
  \end{align}
  the identity element $j_a \in \sV(1\to \cA(a\to a))$ as the mate of $\rho_a\in \sA(a\lhd 1\to a)$,
  and the composition morphism $- \circ_\cA - \in \cV(\cA(a \to b) \cA(b \to c) \to \cA(a \to c))$ as the mate of
  $\alpha_{a, \cA(a \to b), \cA(b \to c)} \circ (\epsilon_{a \to b} \id_{\cA(b \to c)}) \circ \epsilon_{b \to c}$.
  Here, $\epsilon_{a\to b} : a\lhd \cA(a\to b) \to b$ is the mate of $\id_{\cA(a\to b)}$ under Adjunction \eqref{adj:main}.

  When $\sA$ is a (strong) $\sV$-module (where the associators $\alpha$ are isomorphisms), and $\sV$ is rigid monoidal (or more generally closed monoidal), 
  we can promote each functor $a\lhd - : \sV \to \sA$ to a $\sV$-functor $\cV\to \cA$ by defining $(a\lhd -)_{u\to v} \in \sV(\cV(u\to v) \to \cA(a\lhd u \to a\lhd v))$ as the mate of $\alpha^{-1}_{a,u,\cV(u\to v)} \circ (1\lhd (\ev_u \id_v)): a\lhd u\lhd a \cV(u\to v) \to a\lhd v$.
  In this case, each $a\lhd - : \cV\to \cA$ is a left $\cV$-adjoint of $\cA(a\to -)$.
\end{construction}

\begin{defn}
  A strongly unital oplax $\sV$-module $\sM$ is called:
  \begin{itemize}
  \item
  \emph{weakly tensored} if every functor $m \lhd - : \sV \to \sM$ admits a right adjoint, and
  \item
  \emph{tensored} if $\sM$ is a weakly tensored strong $\sV$-module, i.e., all oplaxitors $\alpha_{a,u,v}: a\lhd uv \to a\lhd u \lhd v$ are isomorphisms. 
  \end{itemize}
\end{defn}

%%%%%%%%%%%%%%%%%%%%%%%%%%%%%%%%%%%%%%%%%%%%%%%

Constructions \ref{const:CatToMod} and \ref{const:ModToCat} prove the following theorem, which was originally due to \cite{MR649797} (see also \cite{MR1897810}, \cite{MR1466618}, \cite[Lem.~4.7]{1110.3567}).
The details on the level of objects can be found in \cite[\S3-4]{1809.09782},\footnote{There is an error in the equivalence of objects in \cite[\S3]{1809.09782}, namely in the definition of the modulator.
  This error has propagated to \cite{2104.03121} which cites \cite{1809.09782} for the construction of the modulator.
  We fix this error in Appendix \ref{app:modulators} below; see the proof of Proposition \ref{prop:equivalence-of-dagger-U-mods} for more details.
} 
and some details for higher levels appear in \cite{2104.03121}, which are similar to the proof in the $\sV$-monoidal setting from \cite{2104.07747}.

\begin{thm}
  \label{thm:StrongVMod}
  Let $\sV$ be a closed monoidal category.
  There is an equivalence of 2-categories
  \[
  \left\{\,
  \parbox{3.7cm}{\rm Tensored $\sV$-categories}
  \,\right\}
  \,\,\cong\,\,
  \left\{\,
  \parbox{3.4cm}{\rm Tensored $\sV$-modules}
  \,\right\}.
  \]
\end{thm}

%%%%%%%%%%%%%%%%%%%%%%%%%%%%%%%%%%%%%%%%%%%%%%%
\subsection{Involutions on monoidal categories and unitary monoidal categories}

In this section, we discuss the basic notions of dagger (monoidal) category, involutive monoidal category, and unitary monoidal categories used in this work.
General references for notions in this section include \cite{MR2767048,MR2861112,MR3663592,MR3687214,MR4133163}.

\begin{defn}
  A \emph{dagger category} is a category $\sC$ equipped with anti-linear maps
  $\dag: \sC(a\to b)\to \sC(b\to a)$ for all $a,b\in \sC$ such that $f^{\dag\dag}=f$ for all morphisms $f$ and
  $(f\circ g)^\dag = g^\dag \circ f^\dag$ for all composable morphisms $f,g$.
  When $\sC$ is linear, we require $\dagger$ is anti-linear.

  A morphism $u$ in a dagger category is called \emph{unitary} if $u$ is invertible with $u^{-1}=u^\dag$.

  A \emph{monoidal dagger category} is a dagger category $(\sV,\dag)$ equipped with a monoidal structure
  $(\otimes, 1, \alpha,\rho,\lambda)$
  such that
  $\otimes: \sV\times \sV\to \sV$ is a $\dag$-functor,
  and $\alpha,\lambda,\rho$ are all unitary.
  When $\sV$ is linear, we require $\otimes$ is linear.
\end{defn}

\begin{defn}
  A \emph{dagger functor} between dagger categories is a functor between the underlying categories $\sF: \sA\to \sB$ which preserves the dagger structure, i.e., $\sF(f^\dag)=\sF(f)^\dag$.

  A \emph{dagger monoidal functor} between monoidal dagger categories $(\sF,\mu): \sU\to \sV$ is
  a dagger functor equipped with a unitary monoidal coherence isomorphism
  $\mu: \sF(ab)\to \sF(a)\sF(b)$
  which satisfies unitality and associativity criteria.
\end{defn}

\begin{defn}[{\cite{MR2861112,MR3687214}}]
  \label{defn:involutive linear category}
  An \emph{involutive category} is a category $\sC$ equipped with a \emph{conjugation} functor
  $\overline{\,\cdot\,}:\sC\to \sC$ together with a natural isomorphism
  $\varphi: \id_\sC \to \overline{\overline{\,\cdot\,}}$
  satisfying 
  \begin{itemize}
  \item
    $\overline{\varphi_c} = \varphi_{\overline{c}}: \overline{c} \to \overline{\overline{\overline{c}}}$ for all $c\in \sC$.
  \end{itemize}
  When $\sC$ is linear, we require $\overline{\,\cdot\,}$ is anti-linear.

  An \emph{involutive monoidal category} is an involutive category $(\sV,\overline{\,\cdot\,},\varphi)$ equipped with a monoidal structure $(\otimes, 1, \alpha,\rho,\lambda)$ together with 
  an anti-monoidal coherence isomorphism 
  $$
  \nu_{u,v} : \overline{u} \otimes \overline{v} \to \overline{v\otimes u} 
  $$
  for $\overline{\,\cdot\,}$ and a \emph{real structure} isomorphism $r: 1\to \overline{1}$
  satisfying the  coherence axioms
  \begin{itemize}
  \item 
    ($r$ a real structure)
    $r \circ \overline{r} = \varphi_{1_\sV}$.
  \item 
    ($\varphi: \id_\sV \to \overline{\overline{\,\cdot\,}}$ a monoidal natural isomorphism) 
    $\varphi_{uv} = ( \varphi_u \varphi_v ) \circ \nu_{\overline{u},\overline{v}} \circ \overline{\nu_{u,v}}$
  \item (unitality)
    $(\id_{\overline{u}} r) \circ \nu_{u,1_\sV} = \id_{\overline{u}} = (r \id_{\overline{u}}) \circ \nu_{1_\sV,u}$
  \item (associativity)
    $(\nu_{u,v} \id_{\overline{w}}) \circ \nu_{vu, w} = (\id_{\overline{u}} \nu_{v,w}) \circ \nu_{u,wv}$.
  \end{itemize}
  Above, we have suppressed all associators and unitors in $\sV$ for convenience.
\end{defn}

\begin{defn}
  An \emph{involutive functor} between involutive categories $(\sF,\chi): \sA\to \sB$ is a functor
  $\sF: \sA\to \sB$
  equipped with a natural isomorphism
  $\chi_a: \sF(\overline{a}) \to \overline{\sF(a)}$ satisfying
  \begin{itemize}
  \item $\varphi_{\sF(a)} = \sF(\varphi_a) \circ \chi_{\overline{a}} \circ \overline{\chi_a}$ for all $a\in \sA$.
  \end{itemize}

  An \emph{involutive monoidal functor} 
  $(\sF, \mu, \chi) : \sU \to \sV$ between involutive monoidal categories is a monoidal functor
  $(\sF,\mu): \sU\to \sV$ such that $(\sF, \chi)$ is an involutive functor between the underlying categories,
  and the following additional coherence axioms are satisfied:
  \begin{itemize}
  \item $\sF(r_\sU) \circ \chi_{1_\sU} = r_\sV$
  \item $\mu_{\overline{v}, \overline{u}} \circ \sF(\nu_{v,u}) \circ \chi_{uv} = (\chi_v \chi_u) \circ \nu_{\sF(v), \sF(u)} \circ \overline{\mu_{u,v}}$
  \end{itemize}
  Again, we suppress all associators and unitors in $\sU,\sV$ for convenience.
\end{defn}

\begin{defn}[\cite{MR3663592}]
  A \emph{bi-involutive} category is a quadruple $(\sA,\dag,\overline{\,\cdot\,},\varphi)$
  where $(\sA,\dag)$ is a dagger category and $(\sA, \overline{\,\cdot\,},\varphi)$ is an involutive category
  such that $\overline{\,\cdot\,}$ is a dagger functor, and the natural isomorphism
  $\varphi: \id_\sA\to \overline{\overline{\,\cdot\,}}$
  is unitary.

  A \emph{bi-involutive monoidal category} is a bi-involutive category
  $(\sA,\dag,\overline{\,\cdot\,},\varphi)$
  equipped with a monoidal structure $(\otimes, 1, \alpha,\rho,\lambda)$ and coherators $(\nu,r)$
  so that $\sA$ is simultaneously a monoidal dagger category and an involutive monoidal category,
  with the extra condition that $\nu,r$ are unitary.
\end{defn}

\begin{defn}
  A \emph{bi-involutive functor} between bi-involutive categories is an involutive dagger functor whose coherence natural isomorphism $\chi$ is unitary.

  A \emph{bi-involutive monoidal functor} is a bi-involutive functor equipped with a unitary monoidal coherence isomorphism.
\end{defn}

\begin{defn}
\label{defn:Rigid}
  Suppose $\sV$ is a monoidal category.
  A \emph{dual} for $v\in \sV$ is a choice of dual object $v^*$ together with evaluation and coevaluation maps 
  $$
  \begin{tikzpicture}[baseline=-.1cm]
    \draw (0,0) node[below]{$\scriptstyle v$} arc (180:0:.3cm) node[below]{$\scriptstyle v^*$};
  \end{tikzpicture}
  =
  \ev_v: vv^*\to 1
  \qquad\qquad
  \begin{tikzpicture}[baseline=-.1cm]
    \draw (0,0) node[above]{$\scriptstyle v^*$} arc (-180:0:.3cm) node[above]{$\scriptstyle v$};
  \end{tikzpicture}
  =
  \coev_v:1\to v^*v
  $$ 
  satisfying the zig-zag/snake equations:
  $$
  \tikzmath{
    \draw (-.6,.6) -- node[left]{$\scriptstyle v^*$} (-.6,0) arc (-180:0:.3cm) node[left]{$\scriptstyle v$} arc (180:0:.3cm) -- node[right]{$\scriptstyle v^*$} (.6,-.6);
  }
  =\,\,
  \tikzmath{
    \draw (0,-.6) -- node[right]{$\scriptstyle v^*$} (0,.6);
  }
  \qquad\qquad
  \tikzmath[xscale=-1]{
    \draw (-.6,.6) -- node[right]{$\scriptstyle v$} (-.6,0) arc (-180:0:.3cm) node[left]{$\scriptstyle v^*$} arc (180:0:.3cm) -- node[left]{$\scriptstyle v$} (.6,-.6);
  }
  =\,\,
  \tikzmath{
    \draw (0,-.6) -- node[right]{$\scriptstyle v$} (0,.6);
  }.
  $$
  Any two choices of dual are canonically naturally isomorphic, so having a dual is a property, not additional structure.
  We call $v\in \sV$ \emph{dualizable} if $v$ has a dual and $v$ is isomorphic to the dual of another $v_*\in V$.
  We call $\sV$ \emph{rigid} if every $v\in \sV$ is dualizable.

  Given a dual $(v^*,\ev_v,\coev_v)$ for each $v\in \sV$, we get a canonical \emph{dual functor} $*: \sV\to \sV^{\rm mop}$, the monoidal and arrow opposite of $\sV$, given by
  $$
  f^*
  :=
  \tikzmath{
    \draw (0,.3) arc (180:0:.3cm) -- (.6,-.7) node[below]{$\scriptstyle v^*$};
    \draw (0,-.3) arc (0:-180:.3cm) -- (-.6,.7) node[above]{$\scriptstyle u^*$};
    \roundNbox{fill=white}{(0,0)}{.3}{0}{0}{$f$}
  }
  \qquad\qquad\qquad
  \forall\, f: u\to v.
  $$
  Observe $*$ comes equipped with a canonical monoidal coherator
  $$
  \nu_{u,v} := 
  \tikzmath{
    \draw (1.1,-.6) -- node[right]{$\scriptstyle v^*$} (1.1,0) arc (0:180:.575cm) arc (0:-180:.25cm) -- (-.55,.9);
    \draw (.6,-.6) -- node[right]{$\scriptstyle u^*$} (.6,0) arc (0:180:.3cm) arc (0:-180:.3cm) -- node[left]{$\scriptstyle (vu)^*$} (-.6,.9);
  }
  =
  ( \coev_{vu} \id_{u^*v^*}) \circ (\id_{(vu)^* v} \ev_{u} \id_{v^*}) \circ (\id_{(vu)^*} \ev_v).
  $$
  Again, any two choices of dual functor are canonically monoidally naturally isomorphic,
  so the existence of a dual functor is a property and not really a structure.
\end{defn}

\begin{rem}
  If $(\sV,\dag,\otimes)$ is a monoidal dagger category and $v\in \sV$ has dual $(v^*,\ev_v,\coev_v)$,
  then $v^*$ has dual $(v, \coev_v^\dag, \ev_v^\dag)$.
\end{rem}

\begin{defn}[\cite{MR4133163}]
  A \emph{unitary dual functor} on a monoidal dagger category $(\sV,\dag,\otimes)$ is a choice of dual
  $(v^*,\ev_v,\coev_v)$ for each $v\in \sV$
  such that the canonical dual functor $*$ is a dagger functor, and the canonical monoidal coherator is unitary. 
  Even though any two unitary dual functors on $\sV$ are canonically monoidally naturally isomorphic, this isomorphism need not be unitary.
  Thus a unitary dual functor on $\sV$ is a structure, and not just a property.

  By \cite[\S3.5]{MR4133163}, a unitary dual functor gives rise to a canonical unitary involutive structure $(\overline{\,\cdot\,},\varphi,\nu,r)$ such that $(\sV, \dag,\overline{\,\cdot\,})$ is a bi-involutive monoidal category.
  Indeed, we set $\overline{v}:=v^*$ and $\overline{f} := (f^\dagger)^*$ for all $f \in \sV(u \to v)$. 
  The coherence isomorphisms $r, \nu$, and $\varphi$ are given by the unitary isomorphisms
  \begin{itemize}
  \item $r := \coev_{1_\sU}$
  \item $\nu_{u,v} := ( \coev_{vu} \id_{\overline{u} \, \overline{v}}) \circ (\id_{\overline{vu} v} \ev_{u} \id_{\overline{v}}) \circ (\id_{\overline{vu}} \ev_v)$
  \item $\varphi_u := (\id_u \ev_{\overline{u}}^\dagger) \circ (\ev_u \id_{\overline{\overline{u}}}) = (\coev_{\overline{u}} \id_u) \circ (\id_{\overline{\overline{u}}} \coev_u^\dagger)$.
  \end{itemize}
  A monoidal dagger category $(\sV,\dag,\otimes)$ equipped with a unitary dual functor $*$ and its canonical involutive structure $\overline{\,\cdot\,}$ is called a \emph{unitary monoidal category}.
  In the sequel, we reserve the notation $\sU$ for a unitary monoidal category.
\end{defn}

\begin{ex}
  The most common examples of unitary monoidal categories are \emph{unitary tensor categories}, i.e., semisimple rigid $\rm C^*$ tensor categories with simple unit object, which arise naturally in subfactor theory \cite{MR3166042,1509.00038}, discrete and compact quantum groups \cite{MR1616348,MR3204665}, conformal field theory \cite{MR165864,MR1645078,MR3308880}, and topologically ordered phases of matter \cite{PhysRevB.71.045110,MR2804555,2106.15741}.
\end{ex}

\begin{defn}
  Suppose $\sV$ is a monoidal dagger category.
  A \emph{dagger right $\sV$-module category} is a dagger category $\sM$ together with a dagger functor $\lhd:\sM \times \sV \to \sM$ equipped with associator unitary natural isomorphisms $\alpha^\sM_{u,v,m}: m \lhd u \lhd v \to m \lhd u \otimes v$ and unitor unitary natural isomorphisms $\rho^\sM_{m}: m \lhd 1_\sU \to m$ which satisfy the obvious pentagon and triangle axioms. 

  A \emph{dagger $\sV$-module functor} from $\sM$ to $\sN$ consists of a dagger functor $\sF : \sM \to \sN$ and natural unitary isomorphisms $\theta_{m, u} : \sF(m) \lhd u \to F(m\lhd u)$ such that, suppressing associators in $\sM,\sN$, $(\theta_{m,u} \id_v) \circ \theta_{mu, v} = \theta_{m, uv}$.
\end{defn}

\begin{defn}
  We define a 2-category $\dagvmod$ with:
  \begin{itemize}
  \item 0-cells tensored dagger $\sV$-module categories,
  \item 1-cells dagger $\sV$-module functors with the usual composition, and
  \item 2-cells $\sV$-module natural transformations with the usual horizontal and vertical compositions.
  \end{itemize}
\end{defn}

%%%%%%%%%%%%%%%%%%%%%%%%%%%%%%%%%%%%%%%%%%%%%%%
%%%%%%%%%%%%%%%%%%%%%%%%%%%%%%%%%%%%%%%%%%%%%%%
%%%%%%%%%%%%%%%%%%%%%%%%%%%%%%%%%%%%%%%%%%%%%%%
\section{Enriched dagger categories}
\label{sec:EnrichedDaggerCats}

%%%%%%%%%%%%%%%%%%%%%%%%%%%%%%%%%%%%%%%%%%%%%%%
%%%%%%%%%%%%%%%%%%%%%%%%%%%%%%%%%%%%%%%%%%%%%%%
%%%%%%%%%%%%%%%%%%%%%%%%%%%%%%%%%%%%%%%%%%%%%%%
\subsection{Dagger \texorpdfstring{$\sV$}{V}-categories}
Throughout this subsection, let $\sV$ be an involutive monoidal category.
Recall that we cannot define the op of a $\sV$-enriched category without a braiding on $\sV$.
However, for $\sV$ an involutive tensor category with no braiding, there is a canonical involution on $\sV$-categories. 

\begin{defn}[Conjugate of a $\sV$-category]
  \label{def:ConjugateVCat}
  Given a $\sV$-category $\cA$, we can form a new $\sV$-category $\overline{\cA}$ with the same objects as $\cA$ by defining
  \begin{itemize}
  \item hom objects $\overline{\cA} (a \to b) := \overline{\cA(b \to a)}$,
  \item identities
    $j_a^{\overline{\cA}} := r \circ \overline{j_a^\cA} \in \sV(1_\sV \to \overline{\cA} (a \to a))$, and
  \item composition
    $- \circ_{\overline{\cA}} - := \nu_{\cA(b \to a), \cA(c \to b)} \circ (\overline{- \circ_\cA -}):
    \overline{\cA} (a \to b) \overline{\cA} (b \to c) \to \overline{\cA} (a \to c)$.
  \end{itemize}
  It is straightforward to verify that $\overline{\cA}$ is a $\sV$-category.
\end{defn}

\begin{rem}
  The map $\cA \to \overline{\cA}$ of $\sV$-categories can be extended to a 2-functor, contravariant on 2-cells, on the 2-category of $\sV$-categories.
  On $\sV$-functors $\cF : \cA \to \cB$, take
  $\overline{\cF}_{a \to b} := \overline{\cF_{b \to a}}$,
  and given a $\sV$-natural transformation
  $\theta: \cF \Rightarrow \cG: \cA \to \cB$,
  take
  $\overline{\theta}_a := r \circ \overline{\theta_a} \in \sV(1_\sV \to \overline{\cB}(\cG(a) \to \cF(a)))$. 
\end{rem}

\begin{defn} \label{defn:involutive U-cat}
  An \emph{weak dagger $\sV$-category} is a $\sV$-category $\cA$ equipped with
  a family of isomorphisms
  $\kappa_{a\to b} \in \sV(\overline{\cA(b\to a)} \to \cA(a \to b))$
  satisfying:
  \begin{enumerate}[label=($\kappa$\arabic*)]
  \item
    \label{kappa:1}
    $\varphi_{\cA(a\to b)}\circ \overline{\kappa_{a\to b}} \circ \kappa_{b\to a} = \id_{\cA(a\to b)}$.
  \item \label{kappa:2}
    $(\kappa_{b \to a} \kappa_{c \to b}) \circ (-\circ_\cA-)
    =
    (-\circ_{\overline{\cA}}-) \circ \kappa_{c \to a}
    \in
    \sV(\overline{\cA(b \to a)} \, \overline{\cA(c \to b)} \to \cA(a \to c))$.
  \end{enumerate}  
\end{defn}

Before continuing further, we make a couple remarks on the definition of weak dagger $\sV$-category.

\begin{rem}
  \label{rem:principle of equivalence}
  Observe that 
  $j_a = r \circ \overline{j_{a}} \circ \kappa_{a \to a}$:
  \[
  \begin{tikzpicture}[smallstring, baseline=5]
    \node (top) at (0,1) {$\cA(a \to a)$};
    \node[box] (j) at (0,0) {$j_a$};
    \draw (j) to (top);
  \end{tikzpicture}
  =
  \begin{tikzpicture}[smallstring, baseline=20]
    \node[box] (j) at (0,0) {$j_a$};
    \node[box] (kappainv) at (0,1) {$\kappa_{a \to a}^{-1}$};
    \node[box] (kappa) at (0,2) {$\kappa_{a \to a}$};
    \node (top) at (0,3) {$\cA(a \to a)$};
    \draw (j) to (kappainv);
    \draw (kappainv) to[in=-90,out=90] (kappa);
    \draw (kappa) to[in=-90,out=90] (top);
  \end{tikzpicture}
  \underset{\text{(Def.~\ref{def:ConjugateVCat})}}{=}
  \begin{tikzpicture}[smallstring, baseline=30]
    \node[box] (j) at (0,0) {$j_a$};
    \node[box] (kappainv) at (0,1) {$\kappa_{a \to a}^{-1}$};
    \node[box] (r) at (2,0) {$r$};
    \node[box] (jbar) at (2,1) {$\overline{j_a}$};
    \node[box] (circ) at (1,2) {$- \circ_{\overline{\cA}} -$};
    \node[box] (kappa) at (1,3) {$\kappa_{a \to a}$};
    \node (top) at (1,4) {$\cA(a \to a)$};
    \draw (j) to (kappainv);
    \draw (kappainv) to[in=-90,out=90] (circ.225);
    \draw (r) to (jbar);
    \draw (jbar) to[in=-90,out=90] (circ.-45);
    \draw (circ) to (kappa);
    \draw (kappa) to (top);
  \end{tikzpicture}
  \underset{\text{\ref{kappa:2}}}{=}
  \begin{tikzpicture}[smallstring, baseline=30]
    \node[box] (j) at (0,2) {$j_a$};
    \node[box] (r) at (2,0) {$r$};
    \node[box] (jbar) at (2,1) {$\overline{j_a}$};
    \node[box] (kappa) at (2,2) {$\kappa_{a \to a}$};
    \node[box] (circ) at (1,3) {$- \circ -$};
    \node (top) at (1,4) {$\cA(a \to a)$};
    \draw (j) to[in=-90,out=90] (circ.225);
    \draw (r) to (jbar);
    \draw (jbar) to (kappa);
    \draw (kappa) to[in=-90,out=90] (circ.-45);
    \draw (circ) to (top);
  \end{tikzpicture}
  =
  \begin{tikzpicture}[smallstring, baseline=20]
    \node[box] (r) at (0,0) {$r$};
    \node[box] (jbar) at (0,1) {$\overline{j_a}$};
    \node[box] (kappa) at (0,2) {$\kappa_{a \to a}$};
    \node (top) at (0,3) {$\cA(a \to a)$};
    \draw (r) to (jbar);
    \draw (jbar) to (kappa);
    \draw (kappa) to (top);
  \end{tikzpicture}\, .
  \]
  Thus, similar to how a dagger structure on a category $\sC$ can be viewed as a functor $\dag: \sC^{\rm op}\to \sC$
  which is the \emph{identity on objects} satisfying $\dag\circ \dag =\id_{\sC}$, 
  a weak dagger $\sV$-category $(\cA,\kappa)$ can be viewed as a $\sV$-category equipped with a $\sV$-functor $\kappa: \overline{\cA}\to \cA$ which is the \emph{identity on objects} satisfying \ref{kappa:1}.
  
  However, just as a dagger structure cannot be transported to an equivalent category, a weak enriched dagger structure $\kappa$ on a $\sV$-category cannot be transported to an equivalent $\sV$-category.
  In this sense, just as dagger categories are not merely categories with extra structure,
  weak dagger $\sV$-categories are not merely $\sV$-categories with extra structure, 
  as this extra structure violates the principle of equivalence \cite{nlab:principle_of_equivalence}.
  See \cite[Rem.~3.3]{MR4133163} for further discussion.
\end{rem}

\begin{rem}
  \label{rem:KappaInvertible}
  Observe that \ref{kappa:1} implies $\kappa_{a\to b}$ is invertible for all $a,b\in \cA$.
  Indeed, \ref{kappa:1} swapping $a$ and $b$ shows $\kappa_{a\to b}$ admits a left inverse
  (remember we have swapped the order of composition).
  Now since $\varphi$ is invertible, we see that \ref{kappa:1} implies
  $$
  \kappa_{a\to b}\circ \varphi_{\cA(a\to b)}\circ \overline{\kappa_{b\to a}}
  =
  \varphi_{\cA(a\to b)}\circ \overline{\varphi^{-1}_{\cA(a\to b)}}.
  $$
  Since the right hand side is invertible, we see $\kappa_{a\to b}$ admits a right inverse as well.
\end{rem}

\begin{ex}
  Consider $\sV=\Vect$ with complex conjugation.
  Observe that $\cV=\Vect$ is $\sV$-enriched, but does not admit the structure of a weak dagger $\sV$-category.

  Now consider $\cA=\Hilb$, which is also $\sV$-enriched.
  We get a weak dagger $\sV$-category structure on $\Hilb$ by defining
  $\kappa_{H \to K} \in \Hom(\overline{B(K\to H)} \to B(H\to K))$
  by $\overline{T}\mapsto T^*$ for $T\in B(K\to H)$.
\end{ex}

\begin{ex}
  \label{ex:H*Cat}
  Recall from \cite{MR1448713} or \cite[\S8.2.1]{MR3971584} that an $\rm H^*$-\emph{category} $\sH$
  is a linear dagger category such that
  $\sH(a\to b)$ is a finite dimensional Hilbert space for all $a,b\in \sH$, and 
  $$
  \langle g , f^\dag\circ h\rangle_{\sH(b\to c)}
  =
  \langle f\circ g , h\rangle_{\sH(a\to c)}
  =
  \langle f , h\circ g^\dag\rangle_{\sH(a\to b)}
  \qquad\qquad
  \forall \, f: a\to b\,, g: b\to c,\, h:a\to c.
  $$
  An $\rm H^*$-category $\sH$ is canonically a weak dagger $\Hilb$-category by defining 
  $$
  \kappa_{a\to b} :=\dag: \overline{\sH(b\to a)} \to \sH(a\to b)
  \qquad\qquad
  \overline{f}\mapsto f^\dag.
  $$
\end{ex}

\begin{rem}
  It is straightforward to verify that when $\cA$ is a weak dagger $\sV$-category, $\overline{\cA}$ is also weak dagger
  via %the $\sV$-functor $\kappa^{\overline{\cA}}$ defined as the identity on objects,
  %with 
  $\kappa_{a \to b}^{\overline{\cA}} := \overline{\kappa_{b \to a}^\cA}$.
  %It is straightforward to verify that $\overline{\cA}$ with this $\kappa^{\overline{\cA}}$ is a weak dagger $\sV$-category
%\end{rem}
%\begin{rem}
  Moreover, the coherence isomorphism $\varphi$ lifts to a canonical $\sV$-equivalence
  $\Psi : \cA \to \overline{\overline{\cA}}$
  defined as the identity on objects and
  $\Psi_{a \to b} := \varphi_{\cA(a \to b)} \in \sV(\cA(a \to b) \to \overline{\overline{\cA}} (a \to b))$.
\end{rem}

\begin{defn}
  Suppose $\cA,\cB$ are weak dagger $\sV$-categories and $\cF: \cA \to \cB$ is a $\sV$-functor.
  We say $\cF$ is \emph{dagger} if the following diagram commutes:
  $$
  \begin{tikzcd}
    \cA(a\to b) 
    \arrow[r, "\cF_{a\to b}"]
    \arrow[d, "\kappa^\cA_{a\to b}"]
    &
    \cB(\cF(a) \to \cF(b))
    \arrow[d, "\kappa^\cB_{\cF(a)\to \cF(b)}"]
    \\
    \overline{\cA(b\to a)} 
    \arrow[r, "\overline{\cF_{b\to a}}"]
    &
    \overline{\cB(\cF(b) \to \cF(a))}
  \end{tikzcd}\, .
  $$
\end{defn}

In order to achieve a correspondence with dagger $\sV$-module categories, the concept of weak dagger $\sV$-category is not enough.
We now introduce \emph{dagger $\sV$-categories}.

\begin{defn}
  A tensored weak dagger $\sV$-category $\cA$ is called a \emph{dagger $\sV$-category} if in addition:
  \begin{enumerate}[label=($\kappa$\arabic*)]
    \setcounter{enumi}{2}
  \item \label{kappa:3}
    the maps $\rho_a \in \sA(a\lhd 1 \to a)$ are unitary for all $a\in \sA$, and
  \item \label{kappa:4}
    the left $\sV$-adjoints $(a\lhd -) \dashv_{\sV} \cA(a\to -)$ are dagger for all $a\in \cA$.
  \end{enumerate}
\end{defn}

\begin{rem}
  While the condition \ref{kappa:3} feels somewhat artificial, we do not know of a way to prove Propositions \ref{prop:daglhd1} and \ref{prop:daglhd3} below without it.
  In the proof of Proposition \ref{prop:daglhd1}, we use \ref{kappa:3} twice to be able to apply unitality of the modulator,
  and in the proof of Proposition \ref{prop:daglhd3}, we cancel $\rho^{-1}$ with $(\rho^\dagger)^{-1}$ in a critical way.
\end{rem}

\begin{defn}
  Let $\sV$ be a rigid involutive monoidal category.
  We define a 2-category $\dagvcat$ with:
  \begin{itemize}
  \item 0-cells dagger $\sV$-categories,
  \item 1-cells dagger $\sV$-functors with their usual composition, and
  \item 2-cells $\sV$-natural transformations with their usual horizontal and vertical compositions.
  \end{itemize}
\end{defn}

\begin{construction}
  \label{const:DaggerFromInvolutive}
  Given a weak dagger $\sV$-category $(\cA, \kappa)$,
  we get a dagger structure on the underlying category $\sA$ as follows.
  For $f\in \sA(a\to b) = \sV(1_\sV \to \cA(a\to b))$, we define $f^\dag \in \sA(b\to a)$ as the composite
  $$
  1_\sV
  \xrightarrow{r}
  \overline{1_\sV}
  \xrightarrow{\overline{f}}
  \overline{\cA(a\to b)}
  \xrightarrow{\kappa_{b\to a}}
  \cA(b\to a).
  $$
\end{construction}

\begin{construction}
  \label{const:DaggerFromInvolutiveFunctor}
  Let $(\cA, \kappa^\cA)$ and $(\cB, \kappa^\cB)$ be weak dagger $\sV$-categories,
  and let $\cF: \cA \to \cB$ be a $\sV$-functor.
  If $\cF$ is dagger, then using the dagger structure on the underlying categories introduced
  in Construction \ref{const:DaggerFromInvolutive},
  the underlying functor $\sF : \sA \to \sB$ is a dagger functor.
  Indeed, for $f \in \cA(a \to b)$, we compute
\  $$
  (\sF(f))^\dag
  =
  \tikzmath[smallstring]{
    \node (top) at (0,3) {$\cB(\cF(b) \to \cF(a))$};
    \node[box] (f) at (0,0) {$\overline{f}$};
    \node[box] (F) at (0,1) {$\overline{\cF_{a\to b}}$};
    \node[box] (kappa) at (0,2) {$\kappa^\cB_{\cF(b) \to \cF(a)}$};
    \draw (f) to (F);
    \draw (F) to (kappa);
    \draw (kappa) to (top);
  }
  =
  \tikzmath[smallstring]{
    \node (top) at (0,3) {$\cB(\cF(b) \to \cF(a))$};
    \node[box] (f) at (0,0) {$\overline{f}$};
    \node[box] (F) at (0,2) {$\cF_{b\to a}$};
    \node[box] (kappa) at (0,1) {$\kappa^\cA_{b \to a}$};
    \draw (f) to (kappa);
    \draw (kappa) to (F);
    \draw (F) to (top);
  }
  =
  \sF(f^\dagger)  .
  $$
\end{construction}

\begin{ex}[Dagger self-enrichment] 
  \label{ex:InvolutiveSelfEnrichment}
  Starting with a unitary monoidal category $\sU$, 
  we can construct a weak enriched dagger structure on the self-enrichment $\cU$ from Example \ref{ex:SelfEnrichment}.
  First, we use the unitary dual functor of $\sU$ to construct the self-enrichment,
  which allows us to identify $u^*=\overline{u}$ for all $u\in \sU$.
  We then define 
  $$
  \kappa_{u\to v}^\cU
  :=
  \nu_{u, \overline{v}}^{-1} \circ (\id_{\overline{u}} \varphi_v^{-1})
  \in 
  \sU(\overline{\cU(v\to u)} \to \cU(u\to v))
  =
  \sU(\overline{\overline{v} u} \to \overline{u} v).
  $$
  We leave the straightforward verification to the reader.
  Later on in Example \ref{ex:DaggerSelfEnrichment}, we show that $(\cU,\kappa)$ is a dagger $\sU$-category.

  Now let $\cU^\sU$ be the underlying category of $\cU$, which is equipped with the dagger structure from Construction \ref{const:DaggerFromInvolutive}.
  The functor $\sF:\sU \to \cU^\sU$ which is the identity on objects and has
  $\sF(f:u\to v):= \coev_u\circ (\id_{\overline{u}}f)$
  is a dagger isomorphism with dagger inverse $\sG: \cU^\sU \to \sU$ given by
  $\sG(g\in \sU(1\to \overline{u}v)):= (\id_u g)\circ (\ev_u \id_v)$.
  Indeed, it is clear $\sF,\sG$ are isomorphisms of categories, so it suffices to show $\sF$ is dagger.
  Observe that
  \[
  \sF(f)^\dag
  =
  \begin{tikzpicture}[smallstring, baseline=20]
    \node[box] (f) at (0,0) {$\overline{\sF(f)}$};
    \node[box] (kappa) at (0,1) {$\kappa_{v \to u}^\cU$};
    \node (top) at (0,2) {$\overline{v} u$};
    \draw (f) to (kappa);
    \draw (kappa) to (top);
  \end{tikzpicture}
  =
  \begin{tikzpicture}[smallstring, baseline=30]
    \node[box] (f) at (0.5,0) {$\overline{\coev_u\circ (\id_{\overline{u}}f)}$};
    \node[box] (nu) at (0.5,1) {$\nu_{v, \overline{u}}^{-1}$};
    \node[box] (phi) at (1,2) {$\varphi_u$};
    \node (v) at (0,3) {$\overline{v}$};
    \node (u) at (1,3) {$u$};
    \draw[double] (f) to (nu);
    \draw (nu.115) to[in=-90,out=90] (v |- phi.south);
    \draw (v |- phi.south) to (v);
    \draw (nu.65) to[in=-90,out=90] (phi);
    \draw (phi) to (u);
  \end{tikzpicture}
  =
  \begin{tikzpicture}[smallstring, baseline=30]
    \node[box] (f) at (0,1.5) {$\overline{f}$};
    \node[box] (ev) at (0.5,0) {$\ev_{\overline{u}}^\dagger$};
    \node[box] (phi) at (1,1.5) {$\varphi_u^{-1}$};
    \node (u) at (1,3) {$u$};
    \node (v) at (0,3) {$\overline{v}$};
    \draw (ev.115) to[in=-90,out=90] (f);
    \draw (f) to (v);
    \draw (ev.65) to[in=-90,out=90] (phi);
    \draw (phi) to (u);
  \end{tikzpicture}
  =
  \begin{tikzpicture}[smallstring, baseline=20]
    \node[box] (f) at (0,1) {$\overline{f}$};
    \node (u) at (1,2) {$u$};
    \node (v) at (0,2) {$\overline{v}$};
    \draw (f) to (v);
    \draw (f) to[in=-90,out=-90] (u |- f.south);
    \draw (u |- f.south) to (u);
  \end{tikzpicture}
  =
  \begin{tikzpicture}[smallstring, baseline=20]
    \node[box] (f) at (1,1) {$\overline{f^*}$};
    \node (u) at (1,2) {$u$};
    \node (v) at (0,2) {$\overline{v}$};
    \draw (f) to (u);
    \draw (f) to[in=-90,out=-90] (v |- f.south);
    \draw (v |- f.south) to (v);
  \end{tikzpicture}
  =
  \sF(f^\dagger).
  \]
\end{ex}

%\begin{rem}
%When $\sU$ is a unitary monoidal category, and $\cU$ is the self-enrichment from Example \ref{ex:InvolutiveSelfEnrichment}, there is also another weak dagger $\sU$-category called $\widetilde{\cU}$.
%
%The $\sU$-category $\widetilde{\cU}$ has hom objects $\widetilde{\cU}(u\to v) := \cU(\overline{u} \to \overline{v}) = \overline{\overline{u}}\,\,\overline{v}$,
%the identity elements are given by 
%$$
%j_u^{\widetilde{\cU}} 
%:=
%\coev_{\overline{u}}
%\in 
%\sU(1_\sU \to \widetilde{\cU}(u\to u))
%=
%\sU(1_\sU \to  \overline{\overline{u}}\,\,\overline{u}),
%$$
%and composition is defined as
%\begin{align*}
%\id_{\overline{\overline{u}}} \ev_{\overline{v}} \id_{\overline{w}}
%\in
%\sU(
%\widetilde{\cU}(u\to v) \,\,\widetilde{\cU}(v\to w)
%\to 
%\widetilde{\cU}(u\to w)
%)
%=
%\sU(\overline{\overline{u}}\,\,\overline{v}\,\,\overline{\overline{v}}\,\,\overline{w} \to \overline{\overline{u}}\,\,\overline{w}).
%\end{align*} 
%Finally, we define
%$$
%\kappa^{\widetilde{\sU}}_{u\to v}
%:=
%(\id_{\overline{\overline{u}}} \varphi_{\overline{v}})\circ \nu_{\overline{u},\overline{\overline{v}}}
%\in
%\sU(\widetilde{\cU}(u\to v) \to \overline{\widetilde{\cU}(v\to u)})
%=
%\sU(\overline{\overline{u}}\,\,\overline{v} \to \overline{\overline{\overline{v}}\,\,\overline{u}}).
%$$
%Again, it is straightforward to verify that $\widetilde{\sU}$ is a weak dagger $\sU$-category.
%\end{rem}

%%%%%%%%%%%%%%%%%%%%%%%%%%%%%%%%%%%%%%%%%%%%%%%
\subsection{2-functor \texorpdfstring{$\dagvcat$}{DagVCat} to \texorpdfstring{$\dagvmod$}{DagVMod}}
\label{sec:first 2-functor}

Let $\sV$ be an involutive rigid monoidal category,
and let $\cV$ be the self-enrichment from Example \ref{ex:SelfEnrichment}.
We now assume that $\cV$ is equipped with a weak enriched dagger structure $\kappa$.
There is a canonical isomorphism between $\sV$ and its underlying category $\cV^\sV$ which is the identity on objects.
This isomorphism allows us to transport the dagger structure on $\cV^\sV$ from Construction \ref{const:DaggerFromInvolutive} to $\sV \cong \cV^\sV$ by the formula
$$
\left(
\tikzmath{
  \draw (0,-.7) node[right, yshift=.2cm]{$\scriptstyle u$} -- (0,.7) node[right, yshift=-.2cm]{$\scriptstyle v$};
  \roundNbox{fill=white}{(0,0)}{.3}{0}{0}{$f$}
}
\right)^\dag
:=
\begin{tikzpicture}[smallstring, baseline=35]
  \node (bot) at (0,0) {$v$};
  \node[box] (coev) at (1,1) {$\overline{\coev_u}$};
  \node[box] (f) at (1,2) {$\overline{\id_{u^*} f}$};
  \node[box] (kappa) at (1,3) {$\kappa_{v \to u}$};
  \node (top) at (kappa.45 |- 0,4) {$u$};
  \draw (bot) to (0,0 |- kappa.north);
  \draw (0,0 |- kappa.north) to[in=90,out=90] (kappa.135);
  \draw[double] (coev) to (f);
  \draw[double] (f) to (kappa);
  \draw (kappa.45) to (top);
\end{tikzpicture}
\,.
$$
We remark that this dagger structure is not necessarily compatible with the original involutive structure of $\sV$, unlike the dagger on a unitary monoidal category.
Indeed, we do not enforce any compatibility between $\kappa$ and the cups and caps of $\sV$.

\begin{rem}
  \label{rem:CanonicalIso u* -> u bar}
  Observe that this $\kappa$ on $\cV$ gives a canonical isomorphism 
  $$
  \kappa_{u\to 1} : \overline{\sV(1\to u)} = \overline{u} \to \sV(u\to 1)=u^*.
  $$
  At this time we do not know if this $\kappa_{\bullet \to 1}$ satisfies any additional coherence properties.
\end{rem}

In the remainder of this section, we construct a 2-functor $\Phi : \dagvcat \to \dagvmod$. 
% in the case that $\sV$ is a unitary monoidal category.
Recall that in \cite[\S4]{1809.09782}, one constructs a right $\sV$-module category $\sA$ from a tensored $\sV$-category $\cA$ by defining the $\sV$-action via the adjunction
\begin{equation}
  \label{eq:CatToModAdjunction}
  \sA(a\lhd u \to b)
  \cong
  \sV(u \to \cA(a\to b)).
\end{equation}
Let $(\cA, \kappa)$ be a dagger $\sV$-category,
and define $\Phi((\cA, \kappa))$ to be the $\sV$-module category $\sA$,
equipped with the dagger structure from Construction \ref{const:DaggerFromInvolutive}.
In order to show that $\sA$ is a \emph{dagger} $\sV$-module category,
we must prove the following three things:
\begin{enumerate}[label=($\dag\lhd$\arabic*)]
\item
  \label{eq:daglhd1}
  for $u\in \sV$, the object $a\lhd u \in \sA$ is determined up to canonical \emph{unitary} isomorphism rather than up to canonical isomorphism 
\item
  \label{eq:daglhd2}
  $-\lhd -$ is a dagger functor, and
\item
  \label{eq:daglhd3}
  the coherators $\alpha$ and $\rho$ are unitary.
\end{enumerate}

We begin by proving an essential identity needed to show that $\Phi((\cA, \kappa))$ satisfies \ref{eq:daglhd1}--\ref{eq:daglhd3}.
The following fact allows us to describe left $\sV$-adjoints of $\cA(a\to -)$.

\begin{fact}[{\cite[Lem.~4.9]{1809.09782}}]
  \label{fact:MateOfLa}
  Suppose $\cA$ is a tensored $\sV$-category, so that $\cA(a\to -)$ admits a left $\sV$-adjoint $a\lhd -$.
  Then the mate of $(a\lhd -)_{u\to v}$ under the adjunction
  \begin{equation}
    \label{eq:AdjunctionForLa}
    \sV(\cV(u\to v) \to \cA(a\lhd u\to a\lhd v))
    \cong
    \sA(a\lhd u \lhd \cV(u\to v) \to a\lhd v)
  \end{equation}
  is given by $\alpha_{a, u, u^*v}^{-1} \circ (\id_a \lhd \ev_u \id_v)$. 
  Observe that when $u=1$, we have
  $$
  \alpha_{a, 1, v}^{-1} \circ (\id_a \lhd \id_v)
  =
  (\rho_a \lhd \id_v) \circ (\id_{a \lhd v})
  =
  \rho_a \lhd \id_v,
  $$
  and taking mates under the adjunction, we record the identity
  \begin{equation}
    \label{eq:La - 1->v}
    (a\lhd -)_{1 \to u}
    =
    \begin{tikzpicture}[smallstring, baseline=25]
      \node (bot) at (0,0) {$u$};
      \node[box] (rho) at (2,1) {$\rho_a \lhd \id_u$};
      \node[box] (eta) at (0,1) {$\eta_{a \lhd 1, u}$};
      \node[box] (circ) at (1,2) {$- \circ -$};
      \node (top) at (1,3) {$\cA( a \lhd 1 \to a \lhd u )$};
      \draw (bot) to (eta);
      \draw (rho) to[in=-90,out=90] (circ.-45);
      \draw (eta) to[in=-90,out=90] (circ.225);
      \draw (circ) to (top);
    \end{tikzpicture}
    \underset{\eqref{eq:DefineLhd}}{=}
    \begin{tikzpicture}[smallstring, baseline=25]
      \node (bot) at (2,0) {$u$};
      \node[box] (rho) at (0,1) {$\rho_a$};
      \node[box] (eta) at (2,1) {$\eta_{a, u}$};
      \node[box] (circ) at (1,2) {$- \circ -$};
      \node (top) at (1,3) {$\cA( a \lhd 1 \to a \lhd u )$};
      \draw (bot) to (eta);
      \draw (rho) to[in=-90,out=90] (circ.225);
      \draw (eta) to[in=-90,out=90] (circ.-45);
      \draw (circ) to (top);
    \end{tikzpicture}
    \, .
  \end{equation}
\end{fact}

\begin{ex}
  \label{ex:DaggerSelfEnrichment}
  Suppose $\sU$ is a unitary monoidal category.
  We can now show that the self-enrichment $(\cU,\kappa^\cU)$ from Example \ref{ex:InvolutiveSelfEnrichment} is a dagger $\sU$-category.
  Indeed, $\rho_u: u1\to u$ is unitary, and \ref{kappa:4} follows from the following calculation. 
  Taking the mate of $\kappa_{v \to w}^{\cU} \circ (u \lhd -)_{v \to w}$ and suppressing coherators in $\sU$, we have
  \begin{align*}
    \begin{tikzpicture}[smallstring, baseline=35]
      \node (u-bot) at (0.5,0) {$u v$};
      \node (wv) at (1.75,0) {$\overline{\cU(w \to v)}$};
      \node (u-top) at (0.4,3.5) {$u$};
      \node (w) at (1.875,3.5) {$w$};
      \draw[double] (wv) to (1.75,0.5);
      \draw[rounded corners] (1.25,0.5) rectangle (2.25,1.25) node[midway] {$\nu_{v, \overline{w}}^{-1}$};
      \draw[rounded corners] (1.5,2) rectangle (2.25,2.75) node[midway] {$\varphi_w^{-1}$};
      \draw (0.4,0 |- u-bot.north) to (u-top);
      \draw (u-bot) to (0.5,1.25);
      \draw (0.5,1.25) to[in=90,out=90] (1.5,1.25);
      \draw (1.875,1.25) to (1.875,2);
      \draw (1.875,2.75) to (w);
    \end{tikzpicture}
    =
    \begin{tikzpicture}[smallstring, baseline=20]
      \node (u-bot) at (0.4,0) {$u$};
      \node (wv) at (2.5,0) {$\overline{\cU(w \to v)}$};
      \node (u-top) at (0.4,2) {$u w$};
      \node (v) at (1.35,0) {$v$};
      \draw (u-bot) to (u-top);
      \draw[double] (wv) to (2.5,1);
      \draw (v) to (1.33,1);
      \draw[double] (2.5,1) to[in=90,out=90] (1.3,1);
      \draw (1.27,1) to[in=-90,out=-90] (0.5,1);
      \draw (0.5,1) to (0.5,0 |- u-top.south);
    \end{tikzpicture}
    =
    \begin{tikzpicture}[smallstring, baseline=30]
      \node (top) at (0,3.5) {$u w$};
      \node (uv) at (0,0) {$u v$};
      \node (wv) at (3,0) {$\overline{\cU(w \to v)}$};
      \modulebox{0,0.5}{2}{0.75}{$\varepsilon_{u w \to u v}^\dagger$};
      \draw[rounded corners] (0.25,1.5) rectangle (2.5,2.25) node[midway] {$(u \lhd -)_{w \to v}^\dagger$};
      \draw (uv) to (top);
      \draw (0.1,0 |- uv.north) to (0.1,0.5);
      \draw (0.1,1.25) to (0.1,0 |- top.south);
      \draw (1.25,1.25) to (1.25,1.5);
      \draw (1.25,2.25) to (1.25,2.5);
      \draw (wv) to (3,2.5);
      \draw (3,2.5) to[in=90,out=90] (1.25,2.5);
    \end{tikzpicture}
    =
    \begin{tikzpicture}[smallstring, baseline=20]
      \node (top) at (0,2.5) {$u w$};
      \node (uv) at (0,0) {$u v$};
      \node (wv) at (3.15,0) {$\overline{\cU(w \to v)}$};
      \draw[rounded corners] (1.9,0.5) rectangle (4.4,1.25) node[midway] {$\overline{(u \lhd -)_{w \to v}}$};
      \modulebox{0,0.5}{1.65}{0.75}{$\varepsilon_{u w \to u v}^\dagger$};
      \draw (uv) to (top);
      \draw (0.1,0 |- uv.north) to (0.1,0.5);
      \draw (0.1,1.25) to (0.1,0 |- top.south);
      \draw (wv) to (3.25,0.5);
      \draw (1,1.25) to[in=90,out=90] (3.15,1.25);
    \end{tikzpicture}
    =
    \begin{tikzpicture}[smallstring, baseline=20]
      \node (top) at (1,2.5) {$u w$};
      \node (uv) at (1.75,0) {$u v$};
      \node (wv) at (3.25,0) {$\overline{\cU(w \to v)}$};
      \draw[double] (wv) to (3.25,0.5);
      \draw[rounded corners] (2,0.5) rectangle (4.5,1.25) node[midway] {$\overline{(u \lhd -)_{w \to v}}$};
      \draw (uv) to (1.75,1.25);
      \draw (1,1.25) to (1,0 |- top.south);
      \draw (1.75,1.25) to[in=90,out=90] (3.25,1.25);
      \draw (1.65,1.25) to[in=90,out=90] (3.35,1.25);
      \draw (1,1.25) to[in=-90,out=-90] (1.65,1.25);
    \end{tikzpicture}\, .
  \end{align*}
  It is straightforward to check that this is the mate of
  $\overline{(u \lhd -)_{w \to v}} \circ \kappa_{u v \to u w}^\cA$.
\end{ex}

The following lemma is the essential ingredient in the calculations of this subsection.

\begin{lem} \label{lem:eta-kappa-inv}
  If $\cA$ is a dagger $\sV$-category, then
  $
  \begin{tikzpicture}[smallstring, baseline=30]
    \node (bot) at (0,0) {$v$};
    \node (top) at (0,3) {$\overline{\cA(a \lhd v \to a)}$};
    \node[box] (eta) at (0,1) {$\eta_{a, v}$};
    \node[box] (kappa) at (0,2) {$\kappa_{a \to a \lhd v}^{-1}$};
    \draw (bot) to (eta);
    \draw (eta) to (kappa);
    \draw (kappa) to (top);
  \end{tikzpicture}
  =
  \begin{tikzpicture}[smallstring, baseline=0]
    \node (top) at (0,3) {$\overline{\cA(a \lhd v \to a)}$};
    \node (bot) at (3,-4) {$v$};
    \node[box] (circ3) at (0,2.1) {$- \circ_{\overline{\cA}} -$};
    \node[box] (rho) at (-1,1) {$\overline{(\rho_{a}^{\dagger})^{-1}}$};
    \node[box] (circ2) at (1,1) {$- \circ_{\overline{\cA}} -$};
    \node[box] (ev) at (0,0) {$\overline{\id_a \lhd \ev_v}$};
    \node[box] (circ1) at (2,0) {$- \circ_{\overline{\cA}} -$};
    \node[box] (alpha1) at (1,-1.1) {$\overline{\alpha_{a, v, v^*}^{-1}}$};
    \node[box] (eta) at (3,-1.1) {$\overline{\eta_{a \lhd v , v^*}}$};
    \node[box] (phi) at (3,-3) {$\varphi_v$};
    \node[box] (kappa) at (3,-2) {$\overline{\kappa_{v \to 1}}$};
    \draw (circ3) to (top);
    \draw (rho) to[in=-90,out=90] (circ3.225);
    \draw (circ2) to[in=-90,out=90] (circ3.-45);
    \draw (ev) to[in=-90,out=90] (circ2.225);
    \draw (circ1) to[in=-90,out=90] (circ2.-45);
    \draw (alpha1) to[in=-90,out=90] (circ1.225);
    \draw (eta) to[in=-90,out=90] (circ1.-45);
    \draw (phi) to (kappa);
    \draw (kappa) to (eta);
    \draw (bot) to (phi);
  \end{tikzpicture}
  \,.
  $
\end{lem}

\begin{proof}
  First, take the bar of the claimed identity and postcompose with $\varphi_{\cA(a \lhd v \to a)}^{-1}$.
  Since $\varphi$ is natural, this $\varphi_{\cA(a \lhd v \to a)}^{-1}$ cancels with the existing $\varphi_v$;
  applying \ref{kappa:1} and inverting $\rho$ yields the equivalent equality
  $$
  \begin{tikzpicture}[smallstring, baseline=35]
    \node (bot) at (0,0) {$\overline{v}$};
    \node (top) at (1,4) {$\cA(a \lhd v \to a \lhd 1)$};
    \node[box] (eta) at (0,1) {$\overline{\eta_{a, v}}$};
    \node[box] (kappa) at (0,2) {$\kappa_{a \lhd v \to a}$};
    \node[box] (rho) at (2,2) {$\rho_a^\dagger$};
    \node[box] (circ) at (1,3) {$- \circ -$};
    \draw (bot) to (eta);
    \draw (eta) to (kappa);
    \draw (kappa) to[in=-90,out=90] (circ.225);
    \draw (rho) to[in=-90,out=90] (circ.-45);
    \draw (circ) to (top);
  \end{tikzpicture}
  =
  \begin{tikzpicture}[smallstring, baseline=35]
    \node (bot) at (0,-1) {$\overline{v}$};
    \node (top) at (2,4) {$\cA(a \lhd v \to a \lhd 1)$};
    \node[box] (kappa) at (0,0) {$\kappa_{v \to 1}^{\cV}$};
    \node[box] (eta) at (0,0.9) {$\eta_{a \lhd v , {v^*}}$};
    \node[box] (alpha) at (2,0.9) {$\alpha_{a , v , {v^*}}^{-1}$};
    \node[box] (circ1) at (1,2) {$- \circ -$};
    \node[box] (ev) at (3,2) {$\id_a \lhd \ev_v$};
    \node[box] (circ2) at (2,3) {$- \circ -$};
    \draw (bot) to (kappa);
    \draw (kappa) to (eta);
    \draw (eta) to[in=-90,out=90] (circ1.225);
    \draw (alpha) to[in=-90,out=90] (circ1.-45);
    \draw (circ1) to[in=-90,out=90] (circ2.225);
    \draw (ev) to[in=-90,out=90] (circ2.-45);
    \draw (circ2) to (top);
  \end{tikzpicture}.
  $$
  To prove this equality, we apply \ref{kappa:2} to the left hand side to obtain
  \begin{align*}
    \begin{tikzpicture}[smallstring, baseline=35]
      \node (bot) at (0,0) {$\overline{v}$};
      \node (top) at (1,4) {$\cA(a \lhd v \to a \lhd 1)$};
      \node[box] (eta) at (0,1) {$\overline{\eta_{a, v}}$};
      \node[box] (rho) at (2,1) {$\overline{\rho_a}$};
      \node[box] (circ) at (1,2) {$- \circ_{\overline{\cA}} -$};
      \node[box] (kappa) at (1,3) {$\kappa_{a \lhd v \to a \lhd 1}$};
      \draw (bot) to (eta);
      \draw (eta) to[in=-90,out=90] (circ.225);
      \draw (rho) to[in=-90,out=90] (circ.-45);
      \draw (circ) to (kappa);
      \draw (kappa) to (top);
    \end{tikzpicture}
    &\underset{\eqref{eq:La - 1->v}}{=}
    \begin{tikzpicture}[smallstring, baseline=30]
      \node (bot) at (0,0) {$\overline{v}$};
      \node (top) at (0,3) {$\cA(a \lhd v \to a \lhd 1)$};
      \node[box] (La) at (0,1) {$\overline{(a\lhd -)_{1\to v}}$};
      \node[box] (kappa) at (0,2) {$\kappa_{a\lhd v \to a\lhd 1}$};
      \draw (bot) to (La);
      \draw (eta) to (kappa);
      \draw (kappa) to (top);
    \end{tikzpicture}
    \underset{\text{\ref{kappa:4}}}{=}
    \begin{tikzpicture}[smallstring, baseline=25]
      \node (bot) at (0,0) {$\overline{v}$};
      \node (top) at (0,3) {$\cA(a \lhd v \to a \lhd 1)$};
      \node[box] (l) at (0,2) {$(a\lhd -)_{v \to 1}$};
      \node[box] (kappa) at (0,1) {$\kappa_{v \to 1}^\cV$};
      \draw (bot) to (kappa);
      \draw (kappa) to (l);
      \draw (l) to (top);
    \end{tikzpicture}
    \underset{\text{(Fact~\ref{fact:MateOfLa})}}{=}
    \begin{tikzpicture}[smallstring, baseline=35]
      \node (bot) at (0,-1) {$\overline{v}$};
      \node (top) at (2,4) {$\cA(a \lhd v \to a \lhd 1)$};
      \node[box] (kappa) at (0,0) {$\kappa_{v \to 1}^{\cV}$};
      \node[box] (eta) at (0,0.9) {$\eta_{a \lhd v , {v^*}}$};
      \node[box] (alpha) at (2,0.9) {$\alpha_{a , v , {v^*}}^{-1}$};
      \node[box] (circ1) at (1,2) {$- \circ -$};
      \node[box] (ev) at (3,2) {$\id_a \lhd \ev_v$};
      \node[box] (circ2) at (2,3) {$- \circ -$};
      \draw (bot) to (kappa);
      \draw (kappa) to (eta);
      \draw (eta) to[in=-90,out=90] (circ1.225);
      \draw (alpha) to[in=-90,out=90] (circ1.-45);
      \draw (circ1) to[in=-90,out=90] (circ2.225);
      \draw (ev) to[in=-90,out=90] (circ2.-45);
      \draw (circ2) to (top);
    \end{tikzpicture}.
    \qedhere
  \end{align*}
\end{proof}

We are now in the position to prove \ref{eq:daglhd1}.
Suppose that $\cL^a$ and $a\lhd -$ are two left $\sV$-adjoints of the $\sV$-functor $\cA(a\to -): \cA\to \cV$
satisfying \ref{kappa:3} and \ref{kappa:4},
where $\cV$ is the self-enrichment from Example \ref{ex:SelfEnrichment}.
Consider the two adjunctions
\begin{align}
  \sA(\cL^a(u) \to a\lhd u)
  &\cong
  \sV(u\to \cA(a\to a\lhd u))
  \cong
  \sA(a\lhd u \to a\lhd u)
  \label{eq:GetPsi}
  \\
  \sA(\cL^a(u) \to \cL^a(u))
  &\cong
  \sV(u\to \cA(a\to \cL^a(u)))
  \cong
  \sA(a\lhd u \to \cL^a(u)).
  \label{eq:GetPsiInverse}
\end{align}
By the Yoneda Lemma, the mate of $\id_{a\lhd u}$ under Adjunction \eqref{eq:GetPsi} gives a natural isomorphism of the underlying functors $\psi^a: \sL^a \Rightarrow a\lhd -$,
and its inverse is obtained by taking the mate of $\id_{\cL^a(u)}$ under Adjunction \eqref{eq:GetPsiInverse}.
We immediately get the following identities between the units $\eta^\cL,\eta^\lhd$ of the two underlying adjunctions:
\begin{equation}
  \label{eq:TwoUnitIdentity}
  \tikzmath[smallstring]{
    \node (top) at (0,2) {$\cA(a \to a \lhd u)$};
    \node (bot) at (0,0) {$u$};
    \node[box] (eta) at (0,1) {$\eta^\lhd_{a,u}$};
    \draw (bot) to (eta);
    \draw (eta) to (top);
  }
  =
  \tikzmath[smallstring]{
    \node (top) at (.5,3) {$\cA(a \to a \lhd u)$};
    \node[box] (circ) at (.5,2) {$-\circ_\cA-$};
    \node[box] (eta) at (0,1) {$\eta^\cL_{a,u}$};
    \node[box] (psi) at (1,1) {$\psi^a_u$};
    \node (bot) at (0,0) {$u$};
    \draw (bot) to (eta);
    \draw (eta) to[in=-90,out=90] node[left]{$\scriptstyle \cA(a\to \cL^a(u))$} (circ.-135);
    \draw (psi) to[in=-90,out=90] node[right]{$\scriptstyle \cA(\cL^a(u)\to a\lhd u)$} (circ.-45);
    \draw (circ) to (top);
  }.
\end{equation}
Since the coherators $\alpha, \rho$ for the right module $\sV$-category $\sA$ were defined under Adjunction \ref{eq:CatToModAdjunction},
we get two variants $\alpha^\cL, \alpha^\lhd$ and $\rho^\cL,\rho^\lhd$ depending on which $\sV$-functor we choose.
We have the following identities between these coherators:
\begin{equation}
  \label{eq:TwoAlphaRhoIdentity}
  \begin{tikzcd}
    \cL^a(1) 
    \arrow[dr, "\rho^\cL_a"]
    \arrow[rr, "\psi^a_1"]
    && 
    a\lhd 1 
    \arrow[dl,"\rho_a^\lhd"]
    &&
    \cL^a(uv)
    \arrow[d, "\alpha^\cL_{a,u,v}"]
    \arrow[r, "\psi^a_{uv}"]
    &
    a\lhd uv
    \arrow[r, "\alpha^\lhd_{a,u,v}"]
    &
    a\lhd u \lhd v
    \\
    &
    a
    &&&
    \cL^{\cL^a(u)}(v)
    \arrow[r, "\psi^{\cL^a(u)}_{v}"]
    &
    \cL^a(u) \lhd v
    \arrow[ur, swap, "\psi^a_{u}\lhd \id_v"]
  \end{tikzcd}\, .
\end{equation}
The natural isomorphism $\psi$ actually defines a $\sV$-natural isomorphism $\cL^a\Rightarrow a\lhd -$.

\begin{lem}
  Suppose $\cA$ is a $\sV$-category and $\cL^a, a\lhd - $ are two left $\sV$-adjoints of $\cA(a\to -)$.
  Then
  $$
  \tikzmath[smallstring]{
    \node (top) at (.6,3) {$\cA(\cL^a(u) \to a \lhd v)$};
    \node[box] (circ) at (.6,2) {$-\circ_\cA-$};
    \node[box] (eta) at (0,1) {$\cL^a_{u\to v}$};
    \node[box] (psi) at (1.2,1) {$\psi^a_v$};
    \node (bot) at (0,0) {$\cV(u\to v)=u^*v$};
    \draw (bot) to (eta);
    \draw (eta) to[in=-90,out=90] node[left]{$\scriptstyle \cA(\cL^a(u)\to \cL^a(v))$} (circ.-135);
    \draw (psi) to[in=-90,out=90] node[right]{$\scriptstyle \cA(\cL^a(v)\to a\lhd v)$} (circ.-45);
    \draw (circ) to (top);
  }
  =
  \tikzmath[smallstring]{
    \node (top) at (.2,3) {$\cA(\cL^a(u) \to a \lhd v)$};
    \node[box] (circ) at (.2,2) {$-\circ_\cA-$};
    \node[box] (eta) at (1,1) {$(a\lhd -)_{u\to v}$};
    \node[box] (psi) at (-.6,1) {$\psi^a_u$};
    \node (bot) at (1,0) {$u^*v$};
    \draw (bot) to (eta);
    \draw (eta) to[in=-90,out=90] node[right]{$\scriptstyle \cA(a\lhd u\to a\lhd v)$} (circ.-45);
    \draw (psi) to[in=-90,out=90] node[left]{$\scriptstyle \cA(\cL^a(u)\to a\lhd u)$} (circ.-135);
    \draw (circ) to (top);
  }\,.
  $$
\end{lem}
\begin{proof}
  Using Fact \ref{fact:MateOfLa} for $\cL^a$, the mate of the left hand side is given by
  \begin{align*}
    (\alpha^\cL_{a , u , {u^*}v})^{-1}
    \circ
    \cL^a(\ev_u\id_v)
    \circ
    \psi_v^a
    &=
    (\alpha^\cL_{a , u , {u^*}v})^{-1}
    \circ
    \psi_{u{u^*}v}^a
    \circ
    (\id_a\lhd \ev_u\id_v)
    &&\text{($\psi$ natural)}
    \\&=
    \psi^{\cL^a(u)}_{u^*v}  
    \circ 
    (\psi^{a}_u\lhd \id_{u^*v}) 
    \circ 
    (\alpha_{a,u,{u^*}v}^{\lhd})^{-1}
    \circ
    (\id_a\lhd \ev_u\id_v)
    &&\eqref{eq:TwoAlphaRhoIdentity}
  \end{align*}
  Now taking mates again and simplifying, we have
  \begin{align*}
    \tikzmath[smallstring]{
      \node (bot) at (0,0) {$u^*v$};
      \node (top) at (4,6) {$\cA(a \lhd u \to a \lhd v)$};
      \node[box] (eta) at (0,1) {$\eta^\cL_{\cL^a(u) , u^*v}$};
      \node[box] (psi) at (2,1) {$\psi^{\cL^a(u)}_{u^*v}$};
      \node[box] (circ1) at (1,2) {$- \circ -$};
      \node[box] (psi-id) at (3,2) {$\psi^a_u \lhd \id_{u^*v}$};
      \node[box] (circ2) at (2,3) {$- \circ -$};
      \node[box] (alpha) at (4,3) {$(\alpha^\lhd_{a , u , u^*v})^{-1}$};
      \node[box] (circ3) at (3,4) {$- \circ -$};
      \node[box] (ev) at (5,4) {$\id_a \lhd \ev_u\id_v$};
      \node[box] (circ4) at (4,5) {$- \circ -$};
      \draw (bot) to (eta);
      \draw (eta) to[in=-90,out=90] (circ1.225);
      \draw (psi) to[in=-90,out=90] (circ1.-45);
      \draw (circ1) to[in=-90,out=90] (circ2.225);
      \draw (psi-id) to[in=-90,out=90] (circ2.-45);
      \draw (circ2) to[in=-90,out=90] (circ3.225);
      \draw (alpha) to[in=-90,out=90] (circ3.-45);
      \draw (circ3) to[in=-90,out=90] (circ4.225);
      \draw (ev) to[in=-90,out=90] (circ4.-45);
      \draw (circ4) to (top);
    }
    \underset{\eqref{eq:TwoUnitIdentity}}{=}
    \tikzmath[smallstring]{
      \node (bot) at (1,1) {$u^*v$};
      \node (top) at (4,6) {$\cA(a \lhd u \to a \lhd v)$};
      \node[box] (eta) at (1,2) {$\eta^\lhd_{\cL^a(u) , u^*v}$};
      \node[box] (psi-id) at (3,2) {$\psi^a_u \lhd \id_{u^*v}$};
      \node[box] (circ2) at (2,3) {$- \circ -$};
      \node[box] (alpha) at (4,3) {$(\alpha^\lhd_{a , u , u^*v})^{-1}$};
      \node[box] (circ3) at (3,4) {$- \circ -$};
      \node[box] (ev) at (5,4) {$\id_a \lhd \ev_u\id_v$};
      \node[box] (circ4) at (4,5) {$- \circ -$};
      \draw (bot) to (eta);
      \draw (eta) to[in=-90,out=90] (circ2.225);
      \draw (psi-id) to[in=-90,out=90] (circ2.-45);
      \draw (circ2) to[in=-90,out=90] (circ3.225);
      \draw (alpha) to[in=-90,out=90] (circ3.-45);
      \draw (circ3) to[in=-90,out=90] (circ4.225);
      \draw (ev) to[in=-90,out=90] (circ4.-45);
      \draw (circ4) to (top);
    }
    \underset{\eqref{eq:DefineLhd}}{=}
    \tikzmath[smallstring]{
      \node (bot) at (1,1) {$u^*v$};
      \node (top) at (2,6) {$\cA(a \lhd u \to a \lhd v)$};
      \node[box] (eta) at (1,2) {$\eta^\lhd_{a\lhd u , u^*v}$};
      \node[box] (psi) at (-1,2) {$\psi^a_u$};
      \node[box] (circ2) at (0,3) {$- \circ -$};
      \node[box] (alpha) at (2,3) {$(\alpha^\lhd_{a , u , u^*v})^{-1}$};
      \node[box] (circ3) at (1,4) {$- \circ -$};
      \node[box] (ev) at (3,4) {$\id_a \lhd \ev_u\id_v$};
      \node[box] (circ4) at (2,5) {$- \circ -$};
      \draw (bot) to (eta);
      \draw (eta) to[in=-90,out=90] (circ2.-45);
      \draw (psi) to[in=-90,out=90] (circ2.-135);
      \draw (circ2) to[in=-90,out=90] (circ3.225);
      \draw (alpha) to[in=-90,out=90] (circ3.-45);
      \draw (circ3) to[in=-90,out=90] (circ4.225);
      \draw (ev) to[in=-90,out=90] (circ4.-45);
      \draw (circ4) to (top);
    }
    \,.
  \end{align*}
  By using associativity of $-\circ_\cA-$ and Fact \ref{fact:MateOfLa} for $a\lhd -$, the result follows.
\end{proof}

\begin{prop}[\ref{eq:daglhd1}]
  \label{prop:daglhd1}
  Suppose $\cA$ is a dagger $\sV$-category and $\cL^a, a\lhd - $ are two left $\sV$-adjoints of $\cA(a\to -)$ satisfying \ref{kappa:3} and \ref{kappa:4}.
  Then the natural isomorphism $\psi: \cL^a\Rightarrow a\lhd -$ is unitary.
\end{prop}
\begin{proof}
  Taking the mate of $(\psi^a_{u})^\dag$ under the first step of Adjunction \eqref{eq:GetPsiInverse} and using \ref{kappa:1}, we obtain
  \begin{align*}
    \tikzmath[smallstring]{
      \node (bot) at (0,-1) {$u$};
      \node (top) at (1,4) {$\cA(a \to \cL^a(u))$};
      \node[box] (eta) at (0,0) {$\eta^\lhd_{a, u}$};
      \node[box] (kappa1) at (0,1) {$\kappa^{-1}_{a\to a\lhd u}$};
      \node[box] (psi) at (2,1) {$\overline{\psi^a_u}$};
      \node[box] (circ) at (1,2) {$- \circ_{\overline{\cA}} -$};
      \node[box] (kappa2) at (1,3) {$\kappa_{a \to \cL^a(u)}$};
      \draw (bot) to (eta);
      \draw (eta) to (kappa1);
      \draw (kappa1) to[in=-90,out=90] (circ.225);
      \draw (psi) to[in=-90,out=90] (circ.-45);
      \draw (circ) to (kappa2);
      \draw (kappa2) to (top);
    }
    \quad
    \underset{\text{(Lem.~\ref{lem:eta-kappa-inv})}}{=}
    \quad
    \tikzmath[smallstring]{
      \node (top) at (1,5) {$\cA(a \to \cL^a(u))$};
      \node (bot) at (3,-4) {$v$};
      \node[box] (kappa) at (1,4) {$\kappa_{a \to \cL^a(u)}$};
      \node[box] (circ4) at (1,3) {$- \circ_{\overline{\cA}} -$};
      \node[box] (alpha2) at (2,2) {$\overline{\psi^a_u}$};
      \node[box] (circ3) at (0,2) {$- \circ_{\overline{\cA}} -$};
      \node[box] (rho) at (-1,1) {$\overline{(\rho_{a}^{\lhd})^{\dag -1}}$};
      \node[box] (circ2) at (1,1) {$- \circ_{\overline{\cA}} -$};
      \node[box] (ev) at (0,0) {$\overline{\id_a \lhd \ev_u}$};
      \node[box] (circ1) at (2,0) {$- \circ_{\overline{\cA}} -$};
      \node[box] (alpha1) at (1,-1) {$\overline{(\alpha^\lhd_{a,u,{u^*}})^{-1}}$};
      \node[box] (eta) at (3,-1) {$\overline{\eta^\lhd_{a \lhd u, {u^*}}}$};
      \node[box] (kappa0) at (3,-2) {$\overline{\kappa_{v \to 1}}$};
      \node[box] (phi) at (3,-3) {$\varphi_v$};
      \draw (kappa) to (top);
      \draw (circ4) to (kappa);
      \draw (alpha2) to[in=-90,out=90] (circ4.-45);
      \draw (circ3) to[in=-90,out=90] (circ4.225);
      \draw (rho) to[in=-90,out=90] (circ3.225);
      \draw (circ2) to[in=-90,out=90] (circ3.-45);
      \draw (ev) to[in=-90,out=90] (circ2.225);
      \draw (circ1) to[in=-90,out=90] (circ2.-45);
      \draw (alpha1) to[in=-90,out=90] (circ1.225);
      \draw (eta) to[in=-90,out=90] (circ1.-45);
      \draw (kappa0) to (eta);
      \draw (phi) to (kappa0);
      \draw (bot) to (phi);
    }
    \underset{\eqref{eq:DefineLhd}}{=}
    \quad
    \tikzmath[smallstring]{
      \node (top) at (0,4) {$\cA(a \to \cL^a(u))$};
      \node (bot) at (4,-5) {$v$};
      \node[box] (kappa) at (0,3) {$\kappa_{a \to \cL^a(u)}$};
      \node[box, blue] (circ3) at (0,2) {$- \circ_{\overline{\cA}} -$};
      \node[box, blue] (rho) at (-1,1) {$\overline{(\rho_{a}^{\lhd})^{\dag -1}}$};
      \node[box, blue] (circ2) at (1,1) {$- \circ_{\overline{\cA}} -$};
      \node[box, blue] (ev) at (0,0) {$\overline{\id_a \lhd \ev_u}$};
      \node[box, blue] (circ1) at (2,0) {$- \circ_{\overline{\cA}} -$};
      \node[box, blue] (alpha1) at (1,-1) {$\overline{(\alpha^\lhd_{a, u, {u^*}})^{-1}}$};
      \node[box, blue] (circ) at (3,-1) {$- \circ_{\overline{\cA}} -$};
      \node[box, blue] (alpha) at (2,-2) {$\overline{\psi^a_u \lhd \id_{{u^*}}}$};
      \node[box] (eta) at (4,-2) {$\overline{\eta^\lhd_{\cL^a(u), {u^*}}}$};
      \node[box] (kappa0) at (4,-3) {$\overline{\kappa_{v \to 1}}$};
      \node[box] (phi) at (4,-4) {$\varphi_v$};
      \draw (kappa) to (top);
      \draw (circ3) to[in=-90,out=90] (kappa);
      \draw[blue] (rho) to[in=-90,out=90] (circ3.225);
      \draw[blue] (circ2) to[in=-90,out=90] (circ3.-45);
      \draw[blue] (ev) to[in=-90,out=90] (circ2.225);
      \draw[blue] (circ1) to[in=-90,out=90] (circ2.-45);
      \draw[blue] (alpha1) to[in=-90,out=90] (circ1.225);
      \draw[blue] (circ) to[in=-90,out=90] (circ1.-45);
      \draw[blue] (alpha) to[in=-90,out=90] (circ.225);
      \draw (eta) to[in=-90,out=90] (circ.-45);
      \draw (kappa0) to (eta);
      \draw (phi) to (kappa0);
      \draw (bot) to (phi);
    }\,.
  \end{align*}
  Observe that the blue part of the above diagram is the bar of a morphism in $\sA$.
  Taking the bar and working in $\sA$, we see
  \begin{align*}
    (\psi^a_u \lhd \id_{{u^*}})
    &\circ
    (\alpha^\lhd_{a, u, {u^*}})^{-1}
    \circ
    (\id_a \lhd \ev_u)
    \circ 
    (\rho^\lhd_a)^{\dag-1}
    \\&=
    (\psi^{\cL^a(u)}_{{u^*}})^{-1}
    \circ
    (\alpha^\cL_{a, u, {u^*}})^{-1}
    \circ
    \psi^a_{u{u^*}}
    \circ
    (\id_a \lhd \ev_u)
    \circ 
    (\rho^\lhd_a)^{\dag-1}
    &&
    \eqref{eq:TwoAlphaRhoIdentity}
    \\&=
    (\psi^{\cL^a(u)}_{{u^*}})^{-1}
    \circ
    (\alpha^\cL_{a, u, {u^*}})^{-1}
    \circ
    \cL^a(\ev_u)
    \circ
    \psi^a_{1}
    \circ 
    (\rho^\lhd_a)^{\dag-1}
    &&
    \text{$\psi$ natural}
    \\&=
    (\psi^{\cL^a(u)}_{{u^*}})^{-1}
    \circ
    (\alpha^\cL_{a, u, {u^*}})^{-1}
    \circ
    \cL^a(\ev_u)
    \circ
    \psi^a_{1}
    \circ 
    \rho^\lhd_a
    &&
    \text{\ref{kappa:3} for $\rho^\lhd$}
    \\&=
    (\psi^{\cL^a(u)}_{{u^*}})^{-1}
    \circ
    (\alpha^\cL_{a, u, {u^*}})^{-1}
    \circ
    \cL^a(\ev_u)
    \circ 
    \rho^\cL_a
    &&
    \eqref{eq:TwoAlphaRhoIdentity}
    \\&=
    (\psi^{\cL^a(u)}_{{u^*}})^{-1}
    \circ
    (\alpha^\cL_{a, u, {u^*}})^{-1}
    \circ
    \cL^a(\ev_u)
    \circ 
    (\rho^\cL_a)^{\dag-1}
    &&
    \text{\ref{kappa:3} for $\rho^\cL$.}
  \end{align*}  
  Substituting this last map into the blue part of the above diagram, we obtain
  $$
  \tikzmath[smallstring]{
    \node (top) at (0,4) {$\cA(a \to \cL^a(u))$};
    \node (bot) at (4,-5) {$v$};
    \node[box] (kappa) at (0,3) {$\kappa_{a \to \cL^a(u)}$};
    \node[box, blue] (circ3) at (0,2) {$- \circ_{\overline{\cA}} -$};
    \node[box, blue] (rho) at (-1,1) {$\overline{(\rho_{a}^{\cL})^{\dag -1}}$};
    \node[box, blue] (circ2) at (1,1) {$- \circ_{\overline{\cA}} -$};
    \node[box, blue] (ev) at (0,0) {$\overline{\id_a \lhd \ev_u}$};
    \node[box, blue] (circ1) at (2,0) {$- \circ_{\overline{\cA}} -$};
    \node[box, blue] (alpha1) at (1,-1) {$\overline{(\alpha^\cL_{a, u, {u^*}})^{-1}}$};
    \node[box, blue] (circ) at (3,-1) {$- \circ_{\overline{\cA}} -$};
    \node[box, blue] (alpha) at (2,-2) {$\overline{\psi^a_u \lhd \id_{{u^*}}}$};
    \node[box] (eta) at (4,-2) {$\overline{\eta^\lhd_{\cL^a(u), {u^*}}}$};
    \node[box] (kappa0) at (4,-3) {$\overline{\kappa_{v \to 1}}$};
    \node[box] (phi) at (4,-4) {$\varphi_v$};
    \draw (kappa) to (top);
    \draw (circ3) to[in=-90,out=90] (kappa);
    \draw[blue] (rho) to[in=-90,out=90] (circ3.225);
    \draw[blue] (circ2) to[in=-90,out=90] (circ3.-45);
    \draw[blue] (ev) to[in=-90,out=90] (circ2.225);
    \draw[blue] (circ1) to[in=-90,out=90] (circ2.-45);
    \draw[blue] (alpha1) to[in=-90,out=90] (circ1.225);
    \draw[blue] (circ) to[in=-90,out=90] (circ1.-45);
    \draw[blue] (alpha) to[in=-90,out=90] (circ.225);
    \draw (eta) to[in=-90,out=90] (circ.-45);
    \draw (kappa0) to (eta);
    \draw (phi) to (kappa0);
    \draw (bot) to (phi);
  }
  \underset{\eqref{eq:TwoUnitIdentity}}{=}
  \quad
  \tikzmath[smallstring]{
    \node (top) at (0,4) {$\cA(a \to \cL^a(u))$};
    \node (bot) at (3,-4) {$v$};
    \node[box] (kappa) at (0,3) {$\kappa_{a \to \cL^a(u)}$};
    \node[box] (circ3) at (0,2) {$- \circ_{\overline{\cA}} -$};
    \node[box] (rho) at (-1,1) {$\overline{(\rho_{a}^{\cL})^{\dag -1}}$};
    \node[box] (circ2) at (1,1) {$- \circ_{\overline{\cA}} -$};
    \node[box] (ev) at (0,0) {$\overline{\id_a \lhd \ev_u}$};
    \node[box] (circ1) at (2,0) {$- \circ_{\overline{\cA}} -$};
    \node[box] (alpha1) at (1,-1) {$\overline{(\alpha^\cL_{a, u, {u^*}})^{-1}}$};
    \node[box] (eta) at (3,-1) {$\overline{\eta^\cL_{\cL^a(u), {u^*}}}$};
    \node[box] (kappa0) at (3,-2) {$\overline{\kappa_{v \to 1}}$};
    \node[box] (phi) at (3,-3) {$\varphi_v$};
    \draw (kappa) to (top);
    \draw (circ3) to[in=-90,out=90] (kappa);
    \draw (rho) to[in=-90,out=90] (circ3.225);
    \draw (circ2) to[in=-90,out=90] (circ3.-45);
    \draw (ev) to[in=-90,out=90] (circ2.225);
    \draw (circ1) to[in=-90,out=90] (circ2.-45);
    \draw (alpha1) to[in=-90,out=90] (circ1.225);
    \draw (eta) to[in=-90,out=90] (circ1.-45);
    \draw (kappa0) to (eta);
    \draw (phi) to (kappa0);
    \draw (bot) to (phi);
  }
  \underset{\text{(Lem.~\ref{lem:eta-kappa-inv})}}{=}
  \quad
  \tikzmath[smallstring]{
    \node (bot) at (0,-1) {$u$};
    \node (top) at (0,3) {$\cA(a \to \cL^a(u))$};
    \node[box] (eta) at (0,0) {$\eta^\cL_{a, u}$};
    \node[box] (kappa1) at (0,1) {$\kappa^{-1}_{a\to \cL^a(u)}$};
    \node[box] (kappa2) at (0,2) {$\kappa_{a \to \cL^a(u)}$};
    \draw (bot) to (eta);
    \draw (eta) to (kappa1);
    \draw (kappa1) to (kappa2);
    \draw (kappa2) to (top);
  }
  =
  \eta^\cL_{a,u}.
  $$
  But this is exactly the mate of $\psi^{-1}_{a,u}$ under the first step of Adjunction \eqref{eq:GetPsiInverse}, and the result follows.
\end{proof}

\begin{prop}[\ref{eq:daglhd2}]
  \label{prop:daglhd2}
  If $\cA$ is a dagger $\sV$-category, then $-\lhd - : \sA \times \sV \to \sA$ is a dagger functor.
\end{prop}

\begin{proof}
  First, for $g \in \sV(u \to v)$, the statement that
  $(\id_a \lhd g)^\dagger = \id_a \lhd g^\dagger$
  follows immediately from \ref{kappa:4} and Construction \ref{const:DaggerFromInvolutiveFunctor}.
  Then for $f \in \cA(a \to b)$,
  the mate of $(f \lhd \id_u)^\dagger$ under Adjunction \ref{eq:CatToModAdjunction} is given by the following:
  \begin{align*}
    \begin{tikzpicture}[smallstring, baseline=50]
      \node (top) at (1,5) {$\cA(b \to a \lhd u)$};
      \node (bot) at (0,0) {$u$};
      \node[box] (eta) at (0,1) {$\eta_{b, u}$};
      \node[box] (kappainv) at (0,2) {$\kappa_{b \to b \lhd u}^{-1}$};
      \node[box] (f) at (2,2) {$\overline{f \lhd \id_u}$};
      \node[box] (circ) at (1,3) {$- \circ_{\overline{\cA}} -$};
      \node[box] (kappa) at (1,4) {$\kappa_{b \to a \lhd u}$};
      \draw (bot) to (eta);
      \draw (eta) to (kappainv);
      \draw (kappainv) to[in=-90,out=90] (circ.225);
      \draw (f) to[in=-90,out=90] (circ.-45);
      \draw (circ) to (kappa);
      \draw (kappa) to (top);
    \end{tikzpicture}
    \quad&\underset{\text{(Lem.~\ref{lem:eta-kappa-inv})}}{=}\quad
    \begin{tikzpicture}[smallstring, baseline=50]
      \node (top) at (0,6) {$\cA(b \to a \lhd u)$};
      \node (bot) at (2,-3) {$u$};
      \node[box] (phi) at (2,-2) {$\varphi_u$};
      \node[box] (kappa0) at (2,-1.1) {$\overline{\kappa_{u \to 1}}$};
      \node[box] (eta) at (2,-0.2) {$\eta_{b \lhd u, {u^*}}$};
      \node[box] (alpha) at (0,-0.2) {$\overline{\alpha_{b,u,{u^*}}^{-1}}$};
      \node[box] (ev) at (-1,1) {$\overline{\id_b \lhd \ev_u}$};
      \node[box] (rho) at (-2,2) {$\overline{(\rho_{b}^\dagger)^{-1}}$};
      \node[box] (f) at (1,3) {$\overline{f \lhd \id_u}$};
      \node[box] (kappa) at (0,5) {$\kappa_{b \to a \lhd u}$};
      \node[box] (circ1) at (1,1) {$- \circ_{\overline{\cA}} -$};
      \node[box] (circ2) at (0,2) {$- \circ_{\overline{\cA}} -$};
      \node[box] (circ3) at (-1,3) {$- \circ_{\overline{\cA}} -$};
      \node[box] (circ4) at (0,4) {$- \circ_{\overline{\cA}} -$};
      \draw (bot) to (phi);
      \draw (phi) to (kappa0);
      \draw (kappa0) to (eta);
      \draw (eta) to[in=-90,out=90] (circ1.-45);
      \draw (alpha) to[in=-90,out=90] (circ1.225);
      \draw (circ1) to[in=-90,out=90] (circ2.-45);
      \draw (ev) to[in=-90,out=90] (circ2.225);
      \draw (circ2) to[in=-90,out=90] (circ3.-45);
      \draw (rho) to[in=-90,out=90] (circ3.225);
      \draw (circ3) to[in=-90,out=90] (circ4.225);
      \draw (f) to[in=-90,out=90] (circ4.-45);
      \draw (circ4) to (kappa);
      \draw (kappa) to (top);
    \end{tikzpicture}
    \underset{\eqref{eq:DefineLhd}}{=}
    \quad
    \begin{tikzpicture}[smallstring, baseline=30]
      \node (top) at (-1,5) {$\cA(b \to a \lhd u)$};
      \node (bot) at (3.1,-4) {$u$};
      \node[box] (phi) at (3.1,-3) {$\varphi_u$};
      \node[box] (kappa0) at (3.1,-2.05) {$\overline{\kappa_{u \to 1}}$};
      \node[box] (eta) at (3.1,-1.1) {$\overline{\eta_{b \lhd u, {u^*}}}$};
      \node[box] (f) at (0.9,-1.1) {$\overline{f \lhd \id_u \lhd \id_{{u^*}}}$};
      \node[box] (alpha) at (0,0) {$\overline{\alpha_{b, u, {u^*}}^{-1}}$};
      \node[box] (ev) at (-1,1) {$\overline{\id_b \lhd \ev_u}$};
      \node[box] (rho) at (-2,2) {$\overline{(\rho_{b}^\dagger)^{-1}}$};
      \node[box] (kappa) at (-1,4) {$\kappa_{b \to a \lhd u}$};
      \node[box] (circ0) at (2,0) {$- \circ_{\overline{\cA}} -$};
      \node[box] (circ1) at (1,1) {$- \circ_{\overline{\cA}} -$};
      \node[box] (circ2) at (0,2) {$- \circ_{\overline{\cA}} -$};
      \node[box] (circ3) at (-1,3) {$- \circ_{\overline{\cA}} -$};
      \draw (bot) to (phi);
      \draw (phi) to (kappa0);
      \draw (kappa0) to (eta);
      \draw (eta) to[in=-90,out=90] (circ0.-45);
      \draw (f) to[in=-90,out=90] (circ0.225);
      \draw (circ0) to[in=-90,out=90] (circ1.-45);
      \draw (alpha) to[in=-90,out=90] (circ1.225);
      \draw (circ1) to[in=-90,out=90] (circ2.-45);
      \draw (ev) to[in=-90,out=90] (circ2.225);
      \draw (circ2) to[in=-90,out=90] (circ3.-45);
      \draw (rho) to[in=-90,out=90] (circ3.225);
      \draw (circ3) to[in=-90,out=90] (kappa);
      \draw (kappa) to (top);
    \end{tikzpicture}
    \displaybreak[1]\\
    &=
    \begin{tikzpicture}[smallstring, baseline=15]
      \node (top) at (-0.8,5) {$\cA(b \to a \lhd u)$};
      \node (bot) at (2.9,-4) {$u$};
      \node[box] (phi) at (2.9,-3) {$\varphi_u$};
      \node[box] (kappa0) at (2.9,-2.1) {$\overline{\kappa_{u \to 1}}$};
      \node[box] (eta) at (2.9,-1.2) {$\overline{\eta_{a \lhd u, {u^*}}}$};
      \node[box] (alpha) at (1.1,-1.2) {$\overline{\alpha_{a, u, {u^*}}^{-1}}$};
      \node[box] (ev) at (0,0) {$\overline{\id_a \lhd \ev_u}$};
      \node[box] (rho) at (-0.8,1) {$\overline{(\rho_a^\dagger)^{-1}}$};
      \node[box] (f) at (-1.7,2) {$\overline{f}$};
      \node[box] (kappa) at (-0.8,4) {$\kappa_{b \to a \lhd u}$};
      \node[box] (circ0) at (2,0) {$- \circ_{\overline{\cA}} -$};
      \node[box] (circ1) at (1,1) {$- \circ_{\overline{\cA}} -$};
      \node[box] (circ2) at (0.1,2) {$- \circ_{\overline{\cA}} -$};
      \node[box] (circ3) at (-0.8,3) {$- \circ_{\overline{\cA}} -$};
      \draw (bot) to (phi);
      \draw (phi) to (kappa0);
      \draw (kappa0) to (eta);
      \draw (eta) to[in=-90,out=90] (circ0.-45);
      \draw (f) to[in=-90,out=90] (circ3.225);
      \draw (circ0) to[in=-90,out=90] (circ1.-45);
      \draw (alpha) to[in=-90,out=90] (circ0.225);
      \draw (circ1) to[in=-90,out=90] (circ2.-45);
      \draw (ev) to[in=-90,out=90] (circ1.225);
      \draw (circ2) to[in=-90,out=90] (circ3.-45);
      \draw (rho) to[in=-90,out=90] (circ2.225);
      \draw (circ3) to[in=-90,out=90] (kappa);
      \draw (kappa) to (top);
    \end{tikzpicture}
    \underset{\text{(Lem.~\ref{lem:eta-kappa-inv})}}{=}
    \quad
    \begin{tikzpicture}[smallstring, baseline=40]
      \node (top) at (1,5) {$\cA(b \to a \lhd u)$};
      \node (bot) at (1.8,0) {$u$};
      \node[box] (eta) at (1.8,1) {$\eta_{a, u}$};
      \node[box] (kappainv) at (1.8,2) {$\kappa_{a \to a \lhd u}^{-1}$};
      \node[box] (f) at (0.2,2) {$\overline{f}$};
      \node[box] (kappa) at (1,4) {$\kappa_{b \to a \lhd u}$};
      \node[box] (circ) at (1,3) {$- \circ_{\overline{\cA}} -$};
      \draw (bot) to (eta);
      \draw (eta) to (kappainv);
      \draw (kappainv) to[in=-90,out=90] (circ.-45);
      \draw (f) to[in=-90,out=90] (circ.225);
      \draw (circ) to (kappa);
      \draw (kappa) to (top);
    \end{tikzpicture}
    \quad=\quad
    \begin{tikzpicture}[smallstring, baseline=30]
      \node (top) at (0.5,3) {$\cA(b \to a \lhd u)$};
      \node (bot) at (1,0) {$u$};
      \node[box] (eta) at (1,1) {$\eta_{a, u}$};
      \node[box] (f) at (0,1) {$f^\dagger$};
      \node[box] (circ) at (0.5,2) {$- \circ -$};
      \draw (bot) to (eta);
      \draw (eta) to[in=-90,out=90] (circ.-45);
      \draw (f) to[in=-90,out=90] (circ.225);
      \draw (circ) to (top);
    \end{tikzpicture}
  \end{align*}
  which is exactly the mate of $f^\dagger \lhd \id_u$.
\end{proof}

\begin{prop}[\ref{eq:daglhd3}]
  \label{prop:daglhd3}
  Suppose now that $(\cV,\kappa^\cV)$ is a dagger $\sV$-category,
  i.e., $\kappa^\cV$ satisfies \ref{kappa:3} and \ref{kappa:4}.
  If $\cA$ is a dagger $\sV$-category, then the coherator $\alpha$ is unitary.
\end{prop}
\begin{proof}
  The mate of $\alpha_{a, u, v}^\dagger$, after applying \ref{kappa:2}, is given by
  
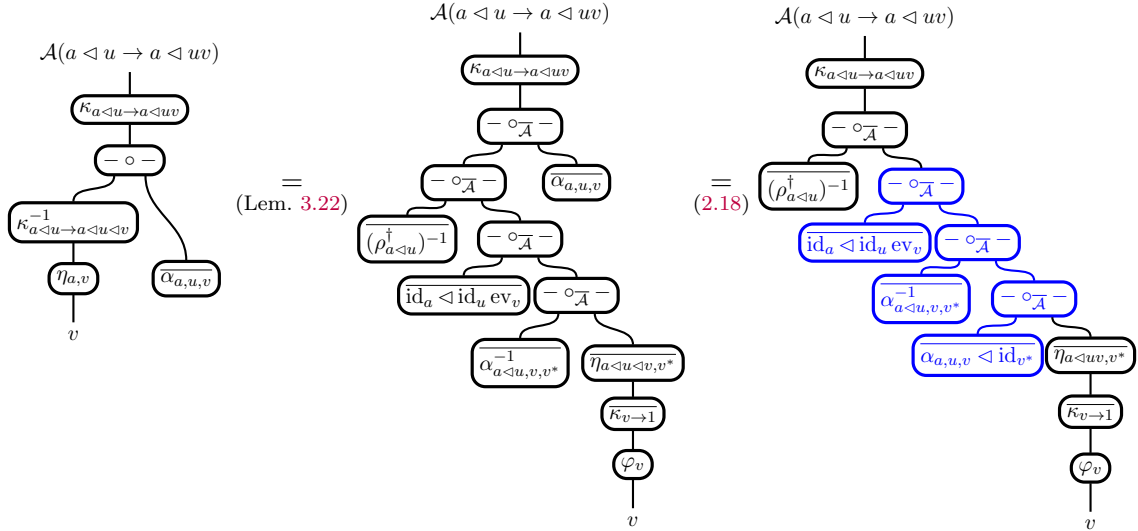
\begin{align*}
    \begin{tikzpicture}[smallstring, baseline=55]
      \node (top) at (1,5) {$\cA(a \lhd u \to a \lhd uv)$};
      \node (bot) at (0,0) {$v$};
      \node[box] (eta) at (0,1) {$\eta_{a,v}$};
      \node[box] (kappainv) at (0,2) {$\kappa_{a \lhd u \to a \lhd u \lhd v}^{-1}$};
      \node[box] (alpha) at (2,1) {$\overline{\alpha_{a,u,v}}$};
      \node[box] (circ) at (1,3.1) {$- \circ -$};
      \node[box] (kappa) at (1,4) {$\kappa_{a \lhd u \to a \lhd u v}$};
      \draw (bot) to (eta);
      \draw (eta) to (kappainv);
      \draw (kappainv) to[in=-90,out=90] (circ.225);
      \draw (alpha) to[in=-90,out=90] (circ.-45);
      \draw (circ) to (kappa);
      \draw (kappa) to (top);
    \end{tikzpicture}
    \underset{\text{(Lem.~\ref{lem:eta-kappa-inv})}}{=}
    \begin{tikzpicture}[smallstring, baseline=40]
      \node (top) at (1,5) {$\cA(a \lhd u \to a \lhd uv)$};
      \node (bot) at (3,-4) {$v$};
      \node[box] (kappa) at (1,4) {$\kappa_{a \lhd u \to a \lhd u v}$};
      \node[box] (circ4) at (1,3) {$- \circ_{\overline{\cA}} -$};
      \node[box] (alpha2) at (2,2) {$\overline{\alpha_{a,u,v}}$};
      \node[box] (circ3) at (0,2) {$- \circ_{\overline{\cA}} -$};
      \node[box] (rho) at (-1,1) {$\overline{(\rho_{a \lhd u}^{\dagger})^{-1}}$};
      \node[box] (circ2) at (1,1) {$- \circ_{\overline{\cA}} -$};
      \node[box] (ev) at (0,0) {$\overline{\id_a \lhd \id_u \ev_v}$};
      \node[box] (circ1) at (2,0) {$- \circ_{\overline{\cA}} -$};
      \node[box] (alpha1) at (1,-1.2) {$\overline{\alpha_{a \lhd u, v, {v^*}}^{-1}}$};
      \node[box] (eta) at (3,-1.2) {$\overline{\eta_{a \lhd u \lhd v , {v^*}}}$};
      \node[box] (kappa0) at (3,-2.1) {$\overline{\kappa_{v \to 1}}$};
      \node[box] (phi) at (3,-3) {$\varphi_v$};
      \draw (kappa) to (top);
      \draw (circ4) to (kappa);
      \draw (alpha2) to[in=-90,out=90] (circ4.-45);
      \draw (circ3) to[in=-90,out=90] (circ4.225);
      \draw (rho) to[in=-90,out=90] (circ3.225);
      \draw (circ2) to[in=-90,out=90] (circ3.-45);
      \draw (ev) to[in=-90,out=90] (circ2.225);
      \draw (circ1) to[in=-90,out=90] (circ2.-45);
      \draw (alpha1) to[in=-90,out=90] (circ1.225);
      \draw (eta) to[in=-90,out=90] (circ1.-45);
      \draw (kappa0) to (eta);
      \draw (phi) to (kappa0);
      \draw (bot) to (phi);
    \end{tikzpicture}
    \underset{\eqref{eq:DefineLhd}}{=}
    \begin{tikzpicture}[smallstring, baseline=20]
      \node (top) at (0,4) {$\cA(a \lhd u \to a \lhd uv)$};
      \node (bot) at (4,-5) {$v$};
      \node[box] (kappa) at (0,3) {$\kappa_{a \lhd u \to a \lhd u v}$};
      \node[box] (circ3) at (0,2) {$- \circ_{\overline{\cA}} -$};
      \node[box] (rho) at (-1,1) {$\overline{(\rho_{a \lhd u}^{\dagger})^{-1}}$};
      \node[box, blue] (circ2) at (1,1) {$- \circ_{\overline{\cA}} -$};
      \node[box, blue] (ev) at (0,0) {$\overline{\id_a \lhd \id_u \ev_v}$};
      \node[box, blue] (circ1) at (2,0) {$- \circ_{\overline{\cA}} -$};
      \node[box, blue] (alpha1) at (1,-1) {$\overline{\alpha_{a \lhd u, v, {v^*}}^{-1}}$};
      \node[box, blue] (circ) at (3,-1) {$- \circ_{\overline{\cA}} -$};
      \node[box, blue] (alpha) at (2,-2) {$\overline{\alpha_{a, u, v} \lhd \id_{{v^*}}}$};
      \node[box] (eta) at (4,-2) {$\overline{\eta_{a \lhd uv, {v^*}}}$};
      \node[box] (kappa0) at (4,-3) {$\overline{\kappa_{v \to 1}}$};
      \node[box] (phi) at (4,-4) {$\varphi_v$};
      \draw (kappa) to (top);
      \draw (circ3) to[in=-90,out=90] (kappa);
      \draw (rho) to[in=-90,out=90] (circ3.225);
      \draw (circ2) to[in=-90,out=90] (circ3.-45);
      \draw[blue] (ev) to[in=-90,out=90] (circ2.225);
      \draw[blue] (circ1) to[in=-90,out=90] (circ2.-45);
      \draw[blue] (alpha1) to[in=-90,out=90] (circ1.225);
      \draw[blue] (circ) to[in=-90,out=90] (circ1.-45);
      \draw[blue] (alpha) to[in=-90,out=90] (circ.225);
      \draw (eta) to[in=-90,out=90] (circ.-45);
      \draw (kappa0) to (eta);
      \draw (phi) to (kappa0);
      \draw (bot) to (phi);
    \end{tikzpicture}
  \end{align*}
  Observe that the blue part of the above diagram is again the bar of a morphism in $\sA$,
  so we can work in $\sA$:
  \begin{align*}
    (\alpha_{a, u, v} \lhd \id_{{v^*}}) &\circ \alpha_{a \lhd u, v, {v^*}}^{-1} \circ (\id_{a \lhd u} \lhd \ev_v)
    &&\\
    &= \alpha_{a, uv, {v^*}}^{-1} \circ \alpha_{a, u, v {v^*}} \circ (\id_{a \lhd u} \ev_v)
    && \text{$\alpha$ associative}\\
    &= \alpha_{a, uv, {v^*}}^{-1} \circ (\id_a \lhd \id_u \ev_v) \circ \rho_{a \lhd u}^{-1}.
    && \text{triangle identity}
  \end{align*}
  Replacing this back into the main calculation and using our assumption \ref{kappa:3} that $\rho$ is unitary,
  we recognize the mate of $\alpha_{a, u v, v^*}^{-1}$, which can be computed from Fact \ref{fact:MateOfLa}, and use that $(a \lhd -)$ is a $\sV$-functor:
  \begin{align*}
    \begin{tikzpicture}[smallstring, baseline=-20]
      \node (top) at (2,2) {$\cA(a \lhd u \to a \lhd uv)$};
      \node (bot) at (4,-5) {$v$};
      \node[box] (kappa) at (2,1) {$\kappa_{a \lhd u \to a \lhd u v}$};
      \node[box, blue] (circ1) at (2,0) {$- \circ_{\overline{\cA}} -$};
      \node[box, blue] (ev) at (1,-1) {$\overline{\id_a \lhd \id_u \ev_v}$};
      \node[box, blue] (circ) at (3,-1) {$- \circ_{\overline{\cA}} -$};
      \node[box, blue] (alpha) at (2,-2.2) {$\overline{\alpha_{a, uv, {v^*}}^{-1}}$};
      \node[box] (eta) at (4,-2.2) {$\overline{\eta_{a \lhd uv, {v^*}}}$};
      \node[box] (kappa0) at (4,-3.1) {$\overline{\kappa_{v \to 1}}$};
      \node[box] (phi) at (4,-4) {$\varphi_v$};
      \draw (kappa) to (top);
      \draw (circ1) to (kappa);
      \draw[blue] (ev) to[in=-90,out=90] (circ1.225);
      \draw[blue] (circ) to[in=-90,out=90] (circ1.-45);
      \draw[blue] (alpha) to[in=-90,out=90] (circ.225);
      \draw (eta) to[in=-90,out=90] (circ.-45);
      \draw (kappa0) to (eta);
      \draw (phi) to (kappa0);
      \draw (bot) to (phi);
    \end{tikzpicture}
    =
    \begin{tikzpicture}[smallstring, baseline=-30]
      \node (top) at (1.9,2) {$\cA(a \lhd u \to a \lhd uv)$};
      \node (bot) at (3,-5) {$v$};
      \node[box] (kappa) at (1.9,1.1) {$\kappa_{a \lhd u \to a \lhd u v}$};
      \node[box] (circ) at (1.9,0.2) {$- \circ_{\overline{\cA}} -$};
      \node[box] (ev) at (0.8,-3) {$\overline{\id_u \ev_v}$};
      \node[box] (a-tri) at (0.8,-2) {$\overline{(a \lhd -)_{u v v^* \to u}}$};
      \node[box] (L) at (3,-1) {$\overline{(a \lhd -)_{uv \to uvv^*}}$};
      \node[box] (eta) at (3,-2) {$\overline{\eta_{u v, v^*}^{\cV}}$};
      \node[box] (kappa0) at (3,-3) {$\overline{\kappa_{v \to 1}}$};
      \node[box] (phi) at (3,-4) {$\varphi_v$};
      \draw (kappa) to (top);
      \draw (circ) to (kappa);
      \draw (ev) to (a-tri);
      \draw (a-tri) to[in=-90,out=90] (circ.225);
      \draw (L) to[in=-90,out=90] (circ.-45);
      \draw (eta) to[in=-90,out=90] (L);
      \draw (kappa0) to (eta);
      \draw (phi) to (kappa0);
      \draw (bot) to (phi);
    \end{tikzpicture}
    =
    \begin{tikzpicture}[smallstring, baseline=-15]
      \node (top) at (1.9,3) {$\cA(a \lhd u \to a \lhd uv)$};
      \node (bot) at (3,-4) {$v$};
      \node[box] (kappa) at (1.9,2) {$\kappa_{a \lhd u \to a \lhd u v}$};
      \node[box] (a-tri) at (1.9,1.1) {$\overline{(a \lhd -)_{u v \to u}}$};
      \node[box] (circ) at (1.9,0.2) {$- \circ_{\overline{\cA}} -$};
      \node[box] (ev) at (0.8,-1) {$\overline{\id_u \ev_v}$};
      \node[box] (eta) at (3,-1) {$\overline{\eta_{u v, v^*}^{\cV}}$};
      \node[box] (kappa0) at (3,-2) {$\overline{\kappa_{v \to 1}}$};
      \node[box] (phi) at (3,-3) {$\varphi_v$};
      \draw (kappa) to (top);
      \draw (a-tri) to (kappa);
      \draw (circ) to (a-tri);
      \draw (ev) to[in=-90,out=90] (circ.225);
      \draw (eta) to[in=-90,out=90] (circ.-45);
      \draw (kappa0) to (eta);
      \draw (phi) to (kappa0);
      \draw (bot) to (phi);
    \end{tikzpicture}
    \underset{\text{(Lem.~\ref{lem:eta-kappa-inv})}}{=}
    \begin{tikzpicture}[smallstring, baseline=-30]
      \node (top) at (0,1) {$\cA(a \lhd u \to a \lhd uv)$};
      \node (bot) at (0,-4) {$v$};
      \node[box] (kappa) at (0,0) {$\kappa_{a \lhd u \to a \lhd u v}$};
      \node[box] (L) at (0,-1) {$\overline{(a \lhd -)_{uv \to u}}$};
      \node[box] (eta) at (0,-3) {$\eta_{u,v}^{\cV}$};
      \node[box] (kappa0) at (0,-2) {$(\kappa_{u \to uv}^{\cV})^{-1}$};
      \draw (kappa) to (top);
      \draw (L) to (kappa);
      \draw (kappa0) to (L);
      \draw (eta) to (kappa0);
      \draw (bot) to (eta);
    \end{tikzpicture},
  \end{align*}
  which, after applying \ref{kappa:4}, simplifies to exactly the mate of $\alpha_{a, u, v}^{-1}$.
\end{proof}

In Construction \ref{const:DaggerFromInvolutiveFunctor},
we saw that a weak dagger $\sV$-functor $\cF: \cA\to \cB$ gives a dagger functor $\sF: \sA \to \sB$.
In \cite[\S4.1]{1809.09782}, it was shown how to equip $\sF$ with the structure of a lax module functor
by defining a modulator $\omega_{a, u}$ to be the mate of
$\eta_{a,u} \circ \cF_{a\to a\lhd u}$ under Adjunction \ref{eq:CatToModAdjunction}.
In fact, when $\cA$ and $\cB$ are both dagger $\sV$-categories, $\omega_{a,u}$ is unitary by Proposition \ref{prop:modulator-from-V-functor} below.
In this case we define $\Phi(\cF) := \sF$.

\begin{prop}
  \label{prop:modulator-from-V-functor}
  Let $\cF : \cA \to \cB$ be a weak dagger $\sV$-functor between dagger $\sV$-categories.
  Then the underlying dagger functor $\sF: \sA \to \sB$ equipped with the modulator $\omega$ defined above is a dagger module functor, i.e., each $\omega_{a, u}$ is unitary.
\end{prop}
\begin{proof}
  When $\sU$ is a unitary monoidal category, $\sU$ is rigid,
  so a lax module functor is automatically a strong module functor by \cite[Lem.~2.10]{MR3934626}.
  Thus $\omega_{a, u}$ is invertible, so it suffices to check that $\omega_{a,u}^\dagger$ is its inverse on only one side.
  The mate of $\omega_{a, u} \circ \omega_{a, u}^\dagger$ is
  \begin{align*}
    \begin{tikzpicture}[smallstring, baseline=25]
      \node (bot) at (0,0) {$u$};
      \node[box] (eta) at (0,1) {$\eta_{a, u}^{\cA}$};
      \node[box] (F) at (0,2) {$\cF_{a \to a \lhd u}$};
      \node[box] (nu) at (2.5,1) {$\overline{\omega_{a, u}}$};
      \node[box] (kappa) at (2.5,2) {$\kappa_{\sF(a \lhd u) \to \sF(a) \lhd u}$};
      \node[box] (circ) at (1.25,3.1) {$- \circ -$};
      \node (top) at (1.25,4) {$\cB(\sF(a) \to \sF(a) \lhd u)$};
      \draw (bot) to (eta);
      \draw (eta) to (F);
      \draw (nu) to (kappa);
      \draw (F) to[in=-90,out=90] (circ.225);
      \draw (kappa) to[in=-90,out=90] (circ.-45);
      \draw (circ) to (top);
    \end{tikzpicture}
    &\underset{\text{\ref{kappa:2}}}{=}
    \begin{tikzpicture}[smallstring, baseline=45]
      \node (bot) at (0,0) {$u$};
      \node[box] (eta) at (0,1) {$\eta_{a, u}^{\cA}$};
      \node[box] (F) at (0,2) {$\cF_{a \to a \lhd u}$};
      \node[box] (kappainv) at (0,3) {$\kappa_{\sF(a) \to \sF(a \lhd u)}^{-1}$};
      \node[box] (nu) at (2.5,3) {$\overline{\omega_{a, u}}$};
      \node[box] (circ) at (1.25,4.2) {$- \circ -$};
      \node[box] (kappa) at (1.25,5.1) {$\kappa_{\sF(a) \to \sF(a) \lhd u}$};
      \node (top) at (1.25,6) {$\cB(\sF(a) \to \sF(a) \lhd u)$};
      \draw (bot) to (eta);
      \draw (eta) to (F);
      \draw (F) to (kappainv);
      \draw (kappainv) to[in=-90,out=90] (circ.225);
      \draw (nu) to[in=-90,out=90] (circ.-45);
      \draw (circ) to (kappa);
      \draw (kappa) to (top);
    \end{tikzpicture}
    =
    \begin{tikzpicture}[smallstring, baseline=45]
      \node (bot) at (0,0) {$u$};
      \node[box] (eta) at (0,1) {$\eta_{a, u}^{\cA}$};
      \node[box] (kappainv) at (0,2) {$\kappa_{a \to a \lhd u}^{-1}$};
      \node[box] (F) at (0,3) {$\overline{\cF_{a \lhd u \to a}}$};
      \node[box] (nu) at (2,3) {$\overline{\omega_{a, u}}$};
      \node[box] (circ) at (1,4.1) {$- \circ_{\overline{\cB}} -$};
      \node[box] (kappa) at (1,5.1) {$\kappa_{\sF(a) \to \sF(a) \lhd u}$};
      \node (top) at (1,6.1) {$\cB(\sF(a) \to \sF(a) \lhd u)$};
      \draw (bot) to (eta);
      \draw (eta) to (kappainv);
      \draw (kappainv) to (F);
      \draw (F) to[in=-90,out=90] (circ.225);
      \draw (nu) to[in=-90,out=90] (circ.-45);
      \draw (circ) to (kappa);
      \draw (kappa) to (top);
    \end{tikzpicture}
    \underset{\text{(Lem.~\ref{lem:eta-kappa-inv})}}{=}
    \begin{tikzpicture}[smallstring, baseline=25]
      \node (bot) at (3,-4) {$u$};
      \node[box] (phi) at (3,-3) {$\varphi_u$};
      \node[box] (kappa1) at (3,-2.2) {$\overline{\kappa_{u \to 1}}$};
      \node[box] (eta) at (3,-1.2) {$\overline{\eta_{a \lhd u, u^*}^{\cA}}$};
      \node[box] (alpha) at (1,-1.2) {$\overline{\alpha_{a, u, u^*}^{-1}}$};
      \node[box] (circ3) at (2,0) {$- \circ_{\overline{\cA}} -$};
      \node[box] (ev) at (0,0) {$\overline{\id_a \lhd \ev_u}$};
      \node[box] (circ2) at (1,1) {$- \circ_{\overline{\cA}} -$};
      \node[box] (rho) at (-1,1) {$\overline{(\rho_a^\dagger)^{-1}}$};
      \node[box] (circ1) at (0,2.1) {$- \circ_{\overline{\cA}} -$};
      \node[box] (F) at (0,3) {$\overline{\cF_{a \lhd u \to a}}$};
      \node[box] (nu) at (2.5,3) {$\overline{\omega_{a, u}}$};
      \node[box] (circ) at (1.25,4.1) {$- \circ_{\overline{\cB}} -$};
      \node[box] (kappa) at (1.25,5.1) {$\kappa_{\sF(a) \to \sF(a) \lhd u}$};
      \node (top) at (1.25,6.1) {$\cB(\sF(a) \to \sF(a) \lhd u)$};
      \draw (bot) to (phi);
      \draw (phi) to (kappa1);
      \draw (kappa1) to (eta);
      \draw (eta) to[in=-90,out=90] (circ3.-45);
      \draw (alpha) to[in=-90,out=90] (circ3.225);
      \draw (circ3) to[in=-90,out=90] (circ2.-45);
      \draw (ev) to[in=-90,out=90] (circ2.225);
      \draw (circ2) to[in=-90,out=90] (circ1.-45);
      \draw (rho) to[in=-90,out=90] (circ1.225);
      \draw (circ1) to (F);
      \draw (F) to[in=-90,out=90] (circ.225);
      \draw (nu) to[in=-90,out=90] (circ.-45);
      \draw (circ) to (kappa);
      \draw (kappa) to (top);
    \end{tikzpicture}
    \displaybreak[1]\\
    &=
    \begin{tikzpicture}[smallstring, baseline=10]
      \node (bot) at (2.25,-5) {$u$};
      \node[box] (phi) at (2.25,-4) {$\varphi_u$};
      \node[box] (kappa1) at (2.25,-3.1) {$\overline{\kappa_{u \to 1}}$};
      \node[box] (eta) at (2.25,-2.2) {$\overline{\eta_{a \lhd u, u^*}^{\cA}}$};
      \node[box] (F) at (2.25,-1.2) {$\overline{\cF_{a \lhd u \to a \lhd u \lhd u^*}}$};
      \node[box] (nu) at (4.05,-1.2) {$\overline{\omega_{a, u}}$};
      \node[box] (circ4) at (3.15,-0.1) {$- \circ_{\overline{\cB}} -$};
      \node[box] (alpha) at (0.85,-0.2) {$\overline{\sF(\alpha_{a, u, u^*}^{-1})}$};
      \node[box] (circ3) at (2,1) {$- \circ_{\overline{\cB}} -$};
      \node[box] (ev) at (0,1) {$\overline{\sF(\id_a \lhd \ev_u)}$};
      \node[box] (circ2) at (1,2) {$- \circ_{\overline{\cB}} -$};
      \node[box] (rho) at (-1,2) {$\overline{\sF((\rho_a^\dagger)^{-1})}$};
      \node[box] (circ1) at (0,3.1) {$- \circ_{\overline{\cB}} -$};
      \node[box] (kappa) at (0,4) {$\kappa_{\sF(a) \to \sF(a) \lhd u}$};
      \node (top) at (0,5) {$\cB(\sF(a) \to \sF(a) \lhd u)$};
      \draw (bot) to (phi);
      \draw (phi) to (kappa1);
      \draw (kappa1) to (eta);
      \draw (eta) to (F);
      \draw (F) to[in=-90,out=90] (circ4.225);
      \draw (nu) to[in=-90,out=90] (circ4.-45);
      \draw (circ4) to[in=-90,out=90] (circ3.-45);
      \draw (alpha) to[in=-90,out=90] (circ3.225);
      \draw (circ3) to[in=-90,out=90] (circ2.-45);
      \draw (ev) to[in=-90,out=90] (circ2.225);
      \draw (circ2) to[in=-90,out=90] (circ1.-45);
      \draw (rho) to[in=-90,out=90] (circ1.225);
      \draw (circ1) to[in=-90,out=90] (kappa);
      \draw (kappa) to (top);
    \end{tikzpicture}
    =
    \begin{tikzpicture}[smallstring, baseline=0]
      \node (bot) at (4.5,-4) {$u$};
      \node[box] (phi) at (4.5,-3.3) {$\varphi_u$};
      \node[box] (kappa1) at (4.5,-2.4) {$\overline{\kappa_{u \to 1}}$};
      \node[box] (eta) at (4.5,-1.5) {$\overline{\eta_{\sF(a) \lhd u, u^*}^{\cB}}$};
      \node[box, blue] (nu1) at (0.8,-2.5) {$\overline{\omega_{a \lhd u, u^*}}$};
      \node[box, blue] (nu2) at (2.8,-2.5) {$\overline{\omega_{a,u} \lhd \id_{u^*}}$};
      \node[box, blue] (circ5) at (1.8,-1.2) {$- \circ_{\overline{\cB}} -$};
      \node[box, blue] (circ4) at (3.15,-0.1) {$- \circ_{\overline{\cB}} -$};
      \node[box, blue] (alpha) at (0.85,-0.2) {$\overline{\sF(\alpha_{a, u, u^*}^{-1})}$};
      \node[box, blue] (circ3) at (2,1) {$- \circ_{\overline{\cB}} -$};
      \node[box, blue] (ev) at (0,1) {$\overline{\sF(\id_a \lhd \ev_u)}$};
      \node[box, blue] (circ2) at (1,2) {$- \circ_{\overline{\cB}} -$};
      \node[box, blue] (rho) at (-1,2) {$\overline{\sF((\rho_a^\dagger)^{-1})}$};
      \node[box, blue] (circ1) at (0,3.1) {$- \circ_{\overline{\cB}} -$};
      \node[box] (kappa) at (0,4) {$\kappa_{\sF(a) \to \sF(a) \lhd u}$};
      \node (top) at (0,5) {$\cB(\sF(a) \to \sF(a) \lhd u)$};
      \draw (bot) to (phi);
      \draw (phi) to (kappa1);
      \draw (kappa1) to (eta);
      \draw (eta) to[in=-90,out=90] (circ4.-45);
      \draw[blue] (nu2) to[in=-90,out=90] (circ5.-45);
      \draw[blue] (nu1) to[in=-90,out=90] (circ5.225);
      \draw[blue] (circ5) to[in=-90,out=90] (circ4.225);
      \draw[blue] (circ4) to[in=-90,out=90] (circ3.-45);
      \draw[blue] (alpha) to[in=-90,out=90] (circ3.225);
      \draw[blue] (circ3) to[in=-90,out=90] (circ2.-45);
      \draw[blue] (ev) to[in=-90,out=90] (circ2.225);
      \draw[blue] (circ2) to[in=-90,out=90] (circ1.-45);
      \draw[blue] (rho) to[in=-90,out=90] (circ1.225);
      \draw (circ1) to[in=-90,out=90] (kappa);
      \draw (kappa) to (top);
    \end{tikzpicture}\, ,
  \end{align*}
  where we noticed that
  \begin{align*}
    \begin{tikzpicture}[smallstring, baseline=40]
      \node (bot) at (0,0) {$u^*$};
      \node[box] (eta) at (0,1) {$\eta_{a \lhd u, u^*}^{\cA}$};
      \node[box] (F) at (0,2) {$\sF_{a \lhd u \to a \lhd u \lhd u^*}$};
      \node[box] (nu) at (-2,2) {$\omega_{a, u}$};
      \node[box] (circ) at (-1,3) {$- \circ -$};
      \node (top) at (-1,4) {$\cB(\sF(a) \lhd u \to \sF(a \lhd u \lhd u^*))$};
      \draw (bot) to (eta);
      \draw (eta) to (F);
      \draw (F) to[in=-90,out=90] (circ.-45);
      \draw (nu) to[in=-90,out=90] (circ.225);
      \draw (circ) to (top);
    \end{tikzpicture}
    =
    \begin{tikzpicture}[smallstring, baseline=40]
      \node (bot) at (-1,0) {$u^*$};
      \node[box] (circ0) at (0,2) {$- \circ -$};
      \node[box] (eta) at (-1,1) {$\eta_{\sF(a \lhd u), u^*}^{\cB}$};
      \node[box] (nu2) at (1,1) {$\omega_{a \lhd u, u^*}$};
      \node[box] (nu) at (-2,2) {$\omega_{a, u}$};
      \node[box] (circ) at (-1,3) {$- \circ -$};
      \node (top) at (-1,4) {$\cB(\sF(a) \lhd u \to \sF(a \lhd u \lhd u^*))$};
      \draw (bot) to (eta);
      \draw (eta) to[in=-90,out=90] (circ0.225);
      \draw (nu2) to[in=-90,out=90] (circ0.-45);
      \draw (circ0) to[in=-90,out=90] (circ.-45);
      \draw (nu) to[in=-90,out=90] (circ.225);
      \draw (circ) to (top);
    \end{tikzpicture}
    \underset{\eqref{eq:DefineLhd}}{=}
    \begin{tikzpicture}[smallstring, baseline=40]
      \node (bot) at (-2.5,0) {$u^*$};
      \node[box] (eta) at (-2.5,1.9) {$\eta_{\sF(a) \lhd u, u^*}^{\cB}$};
      \node[box] (circ) at (-1.25,3.2) {$- \circ -$};
      \node[box] (circ1) at (0,2) {$- \circ -$};
      \node[box] (nu1) at (-1,1) {$\omega_{a, u} \lhd \id_{u^*}$};
      \node[box] (nu2) at (1,1) {$\omega_{a \lhd u, u^*}$};
      \node (top) at (-1.25,4) {$\cB(\sF(a) \lhd u \to \sF(a \lhd u \lhd u^*))$};
      \draw (bot) to (eta);
      \draw (eta) to[in=-90,out=90] (circ.225);
      \draw (nu1) to[in=-90,out=90] (circ1.225);
      \draw (nu2) to[in=-90,out=90] (circ1.-45);
      \draw (circ1) to[in=-90,out=90] (circ.-45);
      \draw (circ) to[in=-90,out=90] (top);
    \end{tikzpicture}\, .
  \end{align*}
  Next, the blue portion is the bar of the following morphism in $\cB$:
  \begin{align*}
    (\omega_{a,u} \lhd \id_{u^*}) \circ &\omega_{a \lhd u, u^*} \circ \sF(\alpha_{a, u, u^*}^{-1} \circ (\id_a \lhd \ev_u) \circ (\rho_a^\dagger)^{-1})\\
    =& \alpha_{\sF(a), u, u^*}^{-1} \circ \omega_{a, u u^*} \circ \sF((\id_a \lhd \ev_u) \circ (\rho_a^\dagger)^{-1})
    && \text{$\omega$ associative}\\
    =& \alpha_{\sF(a), u, u^*}^{-1} \circ (\id_{\sF(a)} \lhd \ev_u) \circ \omega_{a, 1_{\sB}} \circ \sF((\rho_a^\dagger)^{-1})
    && \text{$\omega$ natural}\\
    =& \alpha_{\sF(a), u, u^*}^{-1} \circ (\id_{\sF(a)} \lhd \ev_u) \circ \omega_{a, 1_{\sB}} \circ \sF(\rho_a)
    &&\text{\ref{kappa:3}}\\
    =& \alpha_{\sF(a), u, u^*}^{-1} \circ (\id_{\sF(a)} \lhd \ev_u) \circ \rho_{\sF(a)}
    && \text{$\omega$ unital}\\
    =& \alpha_{\sF(a), u, u^*}^{-1} \circ (\id_{\sF(a)} \lhd \ev_u) \circ (\rho_{\sF(a)}^\dagger)^{-1}
    && \text{\ref{kappa:3}}
  \end{align*}
  We remark that the above manipulation of morphisms is similar to the proof that a lax module functor between modules for a rigid monoidal category is automatically strong \cite[Lem.~2.10]{MR3934626}.
  Substituting the bar of this expression back into the large diagram, we have
  \begin{align*}
    \begin{tikzpicture}[smallstring, baseline=25]
      \node (bot) at (3.15,-3) {$u$};
      \node[box] (phi) at (3.15,-2.1) {$\varphi_u$};
      \node[box] (kappa1) at (3.15,-1.2) {$\overline{\kappa_{u \to 1}}$};
      \node[box] (eta) at (3.15,-0.2) {$\overline{\eta_{\sF(a) \lhd u, u^*}^{\cB}}$};
      \node[box, blue] (alpha) at (0.85,-0.2) {$\overline{\alpha_{\sF(a), u, u^*}^{-1}}$};
      \node[box, blue] (circ3) at (2,1) {$- \circ_{\overline{\cB}} -$};
      \node[box, blue] (ev) at (0,1) {$\overline{\id_{\sF(a)} \lhd \ev_u}$};
      \node[box, blue] (circ2) at (1,2) {$- \circ_{\overline{\cB}} -$};
      \node[box, blue] (rho) at (-1,2) {$\overline{(\rho_{\sF(a)}^\dagger)^{-1}}$};
      \node[box, blue] (circ1) at (0,3.2) {$- \circ_{\overline{\cB}} -$};
      \node[box] (kappa) at (0,4.1) {$\kappa_{\sF(a) \to \sF(a) \lhd u}$};
      \node (top) at (0,5) {$\cB(\sF(a) \to \sF(a) \lhd u)$};
      \draw (bot) to (phi);
      \draw (phi) to (kappa1);
      \draw (kappa1) to (eta);
      \draw (eta) to[in=-90,out=90] (circ3.-45);
      \draw[blue] (alpha) to[in=-90,out=90] (circ3.225);
      \draw[blue] (circ3) to[in=-90,out=90] (circ2.-45);
      \draw[blue] (ev) to[in=-90,out=90] (circ2.225);
      \draw[blue] (circ2) to[in=-90,out=90] (circ1.-45);
      \draw[blue] (rho) to[in=-90,out=90] (circ1.225);
      \draw (circ1) to[in=-90,out=90] (kappa);
      \draw (kappa) to (top);
    \end{tikzpicture}
    \underset{\text{(Lem.~\ref{lem:eta-kappa-inv})}}{=}
    \begin{tikzpicture}[smallstring, baseline=40]
      \node (bot) at (0,0) {$u$};
      \node[box] (eta) at (0,1) {$\eta_{\sF(a), u}$};
      \node[box] (kappainv) at (0,2) {$\kappa_{\sF(a) \to \sF(a) \lhd u}^{-1}$};
      \node[box] (kappa) at (0,3) {$\kappa_{\sF(a) \to \sF(a) \lhd u}$};
      \node (top) at (0,4) {$\cB(\sF(a) \to \sF(a) \lhd u)$};
      \draw (bot) to (eta);
      \draw (eta) to (kappainv);
      \draw (kappainv) to (kappa);
      \draw (kappa) to (top);
    \end{tikzpicture}\, ,
  \end{align*}
  which is the mate of $\id_{\sF(a) \lhd u}$, and thus $\omega_{a, u}$ is unitary.
\end{proof}

%%%%%%%%%%%%%%%%%%%%%%%%%%%%%%%%%%%%%%%%%%%%
\begin{construction}
  \label{const:cat to mod 2-cells}

  Given a $\sV$-natural transformation $\theta : \cF \Rightarrow \cG$,
  we get a natural transformation $\Theta : \sF \Rightarrow \sG$ on the underlying functors by taking
  $\Theta_a := \theta_a \in \sB(\sF(a) \to \sG(a)) = \sV(1_{\sV} \to \cB(\cF(a) \to \cG(a)))$
  \cite[\S1.3]{MR2177301}.
  We check that $\Theta$ is a module natural transformation by taking the mate of
  $(\Theta_a \lhd \id_u) \circ \omega_{a, u}^\sG$:
  \begin{align*}
    \begin{tikzpicture}[smallstring, baseline=45]
      \node (bot) at (0,0) {$u$};
      \node[box] (eta) at (0,1) {$\eta_{\sF(a), u}$};
      \node[box] (theta) at (2,1) {$\theta_a \lhd \id_u$};
      \node[box] (circ1) at (1,2) {$- \circ -$};
      \node[box] (nu) at (2.5,2) {$\omega_{a, u}^\sG$};
      \node[box] (circ2) at (1.75,3) {$- \circ -$};
      \node (top) at (1.75,4) {$\sG(a \lhd u)$};
      \draw (bot) to (eta);
      \draw (eta) to[in=-90,out=90] (circ1.225);
      \draw (theta) to[in=-90,out=90] (circ1.-45);
      \draw (circ1) to[in=-90,out=90] (circ2.225);
      \draw (nu) to[in=-90,out=90] (circ2.-45);
      \draw (circ2) to (top);
    \end{tikzpicture}
    \underset{\eqref{eq:DefineLhd}}{=}
    \begin{tikzpicture}[smallstring, baseline=45]
      \node (bot) at (2,0) {$u$};
      \node[box] (theta) at (0,1) {$\theta_a$};
      \node[box] (eta) at (2,1) {$\eta_{\sG(a), u}$};
      \node[box] (circ1) at (1,2) {$- \circ -$};
      \node[box] (nu) at (2.5,2) {$\omega_{a, u}^\sG$};
      \node[box] (circ2) at (1.75,3) {$- \circ -$};
      \node (top) at (1.75,4) {$\sG(a \lhd u)$};
      \draw (bot) to (eta);
      \draw (eta) to[in=-90,out=90] (circ1.-45);
      \draw (theta) to[in=-90,out=90] (circ1.225);
      \draw (circ1) to[in=-90,out=90] (circ2.225);
      \draw (nu) to[in=-90,out=90] (circ2.-45);
      \draw (circ2) to (top);
    \end{tikzpicture}
    =
    \begin{tikzpicture}[smallstring, baseline=45]
      \node (bot) at (2,0) {$u$};
      \node[box] (eta) at (2,1) {$\eta_{\sG(a), u}$};
      \node[box] (G) at (2,2) {$\sG_{a \to a \lhd u}$};
      \node[box] (theta) at (0,2) {$\theta_a$};
      \node[box] (circ) at (1,3) {$- \circ -$};
      \node (top) at (1,4) {$\sG(a \lhd u)$};
      \draw (bot) to (eta);
      \draw (eta) to (G);
      \draw (G) to[in=-90,out=90] (circ.-45);
      \draw (theta) to[in=-90,out=90] (circ.225);
      \draw (circ) to (top);
    \end{tikzpicture}
    =
    \begin{tikzpicture}[smallstring, baseline=45]
      \node (bot) at (0,0) {$u$};
      \node[box] (theta) at (2,2) {$\theta_{a \lhd u}$};
      \node[box] (eta) at (0,1) {$\eta_{\sG(a), u}$};
      \node[box] (F) at (0,2) {$\sF_{a \to a \lhd u}$};
      \node[box] (circ) at (1,3) {$- \circ -$};
      \node (top) at (1,4) {$\sG(a \lhd u)$};
      \draw (bot) to (eta);
      \draw (eta) to (F);
      \draw (F) to[in=-90,out=90] (circ.225);
      \draw (theta) to[in=-90,out=90] (circ.-45);
      \draw (circ) to (top);
    \end{tikzpicture}\, ,
  \end{align*}
  which is the mate of $\omega_{a, u}^\sF \circ \Theta_{a \lhd u}$.
  The last two equalities used the definition of $\omega_{a, u}^\sG$ and $\sV$-naturality of $\theta$, respectively.
  Thus we can now define $\Phi$ on 2-cells via $\Phi(\theta) := \Theta$.
\end{construction}

\begin{construction}
  \label{const:2-functor}
  Let $\sV$ be a rigid involutive monoidal category equipped with an enriched dagger structure $\kappa^\sV$.
  Define $\Phi: \dagvcat \to \dagvmod$ on 0-cells as in the start of the subsection,
  on 1-cells as in Proposition \ref{prop:modulator-from-V-functor},
  and on 2-cells as in Construction \ref{const:cat to mod 2-cells}.
  To check that $\Phi$ is a (strict) 2-functor, all steps are straightforward except that composition of 1-cells is preserved.

  The composite modulator $\omega_{a, v}^{\Phi(\cF \circ \cG)}$ is the mate of
  $\eta_{a, v} \circ (\cF_{a \to a \lhd v} \circ \cG_{\sF(a) \to \sF(a \lhd v)})$
  under Adjunction \ref{eq:ModToCatAdjunction}, which is given by
  \begin{align*}
    \begin{tikzpicture}[smallstring, baseline=45]
      \node(a) at (0,0) {$\sG(\sF(a))$};
      \node (v) at (2.5,0) {$v$};
      \node (av) at (0,4) {$\sG(\sF(a \lhd v))$};
      \draw[rounded corners] (1.5,0.5) rectangle (3.5,1.25) node[midway] {$\eta_{a, v}$};
      \draw[rounded corners] (1.5,1.5) rectangle (3.5,2.25) node[midway] {$\cF_{a \to a \lhd v}$};
      \modulebox{0,2.5}{3.5}{0.75}{$\mate(\cG_{\sF(a) \to \sF(a \lhd v)})$};
      \draw[blue] (a) to (av);
      \draw (v) to (v |- 0,0.5);
      \draw (v |- 0,1.25) to (v |- 0,1.5);
      \draw (v |- 0,2.25) to (v |- 0,2.5);
    \end{tikzpicture}
    \underset{\text{Lem.~\ref{lem:mate of V-functor}}}{=}
    \hspace{-20pt}
    \begin{tikzpicture}[smallstring, baseline=45]
      \node(a) at (0,0) {$\sG(\sF(a))$};
      \node (v) at (2.5,0) {$v$};
      \node (av) at (0,5) {$\sG(\sF(a \lhd v))$};
      \draw[rounded corners] (1.5,0.5) rectangle (3.5,1.25) node[midway] {$\eta_{a, v}$};
      \draw[rounded corners] (1.5,1.5) rectangle (3.5,2.25) node[midway] {$\cF_{a \to a \lhd v}$};
      \modulebox{0,2.5}{3.5}{0.75}{$\omega^\sG_{\sF(a) , \cB(\sF(a) \to \sF(a \lhd v))}$};
      \modulebox{0,3.5}{3.5}{0.75}{$\sG(\epsilon_{\sF(a) \to \sF(a \lhd v)})$};
      \draw[blue] (a) to (av);
      \draw (v) to (v |- 0,0.5);
      \draw (v |- 0,1.25) to (v |- 0,1.5);
      \draw (v |- 0,2.25) to (v |- 0,2.5);
      \draw (0.1,3.25) to (0.1,3.5);
    \end{tikzpicture}
    =
    \hspace{-13pt}
    \begin{tikzpicture}[smallstring, baseline=45]
      \node(a) at (0,0) {$\sG(\sF(a))$};
      \node (v) at (2.5,0) {$v$};
      \node (av) at (0,5) {$\sG(\sF(a \lhd v))$};
      \modulebox{0,0.5}{3.5}{0.75}{$\omega^\sG_{\sF(a), v}$};
      \modulebox{0,1.5}{3.5}{0.75}{$\sG(\id_{\sF(a)} \lhd \eta_{a, v})$};
      \modulebox{0,2.5}{3.5}{0.75}{$\sG(\id_{\sF(a)} \lhd \cF_{a \to a \lhd v})$};
      \modulebox{0,3.5}{3.5}{0.75}{$\sG(\epsilon_{\sF(a) \to \sF(a \lhd v)})$};
      \draw[blue] (a) to (av);
      \draw (v) to (v |- 0,0.5);
      \draw (0.1,1.25) to (0.1,1.5);
      \draw (0.1,2.25) to (0.1,2.5);
      \draw (0.1,3.25) to (0.1,3.5);
    \end{tikzpicture}
    \underset{\text{Lem.~\ref{lem:mate of V-functor}}}{=}
    \hspace{-20pt}
    \begin{tikzpicture}[smallstring, baseline=45]
      \node(a) at (0,0) {$\sG(\sF(a))$};
      \node (v) at (2.5,0) {$v$};
      \node (av) at (0,5) {$\sG(\sF(a \lhd v))$};
      \modulebox{0,0.5}{3.5}{0.75}{$\omega^\sG_{\sF(a), v}$};
      \modulebox{0,1.5}{3.5}{0.75}{$\sG(\sF(\id_a) \lhd \eta_{a, v})$};
      \modulebox{0,2.5}{3.5}{0.75}{$\sG(\omega^\sF_{a, \cA(a \to a \lhd v)})$};
      \modulebox{0,3.5}{3.5}{0.75}{$\sG(\sF(\epsilon_{a \to a \lhd v}))$};
      \draw[blue] (a) to (av);
      \draw (v) to (v |- 0,0.5);
      \draw (0.1,1.25) to (0.1,1.5);
      \draw (0.1,2.25) to (0.1,2.5);
      \draw (0.1,3.25) to (0.1,3.5);
    \end{tikzpicture}\, ,
  \end{align*}
  where we made use of Lemma \ref{lem:mate of V-functor} below.
  After using $\sV$-module-naturality of $\omega^\sF$ and that
  $(\id_a \lhd \eta_{a, v}) \circ \epsilon_{a \to a \lhd v} = \id_{a \lhd v}$,
  this is exactly equal to $\omega^{\Phi(\cF \circ \cG)} := \omega^\sG \circ \sG(\omega^\sF)$.
  Note that this proof is similar to (and much simpler than) the proof of \cite[Prop.~5.8]{2104.07747}.
\end{construction}

\begin{lem}
  \label{lem:mate of V-functor}
  Given a $\sV$-functor $\cF : \cA \to \cB$, the mate of $\cF_{a \to b}$ under Adjunction \ref{eq:CatToModAdjunction} is
  \begin{equation}
    \begin{tikzpicture}[smallstring,baseline=30]
      \node (bot) at (0,0) {$\sF(a)$};
      \node (top) at (0,3) {$\sF(b)$};
      \node (ab) at (1.5,0) {$\cA(a \to b)$};
      \modulebox{0,0.5}{2}{0.75}{$\omega_{a,\cA(a \to b)}$};
      \modulebox{0,1.5}{2}{0.75}{$\sF(\epsilon_{a \to b})$};
      \draw[blue] (bot) to (top);
      \draw (ab) to (1.5,0.5);
      \draw (0.1,1.25) to (0.1,1.5);
    \end{tikzpicture}
  \end{equation}
  with $\omega$ as defined in the discussion before Proposition \ref{prop:modulator-from-V-functor}.
\end{lem}
\begin{proof}
  The result follows immediately from taking the mate of the given expression and using functoriality of $\cF$.
\end{proof}

%%%%%%%%%%%%%%%%%%%%%%%%%%%%%%%%%%%%%%%%%%%%
%%%%%%%%%%%%%%%%%%%%%%%%%%%%%%%%%%%%%%%%%%%%
%%%%%%%%%%%%%%%%%%%%%%%%%%%%%%%%%%%%%%%%%%%%
\section{Equivalence}
\label{sec:equivalence} 

In this section, we prove that when our enriching involutive monoidal category $\sV$ is replaced with a unitary monoidal category $\sU$, the 2-functor $\Phi : \dagucat \to \dagumod$ is a 2-equivalence.
As in Example \ref{ex:DaggerSelfEnrichment}, we denote the dagger self-enrichment by $\cU$.
Since $\sU$ is equipped with a unitary dual functor $*$ and the induced involutive structure $\overline{\,\cdot\,}$, we may identity $\overline{u}$ with $u^*$ and $\kappa_{u \to 1} = \id_{\overline{u}}$.

To show $\Phi$ is a 2-equivalence, we need constructions in the reverse direction to show that $\Phi$ is essentially surjective on 0-cells, essentially surjective on 1-cells, and fully faithful on 2-cells.
We first give constructions in the reverse direction to $\Phi$, and then use them to check the corresponding properties for $\Phi$.

\subsection{The reverse constructions}
Let $\sU^\natural$ denote the involutive rigid linear monoidal category obtained from $\sU$ by forgetting the dagger structure $\dag$, but remembering the involutive structure $\overline{\,\cdot\,}$.

\begin{construction}
  \label{constr:U-mod-to-dagger-U-cat}
  Let $\sA$ be a 0-cell in $\dagumod$, i.e., a tensored dagger $\sU$-module category.
  Recall from \cite{1809.09782} that we get a $\sU^{\natural}$-enriched category $\cA$ with the same objects as $\sA$ and hom objects $\cA(a\to b)$ determined up to unique isomorphism by the Yoneda Lemma via the adjunction 
  \begin{equation}
    \label{eq:ModToCatAdjunction}
    \sA(a\lhd u \to b)
    \cong
    \sU^{\natural}(u \to \cA(a\to b)).
  \end{equation}
  Since $\cA$ is enriched in $\sU^{\natural}$, the hom objects $\cA(a\to b)$ are only determined up to isomorphism, and not unitary isomorphism in $\sU$.

  As explained in Fact \ref{fact:MateOfLa} above, since $\sA$ is a strong right $\sU^{\natural}$-module, we get a $\sU^{\natural}$-functor
  $a\lhd - : \cU \to \cA$ by defining $(a\lhd -)_{u \to v}$ as the mate of
  $\alpha_{a, u, \cU(u \to v)}^{-1} \circ (\id_a \lhd \epsilon_{u \to v}^{\cU})$
  under Adjunction \eqref{eq:CatToModAdjunction}.

  To get a dagger $\sU$-category from $\cA$, we define
  $\kappa^\cA_{a\to b}: \overline{\cA(b\to a)} \to \cA(a\to b)$ via
  \begin{equation}
    \label{eq:MateOfKappa}
    \mate (\kappa^\cA_{a \to b}) :=
    \begin{tikzpicture}[smallstring, baseline=45]
      \node (a) at (0,0) {$a$};
      \node (ba) at (2.5,0) {$\overline{\cA(b \to a)}$};
      \node (b) at (0,5.5) {$b$};
      \modulebox{0,0.75}{1}{0.75}{$\varepsilon_{b \to a}^\dagger$}
      \modulebox{0,2}{3}{1}{$\alpha_{a, \cA(b \to a), \overline{\cA(b \to a)}}^{-1}$}
      \modulebox{0,4}{3}{0.5}{$\rho_b$}
      \draw[blue] (a) to (b);
      \draw (ba) to (2.5,2);
      \draw (0.75,1.5) to (0.75,2);
      \draw (1.5,3) to[in=90,out=90] node[above] (cap) {} (2.5,3);
      \draw[dotted] (cap |- cap) to (cap |- 0,4);
    \end{tikzpicture}
    \,.
  \end{equation}
  In the remainder of this subsection, we will suppress the unitary coherators $\rho,\alpha$ whenever possible.
\end{construction}

\begin{prop}
  The pair $(\cA, \kappa^{\cA})$ from Construction \ref{constr:U-mod-to-dagger-U-cat} is a dagger $\sU$-category.
\end{prop}
\begin{proof}
  To show that $(\cA, \kappa)$ is a dagger $\sU$-category, we need to check \ref{kappa:1}-\ref{kappa:4}.
  Note that since we assumed that the unitors $\rho$ were unitary, \ref{kappa:3} is automatic.
  The remaining three requirements are checked in Lemmas \ref{lem:kappa-1}, \ref{lem:kappa-2}, and \ref{lem:kappa-4} below.
\end{proof}
\begin{lem}\label{lem:kappa-1}
  The $\kappa^\cA$ defined in \eqref{eq:MateOfKappa} satisfies \ref{kappa:1}.
\end{lem}
\begin{proof}
  We take mates:
  \begin{align*}
    \mate \left(
\, .
  \end{align*}
  The dagger of this morphism is
  \begin{align*}
    %
\, .
  \end{align*}
  Taking dagger of this gives the mate of $\varphi_{\cA(a \to b)}^\dagger$, which proves \ref{kappa:1}.
\end{proof}
\begin{lem}\label{lem:kappa-2}
  The $\kappa^\cA$ defined in \eqref{eq:MateOfKappa} satisfies \ref{kappa:2}.
\end{lem}
\begin{proof}
  Taking the mate of $(- \circ_{\overline{\cA}} -) \circ \kappa_{a \to c}$, we find the mate of $- \circ -$:
  \begin{align*}
    %
,
  \end{align*}
  which is exactly the mate of $(\kappa_{a \to b} \kappa_{b \to c}) \circ (- \circ_\cA -)$.
\end{proof}

\begin{lem} \label{lem:kappa-4}
  The $\sU$-functors $a \lhd - : \cU \to \cA$ are dagger, i.e., \ref{kappa:4} holds.
\end{lem}
\begin{proof}
  Taking the mate of $\kappa_{u \to v}^{\cU} \circ (a \lhd -)_{u \to v}$,
  we use naturality of $\alpha$ and the definitions of $\kappa^{\cU}$, $\nu$, and $\varphi$ to get
  \begin{align*}
    %
\, .
  \end{align*}
  On the other side, after taking the mate of
  $\overline{(a \lhd -)_{v \to u}} \circ \kappa_{a \lhd u \to a \lhd v}^\cA$
  via \eqref{eq:MateOfKappa}, we have
  \[
  %
\, .
  \]
  Now we can use associativity of $\alpha$ and simplify to see that this is exactly the same as the expression we have for the mate of $\kappa_{u \to v}^{\cU} \circ (a \lhd -)_{u \to v}$.
\end{proof}

\begin{construction}
  \label{const:mod-to-cat-1-cell}
  Given dagger $\sU$-categories $\cA$ and $\cB$ and a dagger module functor $(\sF,\omega) : \Phi(\cA) \to \Phi(\cB)$, we construct a $\sU$-functor $\cF : \cA \to \cB$ via $\cF(a) := \sF(a)$ on objects and
  \begin{equation} \label{eqn:U-functor-construction}
    \cF_{a \to b} := \mate \left(
    %
    \right)
  \end{equation}
  under $\sB(\sF(a) \lhd \cA(a \to b) \to \sF(b)) \cong \sU(\cA(a \to b) \to \cB(\cF(a) \to \cF(b)))$.
\end{construction}

\begin{prop}
  If $\cA, \cB \in \vcat$ are 0-cells and $(\sF, \omega)$ is a 1-cell in $\vmod(\Phi(\cA) \to \Phi(\cB))$, then $\cF$ as defined in Construction \ref{const:mod-to-cat-1-cell} is a $\sU$-functor.
\end{prop}
\begin{proof}
  It is straightforward to check that $\cF$ preserves identities, i.e., $j_a^\cA \circ \cF_{a \to a} = j_{\cF(a)}^\cB$,
  so we only check that $\cF$ is functorial:
  \begin{align*}
    %
\, .
  \end{align*}
  Writing $u = \cA(a \to b)$ and $v = \cA(b \to c)$, the mate of the left hand side under Adjunction \ref{eq:CatToModAdjunction} is, using Remark \ref{rem:circ-mate-calculations},
  \begin{align*}
    %
    \underset{\text{(Rem.~\ref{rem:mate of circ})}}{=}
    %
\, .
  \end{align*}
  This is exactly the mate of the right hand side under the same adjunction via Remark \ref{rem:mate of circ}.
\end{proof}

To show that $\cF$ is also dagger, we first prove the following lemma.

\begin{lem} 
  \label{lem:kappa-mate}
  Suppose $\sU$ is a unitary monoidal category and $\cA$ is a dagger $\sU$-category.
  The mate of $\kappa_{a \to b} \in \sU(\overline{\cA(b \to a)} \to \cA(a \to b))$ under Adjunction \eqref{eq:CatToModAdjunction} is given by
  \begin{align*}
    \mate (\kappa_{a \to b}) :=
    %
  \end{align*}
\end{lem}

\begin{proof}
  We show that the mate of the right hand side is equal to $\kappa_{a \to b}$. Setting $u= \cA(a \to b)$ for readability, the mate is
  \begin{align*}
    %
  \end{align*}
  where we used that $f \circ \kappa^{-1} = \overline{f^\dagger}$, that $\alpha$ is unitary, and that $- \lhd -$ is a dagger functor. 
  Next, as in the previous proofs of \ref{eq:daglhd1}, \ref{eq:daglhd2}, and \ref{eq:daglhd3},
  we apply Lemma \ref{lem:eta-kappa-inv} to replace the bottom $\eta \circ \kappa^{-1}$, and we apply the bar of \eqref{eq:DefineLhd} to move the blue morphisms on the top right to the left hand side, which yields the bar of a morphism in the underlying category $\sA$ sandwiched between a $\varphi\circ \overline{\eta}$ 
  and a $\kappa$.
  This morphism in $\sA$ is given by 
  {\tiny{
      \begin{align*}
        (\rho_b^{-1} &\lhd \id_{\overline{\overline{u}}})
        \circ
        (\id_b \lhd \ev_u^\dagger \lhd \id_{\overline{\overline{u}}})
        \circ
        (\alpha_{b, u, \overline{u}} \lhd \id_{\overline{\overline{u}}})
        \circ
        (\epsilon_{b \to a} \lhd \id_{\overline{u}} \lhd \id_{\overline{\overline{u}}})
        \circ
        \alpha_{a, \overline{u}, \overline{\overline{u}}}^{-1}
        \circ
        (\id_a \lhd \ev_{\overline{u}})
        \circ
        \rho_a\\
        &= (\rho_b^{-1} \lhd \id_{\overline{\overline{u}}})
        \circ
        (\id_b \lhd \ev_u^\dagger \lhd \id_{\overline{\overline{u}}})
        \circ
        (\alpha_{b, u, \overline{u}} \lhd \id_{\overline{\overline{u}}})
        \circ
        \alpha_{b \lhd u, \overline{u}, \overline{\overline{u}}}^{-1}
        \circ
        (\epsilon_{b \to a} \lhd \ev_{\overline{u}})
        \circ
        \rho_a
        &&\text{($\alpha$ natural)}\\
        &= (\rho_b^{-1} \lhd \id_{\overline{\overline{u}}})
        \circ
        (\id_b \lhd \ev_u^\dagger \lhd \id_{\overline{\overline{u}}})
        \circ
        (\alpha_{b, u \overline{u}, \overline{\overline{u}}}^{-1})
        \circ
        \alpha_{b, u, \overline{u} \, \overline{\overline{u}}}
        \circ
        (\epsilon_{b \to a} \lhd \ev_{\overline{u}})
        \circ
        \rho_a
        &&\text{($\alpha$ associative)}\\
        &= (\rho_b^{-1} \lhd \id_{\overline{\overline{u}}})
        \circ
        (\alpha_{b, 1_\sV, \overline{\overline{u}}}^{-1})
        \circ
        (\id_b \lhd \ev_u^\dagger \lhd \id_{\overline{\overline{u}}})
        \circ
        (\id_b \lhd \id_u \lhd \ev_u)
        \circ
        \alpha_{b, u,1_\sV}
        \circ
        (\epsilon_{b \to a} \lhd \id_{1_{\sV}})
        \circ
        \rho_a
        &&\text{($\alpha$ natural)}\\
        &= (\rho_b^{-1} \lhd \id_{\overline{\overline{u}}})
        \circ
        (\alpha_{b, 1_\sV, \overline{\overline{u}}}^{-1})
        \circ
        (\id_b \lhd \varphi_u^{-1})
        \circ
        \rho_{b \lhd u}^{-1}
        \circ
        (\epsilon_{b \to a} \lhd \id_{1_{\sV}})
        \circ
        \rho_a
        &&\text{(definition of $\varphi$)}\\
        &= (\rho_b^{-1} \lhd \id_{\overline{\overline{u}}})
        \circ
        (\rho_b \lhd \id_{\overline{\overline{u}}})
        \circ
        (\id_b \lhd \varphi_u^{-1})
        \circ
        \rho_{b \lhd u}^{-1}
        \circ
        (\epsilon_{b \to a} \lhd \id_{1_{\sV}})
        \circ
        \rho_a
        &&\text{(triangle identity)}\\
        &= (\rho_b^{-1} \lhd \id_{\overline{\overline{u}}})
        \circ
        (\rho_b \lhd \id_{\overline{\overline{u}}})
        \circ
        (\id_b \lhd \varphi_u^{-1})
        \circ
        \epsilon_{b \to a}
        \circ
        \rho_a^{-1}
        \circ
        \rho_a
        &&\text{($\rho$ natural)}\\
        &= (\id_b \lhd \varphi_u^{-1})
        \circ
        \epsilon_{b \to a}
        &&\text{(simplify)}
      \end{align*}
  }}
  Note that the steps above depend on unitarity of $\rho$. 
  Substituting the bar of this morphism back between the $\varphi\circ \overline{\eta}$ and $\kappa$ above, we obtain
  \begin{align*}
    \begin{tikzpicture}[smallstring, baseline=40]
      \node (top) at (-2,5) {$u$};
      \node (bot) at (0,-1) {$\overline{u}$};
      \node[box] (phi) at (0,-0.1) {$\varphi_{u}$};
      \node[box] (eta) at (0,0.8) {$\overline{\eta_{a, \overline{\overline{u}}}}$};
      \node[box] (phi2) at (-2,0.8) {$\overline{\id_a \lhd \varphi_{u}^{-1}}$};
      \node[box] (epsilon) at (-3,2) {$\overline{\varepsilon_{b \to a}}$};
      \node[box] (kappa) at (-2,4) {$\kappa_{a \to b}$};
      \node[box] (circ1) at (-1,2) {$- \circ_{\overline{\cA}} -$};
      \node[box] (circ2) at (-2,3) {$- \circ_{\overline{\cA}} -$};
      \draw (bot) to (phi);
      \draw (phi) to (eta);
      \draw (eta) to[in=-90,out=90] (circ1.-45);
      \draw (phi2) to[in=-90,out=90] (circ1.225);
      \draw (circ1) to[in=-90,out=90] (circ2.-45);
      \draw (epsilon) to[in=-90,out=90] (circ2.225);
      \draw (circ2) to (kappa);
      \draw (kappa) to (top);
    \end{tikzpicture}
    \underset{\eqref{eq:MateOf1lhdg}}{=}
    \begin{tikzpicture}[smallstring, baseline=40]
      \node (top) at (-1,4) {$u$};
      \node (bot) at (0,0) {$\overline{u}$};
      \node[box] (eta) at (0,1) {$\overline{\eta_{a, u}}$};
      \node[box] (epsilon) at (-2,1) {$\overline{\varepsilon_{b \to a}}$};
      \node[box] (kappa) at (-1,3) {$\kappa_{a \to b}$};
      \node[box] (circ) at (-1,2) {$- \circ_{\overline{\cA}} -$};
      \draw (bot) to (eta);
      \draw (eta) to[in=-90,out=90] (circ.-45);
      \draw (epsilon) to[in=-90,out=90] (circ.225);
      \draw (circ) to (kappa);
      \draw (kappa) to (top);
    \end{tikzpicture}
  \end{align*}
  which we see is $\kappa_{a \to b}$, after recognizing the mate of an identity.
\end{proof}

\begin{rem}
  A modified version of Lemma \ref{lem:kappa-mate} actually holds in the slightly more general setting of $\sV$ a rigid involutive monoidal category where the self-enrichment $\cV$ is equipped with an enriched dagger structure $\kappa$ and
  $$
  \tikzmath[smallstring]{
    \node[box] (coev1) at (1,0) {$\quad \ev_{u^*}^\dag \quad$};
    \node[box] (f) at (1.5,1) {$\kappa_{u^*\to 1}$};
    \node (t2) at (1.5,2) {$1_\sV$};
    \node (t1) at (.5,2) {$u^*$};
    \draw (1.5,0.35) --node[right]{$\scriptstyle u^{**}$} (f);
    \draw[dotted] (f) to (t2);
    \draw (0.5,0.35) to (t1);
  }
  =
  \tikzmath[smallstring]{
    \node (u) at (1,2) {$1_\sV$};
    \node (v) at (0,2) {$u^*$};
    \node[box] (f) at (1,1) {$\kappa_{1\to u}$};
    \draw[dotted] (f) to (u);
    \draw (f) to[in=-90,out=-90] (v |- f.south);
    \draw (v |- f.south) to (v);
  }\,.
  $$
\end{rem}

\begin{lem}
  The $\sU$-functor $\cF$ from \eqref{eqn:U-functor-construction} above is dagger.
\end{lem}

\begin{proof}
  The mate of $\kappa_{a \to b}^\cA \circ \cF_{a \to b}$ is
  \begin{align*}
    \begin{tikzpicture}[smallstring, baseline=40]
      \node (b) at (0,4.5) {$\sF(b)$};
      \node (a) at (0,0) {$\sF(a)$};
      \node (ba) at (1.5,0) {$\overline{\cA(b \to a)}$};
      \draw[rounded corners] (1,0.75) rectangle (2,1.5) node[midway] {$\kappa_{a \to b}^\cA$};
      \modulebox{0,2}{2}{0.75}{$\omega_{a, \cA(a \to b)}$};
      \modulebox{0,3}{2}{0.75}{$\sF(\epsilon_{a \to b})$};
      \draw[blue] (a) to (b);
      \draw (ba) to (1.5,0.75);
      \draw (1.5,1.5) to (1.5,2);
      \draw (1.5,2.75) to (1.5,3);
    \end{tikzpicture}
    =
    \begin{tikzpicture}[smallstring, baseline=40]
      \node (b) at (0,4.5) {$\sF(b)$};
      \node (a) at (0,0) {$\sF(a)$};
      \node (ba) at (1.5,0) {$\overline{\cA(b \to a)}$};
      \modulebox{0,1}{2.5}{0.75}{$\omega_{a, \overline{\cA(b \to a)}}$};
      \modulebox{0,2}{2.5}{0.75}{$\sF(\id_a \lhd \kappa_{a \to b})$};
      \modulebox{0,3}{2.5}{0.75}{$\sF(\epsilon_{a \to b})$};
      \draw[blue] (a) to (b);
      \draw (ba) to (1.5,1);
      \draw (1.5,1.75) to (1.5,2);
      \draw (1.5,2.75) to (1.5,3);
    \end{tikzpicture}
    \underset{\eqref{eq:MateOfKappa}}{=}
    \begin{tikzpicture}[smallstring, baseline=55]
      \node (b) at (0,6.5) {$\sF(b)$};
      \node (a) at (0,0) {$\sF(a)$};
      \node (ba) at (1.5,0) {$\overline{\cA(b \to a)}$};
      \draw[double] (1.5,3.75) to (1.5,4);
      \modulebox{0,1}{3.5}{0.75}{$\omega_{a, \overline{\cA(b \to a)}}$};
      \modulebox{0,2}{3.5}{0.75}{$\sF(\epsilon_{b \to a}^\dagger \lhd \id_{\overline{\cA(b \to a)}})$};
      \modulebox{0,3}{3.5}{0.75}{$\sF(\alpha_{b, \cA(b \to a), \overline{\cA(b \to a)}}^\dagger)$};
      \modulebox{0,4}{3.5}{0.75}{$\sF(\id_b \lhd \ev_{\cA(b \to a)})$};
      \modulebox{0,5}{1.5}{0.75}{$\sF(\rho_b)$};
      \draw[blue] (a) to (b);
      \draw (ba) to (1.5,1);
      \draw (1.5,1.75) to (1.5,2);
      \draw (1,2.75) to (1,3);
      \draw (2,2.75) to (2,3);
      \draw[dotted] (0.75,4.75) to (0.75,5);
    \end{tikzpicture}
    =
    \begin{tikzpicture}[smallstring, baseline=55]
      \node (b) at (0,6.5) {$\sF(b)$};
      \node (a) at (0,0) {$\sF(a)$};
      \node (ba) at (2.5,0) {$\overline{\cA(b \to a)}$};
      \draw[double] (1.5,3.75) to (1.5,4);
      \modulebox{0,1}{2}{0.75}{$\sF(\epsilon_{b \to a}^\dagger)$};
      \modulebox{0,2}{3.5}{0.75}{$\omega_{b \lhd u, \overline{\cA(b \to a)}}$};
      \modulebox{0,3}{3.5}{0.75}{$\sF(\alpha_{b, \cA(b \to a), \overline{\cA(b \to a)}}^\dagger)$};
      \modulebox{0,4}{3.5}{0.75}{$\sF(\id_b \lhd \ev_{\cA(b \to a)})$};
      \modulebox{0,5}{1.5}{0.75}{$\sF(\rho_b)$};
      \draw[blue] (a) to (b);
      \draw (ba) to (2.5,2);
      \draw (1.5,1.75) to (1.5,2);
      \draw (1.5,2.75) to (1.5,3);
      \draw (2.5,2.75) to (2.5,3);
      \draw[dotted] (0.75,4.75) to (0.75,5);
    \end{tikzpicture}
    =
    \begin{tikzpicture}[smallstring, baseline=75]
      \node (b) at (0,7.5) {$\sF(b)$};
      \node (a) at (0,0) {$\sF(a)$};
      \node (ba) at (2.5,0) {$\overline{\cA(b \to a)}$};
      \draw[double] (1.5,3.75) to (1.5,4);
      \draw[double] (0.75,4.75) to (0.75,5);
      \modulebox{0,1}{2}{0.75}{$\sF(\epsilon_{b \to a}^\dagger)$};
      \modulebox{0,2}{2}{0.75}{$\omega_{b, \cA(b \to a)}^\dagger$};
      \modulebox{0,3}{3.5}{0.75}{$\alpha_{\sF(b), \cA(b \to a), \overline{\cA(b \to a)}}^\dagger$};
      \modulebox{0,4}{3.5}{0.75}{$\omega_{b, \cA(b \to a) \overline{\cA(b \to a)}}$};
      \modulebox{0,5}{3.5}{0.75}{$\sF(\id_b \lhd \ev_{\cA(b \to a)})$};
      \modulebox{0,6}{1.5}{0.75}{$\sF(\rho_b)$};
      \draw[blue] (a) to (b);
      \draw (1.5,1.75) to (1.5,2);
      \draw (1.5,2.75) to (1.5,3);
      \draw (ba) to (2.5,3);
      \draw[dotted] (0.75,5.75) to (0.75,6);
    \end{tikzpicture}
    \displaybreak[1]\\
    =
    \begin{tikzpicture}[smallstring, baseline=75]
      \node (b) at (0,7.5) {$\sF(b)$};
      \node (a) at (0,0) {$\sF(a)$};
      \node (ba) at (2.5,0) {$\overline{\cA(b \to a)}$};
      \modulebox{0,1}{2}{0.75}{$\sF(\epsilon_{b \to a}^\dagger)$};
      \modulebox{0,2}{2}{0.75}{$\omega_{b, \cA(b \to a)}^\dagger$};
      \modulebox{0,3}{3.5}{0.75}{$\alpha_{\sF(b), \cA(b \to a), \overline{\cA(b \to a)}}^\dagger$};
      \modulebox{0,5}{2.25}{0.75}{$\omega_{b, 1_{\sV}}$};
      \modulebox{0,6}{1.5}{0.75}{$\sF(\rho_b)$};
      \draw[blue] (a) to (b);
      \draw (1.5,1.75) to (1.5,2);
      \draw (1.5,2.75) to (1.5,3);
      \draw (ba) to (2.5,3);
      \draw (1.5,3.75)to[in=90,out=90] node[above] (cap) {}  (2.5,3.75);
      \draw[dotted] (cap |- cap) to[in=-90,out=90] (0.75,5);
      \draw[dotted] (0.75,5.75) to (0.75,6);
    \end{tikzpicture}
    =
    \begin{tikzpicture}[smallstring, baseline=60]
      \node (b) at (0,6.25) {$\sF(b)$};
      \node (a) at (0,0) {$\sF(a)$};
      \node (ba) at (2.5,0) {$\overline{\cA(b \to a)}$};
      \modulebox{0,1}{2}{0.75}{$\sF(\epsilon_{b \to a}^\dagger)$};
      \modulebox{0,2}{2}{0.75}{$\omega_{b, \cA(b \to a)}^\dagger$};
      \modulebox{0,3}{3.5}{0.75}{$\alpha_{\sF(b), \cA(b \to a), \overline{\cA(b \to a)}}^\dagger$};
      \modulebox{0,5}{1.5}{0.75}{$\rho_{\sF(b)}$};
      \draw[blue] (a) to (b);
      \draw (1.5,1.75) to (1.5,2);
      \draw (1.5,2.75) to (1.5,3);
      \draw (ba) to (2.5,3);
      \draw (1.5,3.75)to[in=90,out=90] node[above] (cap) {}  (2.5,3.75);
      \draw[dotted] (cap |- cap) to[in=-90,out=90] (0.75,5);
    \end{tikzpicture}
    \underset{\eqref{eqn:U-functor-construction}}{=}
    \begin{tikzpicture}[smallstring, baseline=60]
      \node (b) at (0,6.25) {$\sF(b)$};
      \node (a) at (0,0) {$\sF(a)$};
      \node (ba) at (2.5,0) {$\overline{\cA(b \to a)}$};
      \modulebox{0,1}{2}{0.75}{$\epsilon_{\sF(b) \to \sF(a)}^\dagger$};
      \modulebox{0,2}{2}{0.75}{$\cF_{a \to b}^\dagger$};
      \modulebox{0,3}{3.5}{0.75}{$\alpha_{\sF(b), \cA(b \to a), \overline{\cA(b \to a)}}^\dagger$};
      \modulebox{0,5}{1.5}{0.75}{$\rho_{\sF(b)}$};
      \draw[blue] (a) to (b);
      \draw (1.5,1.75) to (1.5,2);
      \draw (1.5,2.75) to (1.5,3);
      \draw (ba) to (2.5,3);
      \draw (1.5,3.75)to[in=90,out=90] node[above] (cap) {}  (2.5,3.75);
      \draw[dotted] (cap |- cap) to[in=-90,out=90] (0.75,5);
    \end{tikzpicture}
    =
    \begin{tikzpicture}[smallstring, baseline=40]
      \node (b) at (0,6) {$\sF(b)$};
      \node (a) at (0,0) {$\sF(a)$};
      \node (ba) at (3,0) {$\overline{\cA(b \to a)}$};
      \modulebox{0,1}{2}{0.75}{$\epsilon_{\sF(b) \to \sF(a)}^\dagger$};
      \draw[rounded corners] (2.5,1) rectangle (3.5,1.75) node[midway] {$\overline{\cF_{a \to b}}$};
      \modulebox{0,2}{3.5}{0.75}{$\alpha_{\sF(b), \cA(b \to a), \overline{\cA(b \to a)}}^\dagger$};
      \modulebox{0,4.5}{1.5}{0.75}{$\rho_{\sF(b)}$};
      \draw[blue] (a) to (b);
      \draw (ba) to (3,1);
      \draw (3,1.75) to (3,2);
      \draw (1.125,1.75) to (1.125,2);
      \draw (1.125,2.75) to[in=90,out=90] node[above] (cap) {}  (3,2.75);
      \draw[dotted] (cap |- cap) to[in=-90,out=90] (0.75,4.5);
    \end{tikzpicture}\, ,
  \end{align*}
  which is the mate of $\overline{\sF_{b \to a}} \circ \kappa_{\sF(a) \to \sF(b)}^\cB$ by Lemma \ref{lem:kappa-mate}.
\end{proof}
%%%%%%%%%%%%%%%%%%%%%%%%%%%%%%%%%%%%%%%%%%%%

\begin{construction}
  Given 1-cells $\cF, \cG : \cA \to \cB$ in $\dagucat$ and a module natural transformation $\Theta : \Phi(\cF) \Rightarrow \Phi(\cG)$, we construct a $\sU$-natural transformation $\theta$ by defining
  $\Theta_a := \theta_a \in \sU(1_{\sU} \to \cB(\cF(a) \to \cG(a))) = \sB(\sF(a) \to \sG(a))$. 
  The same proof as for the naturality half of \cite[Prop 6.11]{2104.07747} applies to show that $\theta$ is a $\sU$-natural transformation.
\end{construction}

%%%%%%%%%%%%%%%%%%%%%%%%%%%%%%%%%%%%%%%%%%%%
\subsection{Essentially surjective on 0-cells}
\label{sec:equivalence-dagger-U-mods}

Starting with a dagger $\sU$-module category $(\sA, \lhd)$, we can construct a dagger $\sU$-category $(\cA, \kappa)$ as in the previous subsection.
We write $(\sA', \blacktriangleleft)$ for $\Phi(\cA)$.
The following identity appears repeatedly throughout our calculations, so we record it here:

\begin{rem}
  \label{rem:circ-mate-calculations}
  Suppose that $\cA$ is a weakly tensored $\sV$-category with $\sA$ its underlying $\sV$-module category from Construction \ref{constr:U-mod-to-dagger-U-cat}.
  For $f \in \sV(u \to \cA(a \to b))$ and $g \in \sV(v \to \cA(b \to c))$, it follows from naturality and Remark \ref{rem:mate of circ} that as morphisms in $\sA$,
  \begin{equation}
  \label{eq:circ-mate-calculations}
    \begin{tikzpicture}[smallstring, baseline=40]
      \node (bot) at (0,0) {$a$};
      \node (top) at (0,4) {$c$};
      \node (u) at (1,0) {$u$};
      \node (v) at (2,0) {$v$};
      \draw[rounded corners] (0.75,0.5) rectangle (1.25,1) node[midway] {$f$};
      \draw[rounded corners] (1.75,0.5) rectangle (2.25,1) node[midway] {$g$};
      \draw[rounded corners] (0.75,1.5) rectangle (2.25,2) node[midway] {$- \circ_\cA -$};
      \modulebox{0,2.5}{2.5}{0.5}{$\epsilon_{a \to c}$}
      \draw[blue] (bot) to (top);
      \draw (u) to (1,0.5);
      \draw (v) to (2,0.5);
      \draw (1,1) to (1,1.5);
      \draw (2,1) to (2,1.5);
      \draw (1.5,2) to (1.5,2.5);
    \end{tikzpicture}
    =
    \begin{tikzpicture}[smallstring, baseline=40]
      \node (bot) at (0,0) {$a$};
      \node (top) at (0,4) {$c$};
      \node (u) at (1,0) {$u$};
      \node (v) at (2,0) {$v$};
      \modulebox{0,0.5}{2.5}{0.5}{$\alpha_{a, u, v}$}
      \draw[rounded corners] (0.75,1.25) rectangle (1.25,1.75) node[midway] {$f$};
      \draw[rounded corners] (1.75,1.25) rectangle (2.25,1.75) node[midway] {$g$};
      \modulebox{0,2}{1.5}{0.5}{$\epsilon_{a \to b}$}
      \modulebox{0,2.75}{2.5}{0.5}{$\epsilon_{b \to c}$}
      \draw[blue] (bot) to (top);
      \draw (u) to (1,0.5);
      \draw (1,1) to (1,1.25);
      \draw (v) to (2,0.5);
      \draw (2,1) to (2,1.25);
      \draw (1,1.75) to (1,2);
      \draw (2,1.75) to (2,2.75);
    \end{tikzpicture}\, .
  \end{equation}
  If instead we have a strongly unital lax $\sV$-module $\sA$,
  with $\cA$ the $\sV$-category from Construction \ref{const:ModToCat},
  then the above equation still holds in $\sA$, this time by the definition of $- \circ_\cA -$ instead of Remark \ref{rem:mate of circ}.
  We also remind the reader that $\alpha_{a, 1_{\sV}, v} = \rho_a^{-1} \lhd \id_v$ and $\alpha_{a, u, 1_{\sV}} = \rho_{a \lhd u}^{-1}$.
\end{rem}

\begin{prop}
  \label{prop:equivalence-of-dagger-U-mods}
  The dagger $\sU$-module categories $\sA$ and $\sA'$ are isomorphic,
  i.e., there are strong dagger module functors $(\sG, \omega') : \sA \to \sA'$ and $(\sH, \omega): \sA' \to \sA$ that witness an isomorphism of dagger module categories.
\end{prop}

\begin{proof}
  Define functors $\sG: \sA\to \sA'$ and $\sH:\sA'\to \sA$ to be the identity on objects.
  For a morphism $g\in \sA(a\to b) $, we define $\sG(g)$ as the mate of $\rho^\sA_a \circ g$ under Adjunction \eqref{eq:ModToCatAdjunction}.
  To see $\sG$ is a functor, we take the mate of
  $\sG(g_1)\circ  \sG(g_2)$ for $g_1: a\to b$ and $g_2: b\to c$:
  $$
  \mate\left(
  \tikzmath[smallstring]{
    \node (top) at (.75,3) {$\cA(a \to c)$};
    \node[box] (circ) at (.75,2) {$-\circ_\cA-$};
    \node[box] (eta) at (0,1) {$\sG(g_1)$};
    \node[box] (rho) at (1.5,1) {$\sG(g_2)$};
    \draw (eta) to[in=-90,out=90] node[left]{$\scriptstyle \cA(a\to b)$} (circ.-135);
    \draw (rho) to[in=-90,out=90] node[right]{$\scriptstyle \cA(b \to c)$} (circ.-45);
    \draw (circ) to (top);
  }
  \right)
  =
  \begin{tikzpicture}[smallstring, baseline=30]
    \node(a) at (0,0) {$a$};
    \node (c) at (0,3.5) {$c$};
    \modulebox{0,0.5}{3.5}{0.5}{$\alpha_{a, 1_{\sV}, 1_{\sV}}$}
    \modulebox{0,1.5}{3.5}{0.5}{$\mate (\sG(g_1)) \lhd \id_{1_{\sV}}$}
    \modulebox{0,2.5}{2.5}{0.5}{$\mate (\sG(g_2))$}
    \draw[blue] (a) to (c);
    \draw[dotted] (1.5,0) to (1.5,0.5);
    \draw[dotted] (1,1) to (1,1.5);
    \draw[dotted] (2,1) to (2,1.5);
    \draw[dotted] (1,2) to (1,2.5);
  \end{tikzpicture}
  = 
  \begin{tikzpicture}[smallstring, baseline=45]
    \node(a) at (0,0) {$a$};
    \node (c) at (0,5.5) {$c$};
    \modulebox{0,0.5}{2}{0.5}{$\alpha_{a, 1_{\sV}, 1_{\sV}}$}
    \modulebox{0,1.5}{2}{0.5}{$\rho_a \lhd \id_{1_{\sV}}$}
    \modulebox{0,2.5}{2}{0.5}{$g_1 \lhd \id_{1_{\sV}}$}
    \modulebox{0,3.5}{2}{0.5}{$\rho_b$}
    \modulebox{0,4.5}{0.5}{0.5}{$g_2$}
    \draw[blue] (a) to (c);
    \draw[dotted] (1,0) to (1,0.5);
    \draw[dotted] (0.5,1) to (0.5,1.5);
    \draw[dotted] (1.5,1) to (1.5,1.5);
    \draw[dotted] (1,2) to (1,2.5);
    \draw[dotted] (1,3) to (1,3.5);
  \end{tikzpicture}
  =
  \begin{tikzpicture}[smallstring, baseline=45]
    \node(a) at (0,0) {$a$};
    \node (c) at (0,5.5) {$c$};
    \modulebox{0,0.5}{2}{0.5}{$\alpha_{a, 1_{\sV}, 1_{\sV}}$}
    \modulebox{0,1.5}{2}{0.5}{$\rho_a \lhd \id_{1_{\sV}}$}
    \modulebox{0,2.5}{2}{0.5}{$\rho_a$}
    \modulebox{0,3.5}{0.5}{0.5}{$g_1$}
    \modulebox{0,4.5}{0.5}{0.5}{$g_2$}
    \draw[blue] (a) to (c);
    \draw[dotted] (1,0) to (1,0.5);
    \draw[dotted] (0.5,1) to (0.5,1.5);
    \draw[dotted] (1.5,1) to (1.5,1.5);
    \draw[dotted] (1,2) to (1,2.5);
  \end{tikzpicture}
  =
  \begin{tikzpicture}[smallstring, baseline=30]
    \node(a) at (0,0) {$a$};
    \node (c) at (0,3.5) {$c$};
    \modulebox{0,0.5}{1.5}{0.5}{$\rho_a$}
    \modulebox{0,1.5}{0.5}{0.5}{$g_1$}
    \modulebox{0,2.5}{0.5}{0.5}{$g_2$}
    \draw[blue] (a) to (c);
    \draw[dotted] (1,0) to (1,0.5);
  \end{tikzpicture}\, ,
  $$
  which is exactly the mate of $\sG(g_1 \circ g_2)$.
  That $\sG(\id_a)=\id_a$ is straightforward.

  For a morphism $h\in \sA'(a\to b)= \sV(1\to \cA(a\to b))$,
  we define
  $$
  \sH(h):= (\rho^\sA_a)^{-1} \circ (\id_a \lhd h) \circ \epsilon_{a \to b}.
  $$
  To see $\sH$ is a functor, for $h_1: a\to b$ and $h_2:b\to c$, we see
  $$
  \sH(h_1)\circ \sH(h_2)
  =
  \begin{tikzpicture}[smallstring, baseline=60]
    \node (a) at (0,0) {$a$};
    \node (c) at (0,6) {$c$};
    \modulebox{0,0.5}{0.75}{0.75}{$\rho_a^{-1}$};
    \draw[rounded corners] (0.25,1.5) rectangle (0.75,2) node[midway] {$h_1$};
    \modulebox{0,2.25}{1}{0.5}{$\epsilon_{a \to b}$};
    \modulebox{0,3}{0.75}{0.75}{$\rho_b^{-1}$};
    \draw[rounded corners] (0.25,4) rectangle (0.75,4.5) node[midway] {$h_2$};
    \modulebox{0,4.75}{1}{0.5}{$\epsilon_{b \to c}$};
    \draw[blue] (a) to (c);
    \draw[dotted] (0.5,1.25) to (0.5,1.5);
    \draw (0.5,2) to (0.5,2.25);
    \draw[dotted] (0.5,3.75) to (0.5,4);
    \draw (0.5,4.5) to (0.5,4.75);
  \end{tikzpicture}
  =
  \begin{tikzpicture}[smallstring, baseline=60]
    \node (a) at (0,0) {$a$};
    \node (c) at (0,6) {$c$};
    \modulebox{0,0.5}{1}{0.75}{$\rho_a^{-1}$};
    \modulebox{0,1.5}{2}{0.75}{$\rho_{a \lhd 1_{\sV}}^{-1}$};
    \draw[rounded corners] (0.25,2.5) rectangle (0.75,3) node[midway] {$h_1$};
    \modulebox{0,3.25}{1}{0.5}{$\epsilon_{a \to b}$};
    \draw[rounded corners] (1.25,4) rectangle (1.75,4.5) node[midway] {$h_2$};
    \modulebox{0,4.75}{2}{0.5}{$\epsilon_{b \to c}$};
    \draw[blue] (a) to (c);
    \draw[dotted] (0.5,1.25) to (0.5,1.5);
    \draw[dotted] (0.5,2.25) to (0.5,2.5);
    \draw (0.5,3) to (0.5,3.25);
    \draw[dotted] (1.5,2.25) to (1.5,4);
    \draw (1.5,4.5) to (1.5,4.75);
  \end{tikzpicture}
  \underset{\text{\eqref{eq:circ-mate-calculations}}}{=}
  \begin{tikzpicture}[smallstring, baseline=35]
    \node (a) at (0,0) {$a$};
    \node (c) at (0,3.5) {$c$};
    \modulebox{0,0.5}{2}{0.75}{$\rho_a^{-1}$};
    \draw[rounded corners] (0.25,1.5) rectangle (1.75,2) node[midway] {$h_1 \circ h_2$};
    \modulebox{0,2.25}{2}{0.5}{$\epsilon_{a \to c}$};
    \draw[blue] (a) to (c);
    \draw[dotted] (1,1.25) to (1,1.5);
    \draw (1,2) to (1,2.25);
  \end{tikzpicture}
  =
  \sH(h_1 \circ h_2).
  $$
  Again, the proof that $\sH(\id_a)=\id_a$ is straightforward.
  Noting that $\sH(h)$ is $\rho_a^{-1}$ composed with the mate of $h$, it follows immediately that $\sG(\sH(h)) = h$ and $\sH(\sG(g)) = g$. Therefore $\sG, \sH$ witness an isomorphism of categories.

  Lastly, we must construct unitary modulators for $\sG$ and $\sH$ that are inverse to each other.
  In \cite[\S3.3]{1809.09782} it was stated that these modulators could be taken to be the identity;
  this is not true because the module actions need not agree.
  Define $\omega_{a, v} \in \sA( \sH(a) \lhd v \to \sH(a \blacktriangleleft v) ) = \sA(a \lhd v \to a \blacktriangleleft v)$ and $\omega'_{a, v} \in \sA'(a \blacktriangleleft v \to a \lhd v)$ via
  \begin{align*}
    \omega_{a, v} :=
    \begin{tikzpicture}[smallstring,baseline=25]
      \node (a) at (0,0) {$a$};
      \node (tri) at (0.625,0) {$\lhd$};
      \node (v) at (1.25,0) {$v$};
      \node (av) at (0,3) {$a \blacktriangleleft v$};
      \draw[rounded corners] (0.5,0.5) rectangle (2,1.25) node[midway] {$\eta_{a, v}'$};
      \modulebox{0,1.5}{2}{0.75}{$\epsilon_{a \to a \blacktriangleleft v}$}
      \draw[blue] (a) to (av);
      \draw (v) to (1.25,0.5);
      \draw (1.25,1.25) to (1.25,1.5);
    \end{tikzpicture}
    \qquad
    \text{and}
    \qquad
    \omega'_{a, v} :=
    \begin{tikzpicture}[smallstring,baseline=25]
      \node (a) at (0,0) {$a$};
      \node (tri) at (0.625,0) {$\blacktriangleleft$};
      \node (v) at (1.25,0) {$v$};
      \node (av) at (0,3) {$a \lhd v$};
      \draw[rounded corners] (0.5,0.5) rectangle (2,1.25) node[midway] {$\eta_{a, v}$};
      \modulebox{0,1.5}{2}{0.75}{$\epsilon_{a \to a \blacktriangleleft v}'$}
      \draw[blue] (a) to (av);
      \draw (v) to (1.25,0.5);
      \draw (1.25,1.25) to (1.25,1.5);
    \end{tikzpicture}\, ,
  \end{align*}
  where $\eta'_{a, v}$ and $\epsilon'_{a \to b}$ are the unit and counit, respectively, of Adjunction \ref{eq:CatToModAdjunction} for $\sA'$.
  Note that $\omega_{a, v}$ and $\omega'_{a, v}$ are the unit and counit, respectively, of the adjunction formed by composing Adjunction \ref{eq:CatToModAdjunction} and Adjunction \eqref{eq:ModToCatAdjunction}:
  \begin{align*}
    \sA(\sH(a \blacktriangleleft v) \to a \lhd v) \cong \sA'(a \blacktriangleleft v \to \sG(a \lhd v)).
  \end{align*}

  As the proofs that $\omega$ and $\omega'$ are modulators for $\sH$ and $\sG$, respectively, are fairly tedious, we defer them to Appendix \ref{app:modulators}.
  We also prove in the appendix that $\omega'$ is a $\sU$-natural transformation between $(a \lhd -)$ and $(a \blacktriangleleft -)$.
  These proofs (Lemmas \ref{lem:modulator-natural}, \ref{lem:modulator-associative}, and \ref{lem:omega-V-natural}) and
  the subsequent Remark \ref{rem:eta-eta'}
  are used in the proof that the modulator $\omega$ is unitary, which is shown in Lemma \ref{lem:modulator unitary} below.
  \end{proof}

  \begin{lem}
    \label{lem:modulator unitary}
    The modulator $\omega$ is unitary.
  \end{lem}

  \begin{proof}
    Since the mate of $\omega_{a, u}$ under Adjunction \ref{eq:CatToModAdjunction} is $\eta_{a, v}'$,
    we're able to closely follow many of the steps in the proof of Proposition \ref{prop:modulator-from-V-functor},
    with each $\cF_{a \to b}$ being the identity.
    We are unable to directly apply the lemma, because we only have \eqref{eq:DefineLhd} and Lemma \ref{lem:eta-kappa-inv} using $\eta'$ and not $\eta$,
    since $\sA$ is not the underlying category of $\cA$ ($\sA'$ is the underlying category of $\cA$).

    Working with $\omega_{a, u}'$ instead of $\omega_{a, u}$ allows us to only directly use \eqref{eq:DefineLhd} in $\sA'$,
    and then there is just one application of Lemma \ref{lem:eta-kappa-inv} in $\sA$.
    We claim the following as an analog of Lemma \ref{lem:eta-kappa-inv} using $\eta$ instead of $\eta'$, with which we complete the proof:
    \begin{align*}
      \begin{tikzpicture}[smallstring, baseline=45]
        \node (bot) at (0,1) {$u$};
        \node[box] (eta) at (0,2) {$\eta_{a, u}$};
        \node[box] (kappainv) at (0,3) {$\kappa_{a \to a \lhd u}^{-1}$};
        \node (top) at (0,4) {$\overline{\cA(a \lhd u \to a)}$};
        \draw (bot) to (eta);
        \draw (eta) to (kappainv);
        \draw (kappainv) to (top);
      \end{tikzpicture}
      =
      \begin{tikzpicture}[smallstring, baseline=0]
        \node (bot) at (3,-4) {$u$};
        \node[box] (phi) at (3,-3) {$\varphi_u$};
        \node[box] (kappa1) at (3,-2) {$\overline{\kappa_{u \to 1}}$};
        \node[box] (eta) at (3,-1) {$\overline{\eta_{a \lhd u, u^*}}$};
        \node[box] (alpha) at (1,-1.1) {$\overline{\sG(\alpha_{a, u, u^*}^{-1})}$};
        \node[box] (circ3) at (2,0) {$- \circ_{\overline{\cA}} -$};
        \node[box] (ev) at (0,0) {$\overline{\sG(\id_a \lhd \ev_u)}$};
        \node[box] (circ2) at (1,1) {$- \circ_{\overline{\cA}} -$};
        \node[box] (rho) at (-1,1) {$\overline{\sG(\rho_a^\dagger)^{-1}}$};
        \node[box] (circ1) at (0,2.1) {$- \circ_{\overline{\cA}} -$};
        \node (top) at (0,3.1) {$\cA(a \to a \lhd u)$};
        \draw (bot) to (phi);
        \draw (phi) to (kappa1);
        \draw (kappa1) to (eta);
        \draw (eta) to[in=-90,out=90] (circ3.-45);
        \draw (alpha) to[in=-90,out=90] (circ3.225);
        \draw (circ3) to[in=-90,out=90] (circ2.-45);
        \draw (ev) to[in=-90,out=90] (circ2.225);
        \draw (circ2) to[in=-90,out=90] (circ1.-45);
        \draw (rho) to[in=-90,out=90] (circ1.225);
        \draw (circ1) to (top);
      \end{tikzpicture}
    \end{align*}
    Using this at the first application of Lemma \ref{lem:eta-kappa-inv} in the proof of Proposition \ref{prop:modulator-from-V-functor},
    the rest of the proof follows exactly the same,
    with $\sG$ in place of $\sF$ and the remaining uses of \eqref{eq:DefineLhd} and Lemma \ref{lem:eta-kappa-inv} taking place in $\sA'$.

    Finally, to prove the above identity, we take the same steps as in the proof of Lemma \ref{lem:eta-kappa-inv}, so it suffices to prove the equivalent identity
    \begin{align*}
      \begin{tikzpicture}[smallstring, baseline=40]
        \node (bot) at (0,0) {$\overline{v}$};
        \node (top) at (1,4) {$\cA(a \lhd v \to a \lhd 1)$};
        \node[box] (eta) at (0,1) {$\overline{\eta_{a, v}}$};
        \node[box] (kappa) at (0,2) {$\kappa_{a \lhd v \to a}$};
        \node[box] (rho) at (2,2) {$\sG(\rho_a^\dagger)$};
        \node[box] (circ) at (1,3) {$- \circ -$};
        \draw (bot) to (eta);
        \draw (eta) to (kappa);
        \draw (kappa) to[in=-90,out=90] (circ.225);
        \draw (rho) to[in=-90,out=90] (circ.-45);
        \draw (circ) to (top);
      \end{tikzpicture}
      =
      \begin{tikzpicture}[smallstring, baseline=30]
        \node (bot) at (0,-1) {$\overline{v}$};
        \node (top) at (2,4) {$\cA(a \lhd v \to a \lhd 1)$};
        \node[box] (kappa) at (0,0) {$\kappa_{v \to 1}^{\cV}$};
        \node[box] (eta) at (0,0.9) {$\eta_{a \lhd v , {v^*}}$};
        \node[box] (alpha) at (2,0.9) {$\sG(\alpha_{a , v , {v^*}}^{-1})$};
        \node[box] (circ1) at (1,2) {$- \circ -$};
        \node[box] (ev) at (3,2) {$\sG(\id_a \lhd \ev_v)$};
        \node[box] (circ2) at (2,3) {$- \circ -$};
        \draw (bot) to (kappa);
        \draw (kappa) to (eta);
        \draw (eta) to[in=-90,out=90] (circ1.225);
        \draw (alpha) to[in=-90,out=90] (circ1.-45);
        \draw (circ1) to[in=-90,out=90] (circ2.225);
        \draw (ev) to[in=-90,out=90] (circ2.-45);
        \draw (circ2) to (top);
      \end{tikzpicture}\, .
    \end{align*}
    Starting with the left hand side, using that $\sG$ is a dagger functor, and applying \ref{kappa:2} gives
    \begin{align*}
      \begin{tikzpicture}[smallstring, baseline=45]
        \node (bot) at (0.3,0) {\small $\overline{v}$};
        \node (top) at (1,4) {$\cA(a \lhd v \to a \lhd 1)$};
        \node[box] (eta) at (0.3,1) {$\overline{\eta_{a, v}}$};
        \node[box] (rho) at (1.7,1) {$\overline{\sG(\rho_a)}$};
        \node[box] (circ) at (1,2) {$- \circ_{\overline{\cA}} -$};
        \node[box] (kappa) at (1,3) {$\kappa_{a \lhd v \to a \lhd 1}$};
        \draw (bot) to (eta);
        \draw (eta) to[in=-90,out=90] (circ.225);
        \draw (rho) to[in=-90,out=90] (circ.-45);
        \draw (circ) to (kappa);
        \draw (kappa) to (top);
      \end{tikzpicture}
      =
      \hspace{-3pt}
      \begin{tikzpicture}[smallstring, baseline=35]
        \node (bot) at (0.3,-1) {\small $\overline{v}$};
        \node (top) at (1,4) {$\cA(a \lhd v \to a \lhd 1)$};
        \node[box] (eta) at (0.3,1) {$\overline{\eta_{a, v}}$};
        \node[box] (rho) at (1.2,0) {$\overline{\rho_a'}$};
        \node[box] (omega) at (2.2,0) {$\overline{{\omega'}_{a, 1}^{-1}}$};
        \node[box] (circ0) at (1.7,1) {$- \circ_{\overline{\cA}} -$};
        \node[box] (circ) at (1,2) {$- \circ_{\overline{\cA}} -$};
        \node[box] (kappa) at (1,3) {$\kappa_{a \lhd v \to a \lhd 1}$};
        \draw (bot) to (eta);
        \draw (eta) to[in=-90,out=90] (circ.225);
        \draw (rho) to[in=-90,out=90] (circ0.225);
        \draw (omega) to[in=-90,out=90] (circ0.-45);
        \draw (circ0) to[in=-90,out=90] (circ.-45);
        \draw (circ) to (kappa);
        \draw (kappa) to (top);
      \end{tikzpicture}
      \underset{\text{(Rem.~\ref{rem:eta-eta'})}}{=}
      \begin{tikzpicture}[smallstring, baseline=35]
        \node (bot) at (0.5,-1) {\small $\overline{v}$};
        \node (top) at (1,4) {$\cA(a \lhd v \to a \lhd 1)$};
        \node[box] (omega1) at (-0.5,0) {$\overline{{\omega'}_{a, v}}$};
        \node[box] (eta) at (0.5,0) {$\overline{\eta_{a, v}'}$};
        \node[box] (circ1) at (0,1) {$- \circ_{\overline{\cA}} -$};
        \node[box] (rho) at (1.5,0) {$\overline{\rho_a'}$};
        \node[box] (omega2) at (2.5,0) {$\overline{{\omega'}_{a, 1}^{-1}}$};
        \node[box] (circ0) at (2,1) {$- \circ_{\overline{\cA}} -$};
        \node[box] (circ) at (1,2) {$- \circ_{\overline{\cA}} -$};
        \node[box] (kappa) at (1,3) {$\kappa_{a \lhd v \to a \lhd 1}$};
        \draw (bot) to (eta);
        \draw (omega1) to[in=-90,out=90] (circ1.225);
        \draw (eta) to[in=-90,out=90] (circ1.-45);
        \draw (circ1) to[in=-90,out=90] (circ.225);
        \draw (rho) to[in=-90,out=90] (circ0.225);
        \draw (omega2) to[in=-90,out=90] (circ0.-45);
        \draw (circ0) to[in=-90,out=90] (circ.-45);
        \draw (circ) to (kappa);
        \draw (kappa) to (top);
      \end{tikzpicture}
      \hspace{-4pt}
      \underset{\eqref{eq:La - 1->v}}{=}
      \begin{tikzpicture}[smallstring, baseline=35]
        \node (bot) at (1,-1) {\small $\overline{v}$};
        \node (top) at (1,4) {$\cA(a \lhd v \to a \lhd 1)$};
        \node[box] (omega1) at (-1,-0.1) {$\overline{{\omega'}_{a, v}}$};
        \node[box] (tri) at (1,-0.1) {$\overline{(a \blacktriangleleft -)_{1 \to v}}$};
        \node[box] (circ1) at (0,1) {$- \circ_{\overline{\cA}} -$};
        \node[box] (omega2) at (2,0.9) {$\overline{{\omega'}_{a, 1}^{-1}}$};
        \node[box] (circ) at (1,2) {$- \circ_{\overline{\cA}} -$};
        \node[box] (kappa) at (1,3) {$\kappa_{a \lhd v \to a \lhd 1}$};
        \draw (bot) to (tri);
        \draw (omega1) to[in=-90,out=90] (circ1.225);
        \draw (tri) to[in=-90,out=90] (circ1.-45);
        \draw (circ1) to[in=-90,out=90] (circ.225);
        \draw (omega2) to[in=-90,out=90] (circ.-45);
        \draw (circ) to (kappa);
        \draw (kappa) to (top);
      \end{tikzpicture}\, ,
    \end{align*}
    which, by Lemma \ref{lem:omega-V-natural}, is equal to $\overline{(a \lhd -)_{1 \to v}} \circ \kappa_{a \lhd v \to a \lhd 1}$. Meanwhile, starting with the right hand side gives
    \begin{align*}
      \begin{tikzpicture}[smallstring, baseline=35]
        \node (bot) at (0,-1) {$\overline{v}$};
        \node (top) at (2,4) {$\cA(a \lhd v \to a \lhd 1)$};
        \node[box] (kappa) at (0,0) {$\kappa_{v \to 1}^{\cV}$};
        \node[box] (eta) at (0,0.9) {$\eta_{a \lhd v , {v^*}}$};
        \node[box] (alpha) at (2,0.9) {$\sG(\alpha_{a , v , {v^*}}^{-1})$};
        \node[box] (circ1) at (1,2) {$- \circ -$};
        \node[box] (ev) at (3,2) {$\sG(\id_a \lhd \ev_v)$};
        \node[box] (circ2) at (2,3) {$- \circ -$};
        \draw (bot) to (kappa);
        \draw (kappa) to (eta);
        \draw (eta) to[in=-90,out=90] (circ1.225);
        \draw (alpha) to[in=-90,out=90] (circ1.-45);
        \draw (circ1) to[in=-90,out=90] (circ2.225);
        \draw (ev) to[in=-90,out=90] (circ2.-45);
        \draw (circ2) to (top);
      \end{tikzpicture}
      &\underset{\text{(Lem.~\ref{lem:modulator-associative})}}{=}
      \begin{tikzpicture}[smallstring, baseline=10]
        \node (bot) at (-0.5,-3) {$\overline{v}$};
        \node[box] (kappa) at (-0.5,0) {$\kappa_{v \to 1}^{\cV}$};
        \node[box] (eta) at (-0.5,0.9) {$\eta_{a \lhd v , {v^*}}$};
        \node[box] (circ1) at (1,2) {$- \circ -$};
        \node[box] (circ2) at (2,1) {$- \circ -$};
        \node[box] (omega1) at (1,0) {${\omega'}_{a \lhd v, v^*}^{-1}$};
        \node[box] (circ3) at (3,0) {$- \circ -$};
        \node[box] (omega2) at (2,-1) {${\omega'}_{a, v}^{-1} \blacktriangleleft \id_{v^*}$};
        \node[box] (circ4) at (4,-1) {$- \circ -$};
        \node[box] (alpha) at (3,-2) {${\alpha_{a, v, v^*}'}^{-1}$};
        \node[box] (omega3) at (5,-2) {${\omega'}_{a, v v^*}$};
        \node[box] (ev) at (3,2) {$\sG(\id_a \lhd \ev_v)$};
        \node[box] (circ) at (2,3) {$- \circ -$};
        \node (top) at (2,4) {$\cA(a \lhd v \to a \lhd 1)$};
        \draw (bot) to (kappa);
        \draw (kappa) to (eta);
        \draw (eta) to[in=-90,out=90] (circ1.225);
        \draw (omega1) to[in=-90,out=90] (circ2.225);
        \draw (circ2) to[in=-90,out=90] (circ1.-45);
        \draw (omega2) to[in=-90,out=90] (circ3.225);
        \draw (circ3) to[in=-90,out=90] (circ2.-45);
        \draw (alpha) to[in=-90,out=90] (circ4.225);
        \draw (omega3) to[in=-90,out=90] (circ4.-45);
        \draw (circ4) to[in=-90,out=90] (circ3.-45);
        \draw (circ1) to[in=-90,out=90] (circ.225);
        \draw (ev) to[in=-90,out=90] (circ.-45);
        \draw (circ) to (top);
      \end{tikzpicture}
      \displaybreak[1]\\
      \underset{\text{(Lem.~\ref{rem:eta-eta'})}}{=}
      \begin{tikzpicture}[smallstring, baseline=0]
        \node (bot) at (0,-3) {$\overline{v}$};
        \node[box] (kappa) at (0,-2) {$\kappa_{v \to 1}^{\cV}$};
        \node[box] (eta) at (0,-1) {$\eta_{a \lhd v , v^*}'$};
        \node[box] (circ2) at (2,1) {$- \circ -$};
        \node[box] (circ3) at (3,0) {$- \circ -$};
        \node[box] (omega2) at (2,-1) {${\omega'}_{a, v}^{-1} \blacktriangleleft \id_{v^*}$};
        \node[box] (circ4) at (4,-1) {$- \circ -$};
        \node[box] (alpha) at (3,-2) {${\alpha_{a, v, v^*}'}^{-1}$};
        \node[box] (omega3) at (5,-2) {${\omega'}_{a, v v^*}$};
        \node[box] (ev) at (4,1) {$\sG(\id_a \lhd \ev_v)$};
        \node[box] (circ) at (3,2) {$- \circ -$};
        \node (top) at (3,3) {$\cA(a \lhd v \to a \lhd 1)$};
        \draw (bot) to (kappa);
        \draw (kappa) to (eta);
        \draw (eta) to[in=-90,out=90] (circ2.225);
        \draw (omega2) to[in=-90,out=90] (circ3.225);
        \draw (circ3) to[in=-90,out=90] (circ2.-45);
        \draw (alpha) to[in=-90,out=90] (circ4.225);
        \draw (omega3) to[in=-90,out=90] (circ4.-45);
        \draw (circ4) to[in=-90,out=90] (circ3.-45);
        \draw (ev) to[in=-90,out=90] (circ.-45);
        \draw (circ2) to[in=-90,out=90] (circ.225);
        \draw (circ) to (top);
      \end{tikzpicture}
      &\underset{\text{(Lem.~\ref{lem:modulator-natural})}}{=}
      \begin{tikzpicture}[smallstring, baseline=0]
        \node (bot) at (1,-3) {$\overline{v}$};
        \node[box] (kappa) at (1,-2) {$\kappa_{v \to 1}^{\cV}$};
        \node[box] (eta) at (1,-1) {$\eta_{a \lhd v , v^*}'$};
        \node[box] (circ2) at (4,1.1) {$- \circ -$};
        \node[box] (circ3) at (5,0.1) {$- \circ -$};
        \node[box] (omega2) at (3,0) {${\omega'}_{a, v}^{-1} \blacktriangleleft \id_{v^*}$};
        \node[box] (circ4) at (6,-1) {$- \circ -$};
        \node[box] (alpha) at (4,-1) {${\alpha_{a, v, v^*}'}^{-1}$};
        \node[box] (ev) at (5,-2) {$\sG(\id_a) \blacktriangleleft \ev_v$};
        \node[box] (omega3) at (7,-2) {${\omega'}_{a, 1_{\sA}}$};
        \node[box] (circ) at (3,2.1) {$- \circ -$};
        \node (top) at (3,3) {$\cA(a \lhd v \to a \lhd 1)$};
        \draw (bot) to (kappa);
        \draw (kappa) to (eta);
        \draw (eta) to[in=-90,out=90] (circ.225);
        \draw (omega2) to[in=-90,out=90] (circ2.225);
        \draw (circ3) to[in=-90,out=90] (circ2.-45);
        \draw (alpha) to[in=-90,out=90] (circ3.225);
        \draw (omega3) to[in=-90,out=90] (circ4.-45);
        \draw (ev) to[in=-90,out=90] (circ4.225);
        \draw (circ4) to[in=-90,out=90] (circ3.-45);
        \draw (circ2) to[in=-90,out=90] (circ.-45);
        \draw (circ) to (top);
      \end{tikzpicture}
      \displaybreak[1]\\
      \underset{\eqref{eq:DefineLhd}}{=}
      \begin{tikzpicture}[smallstring, baseline=0]
        \node[box] (omega2) at (2.3,1) {${\omega'}_{a, v}^{-1}$};
        \node (bot) at (2.3,-3) {$\overline{v}$};
        \node[box] (kappa) at (2.3,-2) {$\kappa_{v \to 1}^{\cV}$};
        \node[box] (eta) at (2.3,-1) {$\eta_{a \blacktriangleleft v , v^*}'$};
        \node[box] (circ2) at (4,1) {$- \circ -$};
        \node[box] (circ3) at (5,0) {$- \circ -$};
        \node[box] (circ4) at (6,-1) {$- \circ -$};
        \node[box] (alpha) at (4,-1) {${\alpha_{a, v, v^*}'}^{-1}$};
        \node[box] (ev) at (5,-2) {$\sG(\id_a) \blacktriangleleft \ev_v$};
        \node[box] (omega3) at (7,-2) {${\omega'}_{a, 1_{\sA}}$};
        \node[box] (circ) at (3,2) {$- \circ -$};
        \node (top) at (3,3) {$\cA(a \lhd v \to a \lhd 1)$};
        \draw (bot) to (kappa);
        \draw (kappa) to (eta);
        \draw (eta) to[in=-90,out=90] (circ2.225);
        \draw (omega2) to[in=-90,out=90] (circ.225);
        \draw (circ3) to[in=-90,out=90] (circ2.-45);
        \draw (alpha) to[in=-90,out=90] (circ3.225);
        \draw (omega3) to[in=-90,out=90] (circ4.-45);
        \draw (ev) to[in=-90,out=90] (circ4.225);
        \draw (circ4) to[in=-90,out=90] (circ3.-45);
        \draw (circ2) to[in=-90,out=90] (circ.-45);
        \draw (circ) to (top);
      \end{tikzpicture}
      &\underset{\text{(Fact ~\ref{fact:MateOfLa})}}{=}
      \begin{tikzpicture}[smallstring, baseline=10]
        \node[box] (omega2) at (2.4,1) {${\omega'}_{a, v}^{-1}$};
        \node (bot) at (3,-2) {$\overline{v}$};
        \node[box] (kappa) at (3,-1) {$\kappa_{v \to 1}^{\cV}$};
        \node[box] (tri) at (3,0) {$(a \blacktriangleleft -)_{v \to 1}$};
        \node[box] (circ2) at (4,1) {$- \circ -$};
        \node[box] (omega3) at (5,0) {${\omega'}_{a, 1_{\sA}}$};
        \node[box] (circ) at (3.2,2) {$- \circ -$};
        \node (top) at (3.2,3) {$\cA(a \lhd v \to a \lhd 1)$};
        \draw (bot) to (kappa);
        \draw (kappa) to (tri);
        \draw (tri) to[in=-90,out=90] (circ2.225);
        \draw (omega2) to[in=-90,out=90] (circ.225);
        \draw (omega3) to[in=-90,out=90] (circ2.-45);
        \draw (circ2) to[in=-90,out=90] (circ.-45);
        \draw (circ) to (top);
      \end{tikzpicture}\, .
    \end{align*}
    Proposition \ref{prop:modulator-from-V-functor} tells us that this is equal to
    $\kappa_{v \to 1}^{\cV} \circ (a \lhd -)_{v \to 1}$;
    the result follows because $(a \lhd - )$ is dagger.
  \end{proof}

  %%%%%%%%%%%%%%%%%%%%%%%%%%%%%%%%%%%%%%%%%%%%%%%%%%%%%%%%%%%%%%%%%%%%%%%%%%%%%%%%%%%%%
  \subsection{2-Equivalence}

  The following lemma is used to show essential surjectivity of $\Phi$ on 1-cells.
  \begin{lem}
    \label{lem:identity mate}
    Let $\cA$ be a weakly tensored $\sV$-category with underlying category $\sA$, and $f \in \sV(1_\sV \to \cA(a \to b)) = \sA(a \to b)$. Then $f$ is equal to
    \begin{align*}
      \begin{tikzpicture}[smallstring, baseline=30]
        \node (bot) at (0,0) {$a$};
        \node (top) at (0,4) {$b$};
        \modulebox{0,0.5}{2}{0.75}{$\rho_a^{-1}$}
        \draw[rounded corners] (0.75,1.5) rectangle (1.75,2.25) node[midway] {$f$};
        \modulebox{0,2.5}{2}{0.75}{$\epsilon_{a \to b}$};
        \draw[blue] (bot) to (top);
        \draw[dotted] (1.25,1.25) to (1.25,1.5);
        \draw (1.25,2.25) to (1.25,2.5);
      \end{tikzpicture}
      \,.
    \end{align*}
  \end{lem}
  \begin{proof}
    By \ref{mate:L-composite} and \ref{mate:R-composite},
    $f \in \sV(1_\sV \to \cA(a \to b)) $ is the mate of both
    $\rho_a \circ f$ and $(\id_a \lhd f) \circ \epsilon_{a \to b}$ under the adjunction
    $\sA(a \lhd 1 \to b) \cong \sV(1_\sV \to \cA(a \to b)) = \sA(a \to b)$.
    The result follows by invertibility of $\rho_a$.
  \end{proof}

  \begin{thm}
    \label{thm:2-equivalence}
    The 2-functor $\Phi$ from Construction \ref{const:2-functor} is a 2-equivalence.
  \end{thm}

  \begin{proof}
    That $\Phi$ is essentially surjective on 0-cells is exactly the content of Proposition \ref{prop:equivalence-of-dagger-U-mods}.
    
    Next, take two 0-cells $\cA, \cB \in \dagucat$ and a 1-cell $(\sF, \omega^\sF) \in \vmod(\Phi(\cA) \to \Phi(\cB))$.
    We have a 1-cell $\cF$ in $\dagucat(\cA \to \cB)$ from Construction \ref{const:mod-to-cat-1-cell} and write $(\sF', \omega^{\sF'}) := \Phi(\cF)$;
    we show that $(\sF, \omega^\sF)$ and $(\sF', \omega^{\sF'})$ are equal.
    Note that this is very similar to the proof of \cite[Prop.~6.9]{2104.07747}, but the presentation here is simplified and more precise.
    By definition, $\sF'(a) = \sF(a)$ for all objects $a \in \Phi(\cA)$.
    On morphisms $f \in \Phi(\cA)(a \to b)$, $\sF'(f)$ is defined as $f \circ \cF_{a \to b}$, which by Lemma \ref{lem:identity mate} is equal to
    \begin{align*}
      \begin{tikzpicture}[smallstring,baseline=45]
        \node (a) at (0,0) {$\sF(a)$};
        \node (av) at (0,5) {$\sF(b)$};
        \modulebox{0,0.5}{2}{0.75}{$\rho_{\sF(a)}^{-1}$}
        \draw[rounded corners] (0.5,1.5) rectangle (2,2.25) node[midway] {$f$};
        \draw[rounded corners] (0.5,2.5) rectangle (2,3.25) node[midway] {$\cF_{a \to b}$};
        \modulebox{0,3.5}{2}{0.75}{$\epsilon_{\sF(a) \to \sF(b)}$}
        \draw[blue] (a) to (av);
        \draw[dotted] (1.25,1.25) to (1.25,1.5);
        \draw (1.25,2.25) to (1.25,2.5);
        \draw (1.25,3.25) to (1.25,3.5);
      \end{tikzpicture}
      =
      \begin{tikzpicture}[smallstring,baseline=45]
        \node (a) at (0,0) {$\sF(a)$};
        \node (av) at (0,5) {$\sF(b)$};
        \modulebox{0,0.5}{2}{0.75}{$\rho_{\sF(a)}^{-1}$}
        \draw[rounded corners] (0.5,1.5) rectangle (2,2.25) node[midway] {$f$};
        \modulebox{0,2.5}{2}{0.75}{$\omega_{a, \cA(a \to b)}^\sF$}
        \modulebox{0,3.5}{2}{0.75}{$\sF(\epsilon_{a \to b})$}
        \draw[blue] (a) to (av);
        \draw[dotted] (1.25,1.25) to (1.25,1.5);
        \draw (1.25,2.25) to (1.25,2.5);
        \draw (0.1,3.25) to (0.1,3.5);
      \end{tikzpicture}
      =
      \begin{tikzpicture}[smallstring,baseline=45]
        \node (a) at (0,0) {$\sF(a)$};
        \node (av) at (0,5) {$\sF(b)$};
        \modulebox{0,0.5}{2}{0.75}{$\rho_{\sF(a)}^{-1}$}
        \modulebox{0,1.5}{2}{0.75}{$\omega_{a, 1_\sA}^\sF$}
        \modulebox{0,2.5}{2}{0.75}{$\sF(\id_a \lhd f)$}
        \modulebox{0,3.5}{2}{0.75}{$\sF(\epsilon_{a \to b})$}
        \draw[blue] (a) to (av);
        \draw[dotted] (1.25,1.25) to (1.25,1.5);
        \draw (0.1,2.25) to (0.1,2.5);
        \draw (0.1,3.25) to (0.1,3.5);
      \end{tikzpicture}
      =
      \begin{tikzpicture}[smallstring,baseline=40]
        \node (a) at (0,0) {$\sF(a)$};
        \node (av) at (0,4) {$\sF(b)$};
        \modulebox{0,0.5}{2}{0.75}{$\sF(\rho_{a}^{-1})$}
        \modulebox{0,1.5}{2}{0.75}{$\sF(\id_a \lhd f)$}
        \modulebox{0,2.5}{2}{0.75}{$\sF(\epsilon_{a \to b})$}
        \draw[blue] (a) to (av);
        \draw[dotted] (1.25,1.25) to (1.25,1.5);
        \draw (0.1,2.25) to (0.1,2.5);
      \end{tikzpicture}\, ,
    \end{align*}
    and another use of Lemma \ref{lem:identity mate} shows this is equal to $\sF(f)$.
    Finally, to check the modulators are equal: $\omega_{a \to b}^{\sF'}$ is defined as the mate of $\eta_{a, v} \circ \cF_{a \to a \lhd v}$ under Adjunction \ref{eq:CatToModAdjunction}, which is equal to
    \begin{align*}
      \begin{tikzpicture}[smallstring,baseline=30]
        \node (a) at (0,0) {$\sF(a)$};
        \node (v) at (1.5,0) {$v$};
        \node (tri) at (0.75,0) {$\lhd$};
        \node (av) at (0,3) {$\sF(a \lhd v)$};
        \draw[rounded corners] (0.75,0.5) rectangle (2.25,1.25) node[midway] {$\eta_{a, v}$};
        \modulebox{0,1.5}{3}{0.75}{$\mate(\cF_{a \to a \lhd v})$}
        \draw[blue] (a) to (av);
        \draw (v) to (v |- 0,0.5);
        \draw (v |- 0,1.25) to (v |- 0,1.5);
      \end{tikzpicture}
      =
      \begin{tikzpicture}[smallstring,baseline=40]
        \node (a) at (0,0) {$\sF(a)$};
        \node (v) at (1.5,0) {$v$};
        \node (tri) at (0.75,0) {$\lhd$};
        \node (av) at (0,4) {$\sF(a \lhd v)$};
        \draw[rounded corners] (0.75,0.5) rectangle (2.25,1.25) node[midway] {$\eta_{a, v}$};
        \modulebox{0,1.5}{2.25}{0.75}{$\omega_{a, \cA(a \to a \lhd v)}^\sF$}
        \modulebox{0,2.5}{2.25}{0.75}{$\sF(\epsilon_{a \to a \lhd v})$}
        \draw[blue] (a) to (av);
        \draw (v) to (v |- 0,0.5);
        \draw (v |- 0,1.25) to (v |- 0,1.5);
        \draw (0.1,2.25) to (0.1,2.5);
      \end{tikzpicture}
      =
      \begin{tikzpicture}[smallstring,baseline=40]
        \node (a) at (0,0) {$\sF(a)$};
        \node (v) at (1.5,0) {$v$};
        \node (tri) at (0.75,0) {$\lhd$};
        \node (av) at (0,4) {$\sF(a \lhd v)$};
        \modulebox{0,0.5}{2.25}{0.75}{$\omega_{a, v}^\sF$}
        \modulebox{0,1.5}{2.25}{0.75}{$\sF(\id_a \lhd \eta_{a, v})$}
        \modulebox{0,2.5}{2.25}{0.75}{$\sF(\epsilon_{a \to a \lhd v})$}
        \draw[blue] (a) to (av);
        \draw (v) to (v |- 0,0.5);
        \draw (0.1,1.25) to (0.1,1.5);
        \draw (0.1,2.25) to (0.1,2.5);
      \end{tikzpicture}
      =
      \omega_{a, v}^\sF.
    \end{align*}

    Lastly, on 2-cells $\Phi$ takes a $\sU$-natural transformation
    $\theta \in \cU(1_{\sU} \to \cB(\cF(a) \to \sG(a)))$ to itself,
    viewed as a natural transformation in $\sB(\sF(a) \to \sG(a))$,
    so $\Phi$ is fully faithful on 2-cells.
  \end{proof}

  %%%%%%%%%%%%%%%%%%%%%%%%%%%%%%%%%%%%%%%%%%%%%%%
  %%%%%%%%%%%%%%%%%%%%%%%%%%%%%%%%%%%%%%%%%%%%%%%
  %%%%%%%%%%%%%%%%%%%%%%%%%%%%%%%%%%%%%%%%%%%%%%%
  \section{Characterization of unitary braided-enriched monoidal categories}

  In this section, we extend Theorem \ref{thm:EnrichedDaggerEquivalence} to the unitary braided-enriched setting to prove Theorem \ref{thm:EnrichedMonoidalDaggerEquivalence}.
  We first recall definitions from \cite{MR3961709} in the non-unitary setting.

  %%%%%%%%%%%%%%%%%%%%%%%%%%%%%%%%%%%%%%%%%%%%%%%
  \subsection{Characterization of braided-enriched monoidal categories}
  Suppose $\sV$ is a braided monoidal category.
  As before, we suppress tensor symbols, associators, and unitors whenever possible.

  %%%%%%%%%%%%%%%%%%%%%%%%%%%%%%%%%%%%%%%%%%%%%%%
  \subsubsection{Braided-enriched monoidal categories}

  \begin{defn}[\cite{MR3961709}]
    A \emph{(strict) $\sV$-monoidal category} is a $\sV$-category $\cA$ along with the additional structure of
    \begin{itemize}
    \item A tensor unit $1_\cA$,
    \item A tensor product on objects, that is, for all $a , b \in \cA$, an object $ab \in \cA$, and
    \item A distinguished tensor product morphism $- \otimes_\cA - \in \sV(\cA(a \to c) \cA(b \to d) \to \cA(ab \to cd))$ for all $a , b , c,d \in \cA$.
    \end{itemize}
    The above data must satisfy, for all $a , b , c , d , e , f \in \cA$:
    \begin{itemize}
    \item Tensor unit: $1_\cA a = a = a 1_\cA$,
    \item Object associativity: $(ab)c = a(bc)$,
    \item Unitality: $(j_{1_\cA} \id_{\cA(a \to b)}) \circ (- \otimes_\cA -) = \id_{\cA(a \to b)} = (\id_{\cA(a \to b)} j_{1_\cA}) \circ (- \otimes_\cA -)$,
    \item Associativity: $((- \otimes_\cA -) \id_{\cA(a \to b)}) \circ (- \otimes_\cA -) = \id_{\cA(a \to b)} = (\id_{\cA(a \to b)} (- \otimes_\cA -)) \circ (- \otimes_\cA -)$,
    \item Braided interchange:
      \begin{equation}
        \label{eq:BraidedInterchange}
        \begin{tikzpicture}[baseline=50, smallstring]
          \node (a-b) at (0,0) {$\mathcal A(a \to b)$};
          \node (d-e) at (2,0) {$\mathcal A(d \to e)$};
          \node (b-c) at (4,0) {$\mathcal A(b \to c)$};
          \node (e-f) at (6,0) {$\mathcal A(e \to f)$};
          \node[box] (t1) at (1,2) {$ -\otimes- $};
          \draw (a-b) to[in=-90,out=90] (t1.-135);
          \draw (d-e) to[in=-90,out=90] (t1.-45);
          \node[box] (t2) at (5,2) {$ -\otimes- $};
          \draw (b-c) to[in=-90,out=90] (t2.-135);
          \draw (e-f) to[in=-90,out=90] (t2.-45);
          \node[box] (c) at (3,4) {$ - \circ -  $};
          \draw (t1) to[in=-90,out=90] node[left=13pt] {$\mathcal A(ad\to be)$} (c.-135);
          \draw (t2) to[in=-90,out=90] node[right=13pt] {$\mathcal A(be\to cf)$} (c.-45);
          \node (r) at (3,5.5) {$\mathcal A(ad \to cf)$};
          \draw (c) -- (r);
        \end{tikzpicture}
        =
        \begin{tikzpicture}[baseline=50, smallstring]
          \node (a-b) at (0,0) {$\mathcal A(a \to b)$};
          \node (d-e) at (2,0) {$\mathcal A(d \to e)$};
          \node (b-c) at (4,0) {$\mathcal A(b \to c)$};
          \node (e-f) at (6,0) {$\mathcal A(e \to f)$};
          \node[box] (t1) at (1,2) {$ -\circ- $};
          \node[box] (t2) at (5,2) {$ -\circ- $};
          \draw (a-b) to[in=-90,out=90] (t1.-135);
          \draw (b-c) to[in=-90,out=90] (t1.-45);
          \draw[knot] (d-e) to[in=-90,out=90] (t2.-135);
          \draw (e-f) to[in=-90,out=90] (t2.-45);
          \node[box] (c) at (3,4) {$ - \otimes -  $};
          \draw (t1) to[in=-90,out=90] node[left=13pt] {$\mathcal A(a\to c)$} (c.-135);
          \draw (t2) to[in=-90,out=90] node[right=13pt] {$\mathcal A(d\to f)$} (c.-45);
          \node (r) at (3,5.5) {$\mathcal A(ad \to cf)$};
          \draw (c) -- (r);
      \end{tikzpicture}
      \,.
    \end{equation}
  \end{itemize}
\end{defn}

\begin{ex}[{\cite[\S2.3]{MR3961709}}]
  \label{ex:MonoidalSelfEnrichment}
  Suppose $\sV$ is a rigid monoidal category.
  The self-enrichment $\cV$ from Example \ref{ex:SelfEnrichment} has a $\cV$-monoidal structure given as follows, where we identify $w^*u^*=(uw)^*$:
  $$
  \tikzmath[smallstring]{
    \node (top) at (.75,3) {$\cV(uw\to vx)$};
    \node[box] (circ) at (.75,2) {$-\otimes_\cV-$};
    \node (eta) at (-.1,1) {$\cV(u\to v)$};
    \node (rho) at (1.6,1) {$\cV(w\to x)$};
    \draw (eta) to[in=-90,out=90] (circ.-135);
    \draw (rho) to[in=-90,out=90] (circ.-45);
    \draw (circ) to (top);
  }
  :=
  \tikzmath[smallstring]{
    \node (b1) at (0,0) {$u^*$};
    \node (b2) at (1,0) {$v$};
    \node (b3) at (2,0) {$w^*$};
    \node (b4) at (3,0) {$x$};
    \node (t1) at (0,2) {$w^*$};
    \node (t2) at (1,2) {$u^*$};
    \node (t3) at (2,2) {$v$};
    \node (t4) at (3,2) {$x$};
    \draw (b1) to (t2);
    \draw (b2) to (t3);
    \draw[knot] (b3) to (t1);
    \draw (b4) to (t4);
  }\,.
  $$
\end{ex}

\begin{defn}
  Given $\sV$-monoidal categories $\cA, \cB$, a \emph{$\sV$-monoidal functor} $\cF: \cA \to \cB$ is a $\sV$-functor $\cF$ along with a family of natural isomorphisms $\mu_{a , b} \in \sV(1_\sV \to \cB(\cF(a) \cF(b) \to \cF(ab)))$ such that
  \begin{align*}
    \begin{tikzpicture}[baseline=40, smallstring]
      \node (ac) at (2,0) {\small $\mathcal A ( a \to c )$};
      \node (bd) at (4,0) {\small $\mathcal A ( b \to d )$};
      \node (top) at (2,4) {\small $\mathcal B ( \cF ( a ) \cF ( b ) \to \cF ( c d ) )$};
      \node[box] (otimes) at (3,1) {$- \otimes_{\mathcal A} -$};
      \node[box] (r) at (3,2) {$\cF_{a b \to c d}$};
      \node[box] (rho) at (1,2) {$\mu_{a , b}$};
      \node[box] (circ) at (2,3) {$- \circ_{\mathcal B} -$};
      \draw (ac) to[in=-90,out=90] (otimes.225);
      \draw (bd) to[in=-90,out=90] (otimes.-45);
      \draw (otimes) to[in=-90,out=90] (r);
      \draw (r) to[in=-90,out=90] (circ.-45);
      \draw (rho) to[in=-90,out=90] (circ.225);
      \draw (circ) to[in=-90,out=90] (top);
    \end{tikzpicture}
    \hspace{-5pt}
    =
    \hspace{-5pt}
    \begin{tikzpicture}[baseline=40, smallstring]
      \node (ac) at (0,0) {\small $\mathcal A ( a \to c )$};
      \node (bd) at (2,0) {\small $\mathcal A ( b \to d )$};
      \node (top) at (2,4) {\small $\mathcal B ( \cF ( a ) \cF ( b ) \to \cF ( c d ) )$};
      \node[box] (rac) at (0,1) {$\cF_{a \to c}$};
      \node[box] (rbd) at (2,1) {$\cF_{b \to d}$};
      \node[box] (otimes) at (1,2) {$- \otimes_{\mathcal B} -$};
      \node[box] (rho) at (3,2) {$\mu_{c , d}$};
      \node[box] (circ) at (2,3) {$- \circ_{\mathcal B} -$};
      \draw (ac) to[in=-90,out=90] (rac);
      \draw (bd) to[in=-90,out=90] (rbd);
      \draw (rac) to[in=-90,out=90] (otimes.225);
      \draw (rbd) to[in=-90,out=90] (otimes.-45);
      \draw (otimes) to[in=-90,out=90] (circ.225);
      \draw (rho) to[in=-90,out=90] (circ.-45);
      \draw (circ) to[in=-90,out=90] (top);
    \end{tikzpicture}
    \text{\quad and \quad}
    \begin{tikzpicture}[baseline=25, smallstring]
      \node (top) at (2,3) {\small $\mathcal B ( \cF ( a ) \cF ( b ) \cF ( c ) \to \cF ( a b c ) )$};
      \node[box] (j) at (0,0) {$j_{\cF ( a )}^{\mathcal B}$};
      \node[box] (rho1) at (2,0) {$\mu_{b , c}$};
      \node[box] (otimes) at (1,1) {$- \otimes_{\mathcal B} -$};
      \node[box] (rho2) at (3,1) {$\mu_{a , b c}$};
      \node[box] (circ) at (2,2) {$- \circ_{\mathcal B} -$};
      \draw (j) to[in=-90,out=90] (otimes.225);
      \draw (rho1) to[in=-90,out=90] (otimes.-45);
      \draw (otimes) to[in=-90,out=90] (circ.225);
      \draw (rho2) to[in=-90,out=90] (circ.-45);
      \draw (circ) to[in=-90,out=90] (top);
    \end{tikzpicture}
    =
    \begin{tikzpicture}[baseline=25, smallstring]
      \node (top) at (2,3) {\small $\mathcal B ( \cF ( a ) \cF ( b ) \cF ( c ) \to \cF ( a b c ) )$};
      \node[box] (j) at (2,0) {$j_{\cF ( c )}^{\mathcal B}$};
      \node[box] (rho1) at (0,0) {$\mu_{a , b}$};
      \node[box] (otimes) at (1,1) {$- \otimes_{\mathcal B} -$};
      \node[box] (rho2) at (3,1) {$\mu_{a b , c}$};
      \node[box] (circ) at (2,2) {$- \circ_{\mathcal B} -$};
      \draw (j) to[in=-90,out=90] (otimes.-45);
      \draw (rho1) to[in=-90,out=90] (otimes.225);
      \draw (otimes) to[in=-90,out=90] (circ.225);
      \draw (rho2) to[in=-90,out=90] (circ.-45);
      \draw (circ) to[in=-90,out=90] (top);
    \end{tikzpicture}\, .
  \end{align*}
  For convenience, we further assume that all $\sV$-monoidal functors are \emph{strictly unital}, i.e., $\cF(1_\cA) = 1_\cB$, $j_{1_\cA}^\cA \circ \cF_{1_\cA \to 1_\cA} = j_{1_\cB}^\cB$, and $\mu_{1_\cA, a} = \id_a = \mu_{a, 1_\cA}$. 
  Given two $\sV$-monoidal functors $\cF: \cA\to \cB$ and $\cG: \cB\to \cC$,
  the composite laxitor $[\mu_{\cF(a), \cF(b)}^{\cG} (\mu_{a, b}^{\cF} \circ \cG_{\cF(a) \cF(b) \to \cF(ab)})] \circ (- \circ_\cC -)$ gives the composite $\cF \circ \cG$ of $\sV$-monoidal functors the structure of a $\sV$-monoidal functor.
  
  A \emph{$\sV$-monoidal natural transformation} is a $\sV$-natural transformation $\theta: \cF \Rightarrow \cG$ such that
  \begin{align*}
    \begin{tikzpicture}[baseline=15, smallstring]
      \node (top) at (1,2) {$\mathcal B ( \cF ( a ) \cF ( b ) \to \cG ( a b ) )$};
      \node[box] (rho) at (0,0) {$\mu_{a , b}^F$};
      \node[box] (theta) at (2,0) {$\theta_{a b}$};
      \node[box] (circ) at (1,1) {$- \circ_{\mathcal B} -$};
      \draw (rho) to[in=-90,out=90] (circ.225);
      \draw (theta) to[in=-90,out=90] (circ.-45);
      \draw (circ) to[in=-90,out=90] (top);
    \end{tikzpicture}
    =
    \begin{tikzpicture}[baseline=30, smallstring]
      \node(top) at (2,3) {$\mathcal B ( \cF ( a ) \cF ( b ) \to \cG ( a b ) )$};
      \node[box] (eta) at (0,0) {$\theta_a$};
      \node[box] (etb) at (2,0) {$\theta_b$};
      \node[box] (tens) at (1,1) {$- \otimes_{\mathcal B} -$};
      \node[box] (sigma) at (3,1) {$\mu_{a , b}^G$};
      \node[box] (circ) at (2,2) {$- \circ_{\mathcal B} -$};
      \draw (eta) to[in=-90,out=90] (tens.225);
      \draw (etb) to[in=-90,out=90] (tens.-45);
      \draw (tens) to[in=-90,out=90] (circ.225);
      \draw (sigma) to[in=-90,out=90] (circ.-45);
      \draw (circ) to[in=-90,out=90] (top);
    \end{tikzpicture}.
  \end{align*}
\end{defn}

\begin{defn}
  Given a $\sV$-monoidal category $\cA$, the \emph{underlying monoidal category} $\sA = \cA^\sV$ is the underlying category equipped with the same tensor product on objects, and the tensor product $f \otimes g := (f g) \circ (- \otimes_{\cA} -)$ on morphisms. 
  We call a $\sV$-monoidal category $\cA$ \emph{rigid} if the underlying monoidal category $\sA$ is rigid.

  Given a $\sV$-monoidal functor $\cF : \cA \to \cB$, we obtain the \emph{underlying monoidal functor} $\sF = \cF^\sV$ of $\cF$ by equipping the underlying functor with the tensorator $\mu_{a,b}^\sV := \mu_{a,b} \in \sV(1_{\sV} \to \cB(\cF(ab) \to \cF(a) \cF(b))) = \sB(\cF(ab) \to \cF(a) \cF(b))$. 
  Under this definition, $\sF$ is an ordinary monoidal functor.

  Similarly, a $\sV$-monoidal natural transformation $\theta: \cF\Rightarrow \cG$ gives an ordinary monoidal natural transformation $\sF\Rightarrow \sG$.
\end{defn}

%%%%%%%%%%%%%%%%%%%%%%%%%%%%%%%%%%%%%%%%%%%%%%%
\subsubsection{Module monoidal categories}

The definitions in this section are originally from \cite[\S3.2]{MR3578212} under the name of module tensor category.

\begin{defn}
  A \emph{$\sV$-module monoidal category} is a pair $(\sA, \sF_\sA^Z)$ consisting of a monoidal category $\sA$ and a braided monoidal functor $(\sF_\sA^Z, \mu^\sA): \sV \to Z(\sA)$. 
  We denote by $\sF_\sA:= \sF_\sA^Z \circ \Forget_Z$ where $\Forget_Z: Z(\sA)\to \sA$ is the forgetful functor.
  We call a $\sV$-module monoidal category $(\sA, \sF_\sA^Z)$:
  \begin{itemize}
  \item \emph{rigid} if $\sA$ is rigid, and
    %  \item \emph{weakly tensored} if there is a right adjoint to $\cF := \cF^Z \circ \Forget_Z$, where $\Forget_Z : Z(A) \to A$ is the forgetful functor,
  \item \emph{tensored} if $\sF_\sA:\sV \to \sA$ admits a right adjoint.
  \end{itemize}
\end{defn}

\begin{defn}
  \label{defn:vmodmon 1-cells}
  A 1-morphism between $\sV$-module monoidal categories is a monoidal functor $(\sH, \mu^\sH) : \sA \to \sB$ equipped with an \emph{action coherence} natural transformation $h : \sF_\sB \Rightarrow \sF_\sA \circ \sH$ that is compatible with $\mu^\sA$, $\mu^\sB$, and the half-braidings $e_{a, \sF(v)}$.
  That is, we require
  \begin{align*}
    \begin{tikzpicture}[baseline=40, smallstring]
      \node (top) at (1,4) {$\sH ( \sF_A ( v ) a )$};
      \node (ra) at (0.25,0) {$\sH ( a )$};
      \node (v) at (1.75,0) {$\sF_B ( v )$};
      \node[box] (r) at (0.25,2) {$h_v$};
      \node[box] (rho) at (1,3) {$\mu^\sH_{\sF_A ( v ) , a}$};
      \node[box] (e) at (1,1) {$e_{\sH ( a ) , \sF_B ( v )}$};
      \draw (v) to[in=-90,out=90] (e.-45);
      \draw (ra) to[in=-90,out=90] (e.225);
      \draw (e.135) to[in=-90,out=90] (r);
      \draw (e.45) to (e.-45 |- rho.-45);
      \draw (r) to[in=-90,out=90] (rho.225);
      \draw (rho) to (top);
    \end{tikzpicture}
    =
    \begin{tikzpicture}[baseline=40, smallstring]
      \node (top) at (1,4) {$\sH ( \sF_A ( v ) a )$};
      \node (ra) at (0.25,0) {$\sH ( a )$};
      \node (v) at (1.75,0) {$\sF_B ( v )$};
      \node[box] (r) at (1.75,1) {$h_v$};
      \node[box] (rho) at (1,2) {$\mu^\sH_{a , \sF_A ( v )}$};
      \node[box] (e) at (1,3) {$\sH ( e_{a , \sF_A ( v )} )$};
      \draw (v) to[in=-90,out=90] (r);
      \draw (ra) to[in=-90,out=90] (rho.225);
      \draw (r) to[in=-90,out=90] (rho.-45);
      \draw (rho) to (e);
      \draw (e) to (top);
    \end{tikzpicture}
    \hspace{3em}
    \text{ and }
    \hspace{3em}
    \begin{tikzpicture}[baseline=40, smallstring]
      \node (top) at (0,3) {$\sH ( \sF_A ( u ) \sF_A ( v ) )$};
      \node (bot) at (0,0) {$\sF_B ( u v )$};
      \node[box] (r) at (0,1) {$h_{uv}$};
      \node[box] (mu) at (0,2) {$\sH ( \mu_{u , v}^{A} )$};
      \draw (bot) to (r);
      \draw (r) to (mu);
      \draw (mu) to (top);
    \end{tikzpicture}
    =
    \begin{tikzpicture}[baseline=45, smallstring]
      \node (top) at (1,4) {$\sH ( \sF_A ( u ) \sF_A ( v ) )$};
      \node (bot) at (1,0) {$\sF_B ( u v )$};
      \node[box] (mu) at (1,1) {$\,\,\mu_{u , v}^{B}\,\,$};
      \node[box] (r1) at (0.25,2) {$h_u$};
      \node[box] (r2) at (1.75,2) {$h_v$};
      \node[box] (rho) at (1,3) {$\mu^\sH_{\sF_A ( u ) , \sF_A ( v )}$};
      \draw (bot) to (mu);
      \draw (mu.135) to[in=-90,out=90] (r1);
      \draw (mu.45) to[in=-90,out=90] (r2);
      \draw (r1) to[in=-90,out=90] (rho.225);
      \draw (r2) to[in=-90,out=90] (rho.-45);
      \draw (rho) to (top);
    \end{tikzpicture}
    \,.
  \end{align*}
  We define 2-morphisms between such 1-morphisms these as monoidal natural transformations $\Theta: (\sH, \mu^\sH, h) \Rightarrow (\sH', \mu^{\sH'}, h')$ that are compatible with the action coherence transformations, that is, $h_v \circ \Theta_{\sF_A (v)} = h_v'$.
\end{defn}

With these definitions, \cite{2104.07747, 2104.03121} show the following result.
The details on the level of objects appear in \cite{MR3961709}, with \cite{1809.09782} providing some simplifications.

\begin{thm}[{\cite{MR3961709,1809.09782,2104.07747, 2104.03121}}]
  Let $\sV$ be a closed braided monoidal category.
  There is an equivalence of 2-categories
  \[
 \left\{\,
  \parbox{3.3cm}{\rm Tensored rigid $\sV$-monoidal categories}
  \,\right\}
  \,\,\cong\,\,
  \left\{\,
  \parbox{4.4cm}{\rm Tensored rigid $\sV$-module monoidal categories}
  \,\right\}.
  \]
\end{thm}

%Theorem \ref{thm:EnrichedMonoidalEquivalence} extends to a 2-equivalence of 2-categories.
%
%\begin{thm}[{\cite{2104.07747, 2104.03121}}]
%  Let $\sV$ be a closed braided monoidal category.
%  There is an equivalence of 2-categories
%  \[
% \left\{\,
%  \parbox{3.6cm}{\rm Tensored rigid $\sV$-monoidal categories}
%  \,\right\}
%  \,\,\cong\,\,
%  \left\{\,
%  \parbox{4.4cm}{\rm Tensored rigid $\sV$-module monoidal categories}
%  \,\right\}.
%  \]
%\end{thm}

In \cite{2104.07747}, the 2-equivalence is also carried to the setting of $G$-graded $\sV$-categories and $G$-extensions of $\sV$-categories which were studied in \cite{1910.03178},
while in \cite{2104.03121} changes in the enriching category were considered.

%%%%%%%%%%%%%%%%%%%%%%%%%%%%%%%%%%%%%%%%%%%%%%%
\subsection{Dagger braided-enriched monoidal categories and dagger module monoidal categories}
\label{sec:BraidedEnrichedDaggerStuff}

We now require that our enriching category $\sV$ be both braided and involutive. 
Naturally, this leads to a compatibility condition between the braiding $\beta$ and the involutive structure $\nu$.

\begin{defn}
  A braided involutive category $\sV$ is a braided monoidal category equipped with an involutive structure $(\overline{\cdot},\nu, \varphi,r)$ as in Definition \ref{defn:involutive linear category} such that
  \begin{align}\label{eqn:braided involutive}
    \begin{tikzpicture}[smallstring, baseline=25]
      \node (a) at (0.5,0) {$\overline{u}$};
      \node (b) at (1.5,0) {$\overline{v}$};
      \node (top) at (1,3) {$\overline{uv}$};
      \node[box] (beta) at (1,1) {$\beta_{\overline{u},\overline{v}}$};
      \node[box] (nu) at (1,2) {$\nu_{v,u}$};
      \draw (a) to[in=-90,out=90] (beta.245);
      \draw (b) to[in=-90,out=90] (beta.-65);
      \draw (beta) to (nu);
      \draw (nu) to (top);
    \end{tikzpicture}
    =
    \begin{tikzpicture}[smallstring, baseline=25]
      \node (a) at (0.5,0) {$\overline{u}$};
      \node (b) at (1.5,0) {$\overline{v}$};
      \node (top) at (1,3) {$\overline{uv}$};
      \node[box] (nu) at (1,1) {$\nu_{u,v}$};
      \node[box] (beta) at (1,2) {$\overline{\beta_{u,v}^{-1}}$};
      \draw (a) to[in=-90,out=90] (nu.245);
      \draw (b) to[in=-90,out=90] (nu.-65);
      \draw (nu) to (beta);
      \draw (beta) to (top);
    \end{tikzpicture}
  \,.
  \end{align}
\end{defn}
\begin{defn}[Conjugate of a monoidal $\sV$-category]
  Given a monoidal $\sV$-category $\cA$, we can equip the $\sV$-category $\overline{\cA}$ with a monoidal structure by defining
  \begin{align*}
    - \otimes_{\overline{\cA}} - := \beta_{\overline{\cA(c \to a)}, \overline{\cA(d \to b)}}^{-1} \circ \nu_{\cA(d \to b), \cA(c \to a)} \circ (\overline{- \otimes -}) \in \overline{\cA}(a \to c) \overline{\cA} (b \to d) \to \overline{\cA} (ab \to cd)
  \end{align*}
  with the same tensor unit, $1_{\overline{\cA}} := 1_\cA$. It is straightforward to check that this gives $\overline{\cA}$ the structure of a $\sV$-monoidal category.
\end{defn}

\begin{defn} \label{def:involutive monoidal}
  An \emph{weak dagger $\sV$-monoidal category} is a $\sV$-monoidal category $\cA$ with an involutive structure $\kappa$ as in Definition \ref{defn:involutive U-cat} such that $\kappa$ is compatible with $- \otimes -$, i.e., $(- \otimes_{\overline{\cA}} -) \circ \kappa = (\kappa \kappa) \circ (- \otimes_\cA -)$:
  \begin{enumerate}[label=($\kappa$\arabic*)]
    \setcounter{enumi}{4}
  \item
    \label{kappa:5}
    $
    \begin{tikzpicture}[smallstring, baseline=35]
      \node (top) at (1,3) {$\cA(ab \to cd)$};
      \node (ac) at (0,0) {$\overline{\cA(c \to a)}$};
      \node (bd) at (2,0) {$\overline{\cA(d \to b)}$};
      \node[box] (kappa) at (0,1) {$\kappa_{a \to c}$};
      \node[box] (kappa2) at (2,1) {$\kappa_{b \to d}$};
      \node[box] (tens) at (1,2) {$- \otimes -$};
      \draw (ac) to (kappa);
      \draw (bd) to (kappa2);
      \draw (kappa) to[in=-90,out=90] (tens.225);
      \draw (kappa2) to[in=-90,out=90] (tens.-45);
      \draw (tens) to (top);
    \end{tikzpicture}
    =
    \begin{tikzpicture}[smallstring, baseline=60]
      \node (top) at (1,5) {$\cA(ab \to cd)$};
      \node (ac) at (0,0) {$\overline{\cA(c \to a)}$};
      \node (bd) at (2,0) {$\overline{\cA(d \to b)}$};
      \node[box] (nu) at (1,2) {$\nu_{\cA(d \to b), \cA(c \to a)}$};
      \node[box] (tens) at (1,3) {$\overline{- \otimes -}$};
      \node[box] (kappa) at (1,4) {$\kappa_{ab \to cd}$};
      \draw (ac) to[in=-90,out=90] (nu.-30);
      \draw[knot] (bd) to[in=-90,out=90] (nu.210);
      \draw (nu) to (tens);
      \draw (tens) to (kappa);
      \draw (kappa) to (top);
    \end{tikzpicture}\, .
    $
  \end{enumerate}
  We call $\cA$ a \emph{dagger $\cV$-monoidal category} if in addition $\cA$ satisfies \ref{kappa:3} and \ref{kappa:4}, so that $\cA$ satisfies all of \ref{kappa:1}--\ref{kappa:5}.
\end{defn} 

\begin{rem}
  One could view such a $\kappa : \overline{\cA} \to \cA$, when equipped with the identity tensorator $\mu_{a, b}^{\overline{\cA}} := j_{ab}$, as a $\sV$-monoidal functor that is the identity on objects.
  As in Remark \ref{rem:principle of equivalence}, such a structure violates the principle of equivalence.
\end{rem}

\begin{ex}
  Let $\sU$ be a braided unitary monoidal category, i.e., a unitary monoidal category equipped with a unitary braiding.
  The dagger self-enrichment $(\cU,\kappa)$ from Example \ref{ex:InvolutiveSelfEnrichment} equipped with the $\sU$-monoidal structure from Example \ref{ex:MonoidalSelfEnrichment} satisfies \ref{kappa:5}.
\end{ex}

\begin{defn}
\label{def:DaggerVMonoidalFunctor}
A $\sV$-monoidal functor $(\cG,\mu^\cG):\cA\to \cB$ between dagger $\sV$-monoidal categories is called a \emph{dagger $\sV$-monoidal functor} if $\cG$ is a dagger $\sV$-functor and $\overline{\mu^\cG_{a,b}}\circ \kappa_{\cG(ab)\to \cG(a)\to \cG(b)} = (\mu^\cG_{a,b})^{-1}$.
\end{defn}

\begin{defn}
  Let $\sV$ be a rigid involutive monoidal category.
  We define a 2-category $\dagvmoncat$ with:
  \begin{itemize}
  \item 0-cells dagger $\sV$-monoidal categories,
  \item 1-cells dagger $\sV$-monoidal functors with their usual composition, and
  \item 2-cells $\sV$-monoidal natural transformations with their usual horizontal and vertical compositions.
  \end{itemize}
\end{defn}

\begin{defn}
  A \emph{dagger $\sV$-module monoidal category} is a pair $(\sA , \sF_\sA^Z)$ where
  $\sA$ is a dagger tensor category and 
  $\sF^Z_\sA: \sV \to Z^\dag(\sA)$ is a braided monoidal dagger functor 
  which equips $\sA$ with the structure of a tensored $\sV$-module monoidal category.

  Here, $Z^\dag(A)$ denotes the \emph{dagger Drinfeld center} of $\sA$ whose objects are pairs $(a,e_a)$ where $a\in \sA$ and $e_a=\{e_{a,b}: b\otimes a \to a\otimes b\}_{b\in \sA}$ is a unitary half-braiding for $a$ \cite[Def.~6.1]{MR1966525}. 

  We call $(\sA , \sF_\sA^Z)$ rigid/tensored if the underlying $\sV$-module monoidal category is rigid/tensored respectively.
  %Finally, we call $(\sA , \sF_\sA^Z)$ a $\sV$-\emph{module unitary monoidal category} if $\sA$ is a unitary monoidal category.
\end{defn}

\begin{defn}
  We define a 2-category $\dagvmodmon$ with:
  \begin{itemize}
  \item 0-cells tensored dagger $\sV$-module monoidal categories,
  \item 1-cells dagger $\sV$-module monoidal functors with the usual composition,
    i.e., 1-morphisms $(\sH,\mu^\sH)$ as in Definition \ref{defn:vmodmon 1-cells} such that ($\sH,\mu^\sH):\sA\to \sB$ is a dagger monoidal functor and $h:\sF_\sB\Rightarrow \sF_\sA \circ \sH$ is a unitary natural isomorphism, and
  \item 2-cells monoidal natural transformations as in Definition \ref{defn:vmodmon 1-cells} with the usual horizontal and vertical compositions.
  \end{itemize}
\end{defn}

%%%%%%%%%%%%%%%%%%%%%%%%%%%%%%%%%%%%%%%%%%%%%%%
\subsection{2-functor \texorpdfstring{$\dagvmoncat$}{DagVMonCat} to \texorpdfstring{$\dagvmodmon$}{DagVModMon}}

For this subsection, we assume $\sV$ is a rigid braided involutive monoidal category, and the self-enrichment $\cV$ is equipped with a daggger structure $\kappa^\cV$.

\begin{construction}
  \label{const:dagger-U-mon-cat-to-dagger-mtc}
Suppose we have a rigid dagger $\sV$-monoidal category $\cA$ such that $a \mapsto \cA(1 \to a)$ admits a left adjoint, denoted by $\sF : \sV \to \sA$. 
We denote the unit of this adjunction by $\eta_v\in \sV(v\to \cA(1\to \sF(v)))$.
Then $\sA$ is endowed with a right $\sV$-module structure given by $a \lhd u := a \sF(u)$.
Moreover, for all $a\in \sA$ and $u\in \sV$, we have an adjunction \cite[(4.1)]{MR3961709}:
\begin{equation}\label{adj:mon-main}
  \sA(a \sF(v) \to b) \cong \sV(v \to \cA(a \to b)).
\end{equation}
We denote the unit of this adjunction by $\eta_{a,v}\in \sV(v\to \cA(a\to a\sF(v)))$.
In \cite[Lem.~4.4]{MR3961709}, it was shown how to equip $\sF$ with a tensorator $\mu_{u , v} \in \sA(\sF(uv) \to \sF(u) \sF(v))$ defined as the mate of $(\eta_u \eta_v) \circ (- \otimes -)$ under Adjunction \ref{adj:mon-main} to obtain an oplax monoidal functor.
This tensorator is an isomorphism for all $u,v\in \sV$ if and only if $\cA(1\to -)$ admits a right $\sV$-adjoint $\cF: \cV\to \cA$ such that $\sF$ is the underlying functor of $\cF$ \cite[Cor.~7.3]{1809.09782}.
We assume this in the sequel.

By \cite[\S5.1]{MR3961709}, $\sF$ can be lifted to a braided monoidal functor $\sF^Z : \sV \to Z(\sA)$ by defining
the half braiding $e_{a, \sF(v)} \in \sA(a \sF(v) \to \sF(v) a)$ as the mate of $(\eta_v j_a) \circ (- \otimes -)$ under Adjunction \ref{adj:mon-main},
and each $e_{a, \sF(v)}$ is invertible.\footnote{
  In \cite{1809.09782}, it was incorrectly claimed that for a closed $\sV$-monoidal category $\cA$, the half-braidings $e_{a, \sF(v)}$ are invertible.
  The proof there says that this result is similar to \cite[Lem.~5.2]{MR3961709} which uses rigidity in an essential way.
  This is similar to how a lax module functor between modules for a rigid monoidal category is automatically strong \cite[Lem.~2.10]{MR3934626}.
  The $\sV$-monoidal results of \cite{1809.09782} can be amended by one of the two options below.
  (1) 
  Only use (oplax) monoidal functors to $Z^{\rm lax}(\sA)$, the lax center of the underlying monoidal category whose half-braidings are not required to be invertible.
  (2)
  Use rigid $\sV$-monoidal categories instead of closed $\sV$-monoidal categories.
  In this article, we stay in the rigid setting, thus avoiding this issue.
}
\end{construction}
Several proofs were omitted in \cite{MR3961709}:
\begin{itemize}
\item
  The proof that for every $f\in \sV(u\to v)$, $\sF(f)\in Z(\sA)(\sF(u)\to \sF(v))$ was omitted.
  This proof appears in \cite[Rem.~5.1]{2104.07747}.
\item
  The proof that the tensorator of $\sF^Z$ actually lives in $Z(\sA)$ was omitted.
  This proof is supplied in the next lemma below.
\end{itemize}

\begin{lem}
  The tensorator $\mu_{u,v}: \sF(u)\sF(v)\to \sF(uv)$ is a morphism in $Z(\sA)$.
  In diagrams,
  \begin{equation}
    \label{lem:reverse hexagon}
    \begin{tikzpicture}[smallstring, baseline=35]
      \node (a-bot) at (0,0) {$a$};
      \node (v-bot) at (1.5,0) {$\sF(u v)$};
      \node (a-top) at (2,4) {$a$};
      \node (v-top) at (0,4) {$\sF(u)$};
      \node (vbar-top) at (1,4) {$\sF(v)$};
      \node[box] (e1) at (0.5,2) {$e_{a, \sF(u)}$};
      \node[box] (e2) at (1.5,3) {$e_{a, \sF(v)}$};
      \node[box] (mu) at (1.5,1) {$\, \, \mu_{u, v} \, \,$};
      \draw (a-bot) to[in=-90,out=90] (e1.225);
      \draw[double] (v-bot) to[in=-90,out=90] (mu);
      \draw (mu.135) to[in=-90,out=90] (e1.-45);
      \draw (mu.45) to (mu.45 |- e2.-45);
      \draw (e1.135) to[in=-90,out=90] (v-top);
      \draw (e1.45) to[in=-90,out=90] (e2.225);
      \draw (e2.135) to[in=-90,out=90] (vbar-top);
      \draw (e2.45) to[in=-90,out=90] (a-top);
    \end{tikzpicture}
    =
    \begin{tikzpicture}[smallstring, baseline=35]
      \node (a-bot) at (0,0) {$a$};
      \node (v-bot) at (2,0) {$\sF(u v)$};
      \node (a-top) at (2,4) {$a$};
      \node (v-top) at (0,4) {$\sF(u)$};
      \node (vbar-top) at (1,4) {$\sF(v)$};
      \node[box] (e) at (1,1.5) {$e_{a, \sF(u v)}$};
      \node[box] (mu2) at (0.5,2.75) {$\mu_{u, v}$};
      \draw (a-bot) to[in=-90,out=90] (e.225);
      \draw [double] (v-bot) to[in=-90,out=90] (e.-45);
      \draw[double] (e.135) to[in=-90,out=90] (mu2);
      \draw (mu2.135) to[in=-90,out=90] (v-top);
      \draw (mu2.45) to[in=-90,out=90] (vbar-top);
      \draw (e.45) to[in=-90,out=90] (a-top);
    \end{tikzpicture}
    \,.
  \end{equation}
\end{lem}
\begin{proof}
  On the left hand side, taking mates gives:
  \begin{align*}
    &\begin{tikzpicture}[smallstring, baseline=60]
       \node (v) at (0,0) {$u$};
       \node (v-bar) at (1,0) {$v$};
       \node (top) at (3.0625,6.5) {$\cA(a \to \sF(u) \sF(v) a)$};
       \node[box] (j1) at (-0.5,2) {$j_a$};
       \node[box] (eta1) at (0,1) {$\eta_u$};
       \node[box] (eta2) at (1,1) {$\eta_{v}$};
       \node[box] (e1) at (2,2) {$e_{a, \sF(u)}$};
       \node[box] (j2) at (3.5,2) {$j_{\sF(v)}$};
       \node[box] (j3) at (4,3) {$j_{\sF(u)}$};
       \node[box] (e2) at (5.5,3) {$e_{a, \sF(v)}$};
       \node[box] (otimes0) at (0.5,2) {$- \otimes -$};
       \node[box] (otimes1) at (0,3) {$- \otimes -$};
       \node[box] (otimes2) at (2.75,3) {$- \otimes -$};
       \node[box] (otimes3) at (4.75,4.25) {$- \otimes -$};
       \node[box] (circ1) at (1.375,4.25) {$- \circ -$};
       \node[box] (circ2) at (3.0625,5.5) {$- \circ -$};
       \draw (v) to (eta1);
       \draw (v-bar) to (eta2);
       \draw (j1) to[in=-90,out=90] (otimes1.225);
       \draw (eta1) to[in=-90,out=90] (otimes0.225);
       \draw (eta2) to[in=-90,out=90] (otimes0.-45);
       \draw (otimes0) to[in=-90,out=90] (otimes1.-45);
       \draw (e1) to[in=-90,out=90] (otimes2.225);
       \draw (j2) to[in=-90,out=90] (otimes2.-45);
       \draw (j3) to[in=-90,out=90] (otimes3.225);
       \draw (e2) to[in=-90,out=90] (otimes3.-45);
       \draw (otimes1) to[in=-90,out=90] (circ1.225);
       \draw (otimes2) to[in=-90,out=90] (circ1.-45);
       \draw (circ1) to[in=-90,out=90] (circ2.225);
       \draw (otimes3) to[in=-90,out=90] (circ2.-45);
       \draw (circ2) to (top);
     \end{tikzpicture}
    \underset{\eqref{eq:BraidedInterchange}}{=}
    \begin{tikzpicture}[smallstring, baseline=40]
      \node (v) at (0.25,-1) {$u$};
      \node (v-bar) at (2.35,-1) {$v$};
      \node (top) at (2.25,5) {$\cA(a \to \sF(u) \sF(v) a)$};
      \node[box] (j1) at (-0.75,0) {$j_a$};
      \node[box] (eta1) at (0.25,0) {$\eta_u$};
      \node[box] (e1) at (1.25,1) {$e_{a, \sF(u)}$};
      \node[box] (eta2) at (2.35,1) {$\eta_{v}$};
      \node[box] (j2) at (2.75,2) {$j_{\sF(u)}$};
      \node[box] (e2) at (4.25,2) {$e_{a, \sF(v)}$};
      \node[box] (otimes0) at (-0.25,1) {$- \otimes -$};
      \node[box] (circ1) at (0.5,2) {$- \circ -$};
      \node[box] (otimes1) at (1,3) {$- \otimes -$};
      \node[box] (otimes2) at (3.5,3) {$- \otimes -$};
      \node[box] (circ2) at (2.25,4.1) {$- \circ -$};
      \draw (v) to (eta1);
      \draw (j1) to[in=-90,out=90] (otimes0.225);
      \draw (eta1) to[in=-90,out=90] (otimes0.-45);
      \draw (otimes0) to[in=-90,out=90] (circ1.225);
      \draw (e1) to[in=-90,out=90] (circ1.-45);
      \draw (v-bar) to (eta2);
      \draw (circ1) to[in=-90,out=90] (otimes1.225);
      \draw (eta2) to[in=-90,out=90] (otimes1.-45);
      \draw (otimes1) to[in=-90,out=90] (circ2.225);
      \draw (j2) to[in=-90,out=90] (otimes2.225);
      \draw (e2) to[in=-90,out=90] (otimes2.-45);
      \draw (otimes2) to[in=-90,out=90] (circ2.-45);
      \draw (circ2) to (top);
    \end{tikzpicture}
    \underset{\text{\cite[Cor.~4.7]{MR3961709}}}{=}
    \begin{tikzpicture}[smallstring, baseline=45]
      \node (v) at (0.1,0) {$u$};
      \node (v-bar) at (2,0) {$v$};
      \node (top) at (2.15,5) {$\cA(a \to \sF(u) \sF(v) a)$};
      \node[box] (j1) at (0.9,1) {$j_a$};
      \node[box] (eta1) at (0.1,1) {$\eta_u$};
      \node[box] (eta2) at (2,1) {$\eta_{v}$};
      \node[box] (j2) at (2.55,2) {$j_{\sF(u)}$};
      \node[box] (e2) at (3.85,2) {$e_{a, \sF(v)}$};
      \node[box] (otimes0) at (0.5,2) {$- \otimes -$};
      \node[box] (otimes1) at (1,3) {$- \otimes -$};
      \node[box] (otimes2) at (3.2,3) {$- \otimes -$};
      \node[box] (circ) at (2.15,4.1) {$- \circ -$};
      \draw (v) to (eta1);
      \draw (j1) to[in=-90,out=90] (otimes0.-45);
      \draw (eta1) to[in=-90,out=90] (otimes0.225);
      \draw (otimes0) to[in=-90,out=90] (otimes1.225);
      \draw (v-bar) to (eta2);
      \draw (otimes0) to[in=-90,out=90] (otimes1.225);
      \draw (eta2) to[in=-90,out=90] (otimes1.-45);
      \draw (otimes1) to[in=-90,out=90] (circ.225);
      \draw (j2) to[in=-90,out=90] (otimes2.225);
      \draw (e2) to[in=-90,out=90] (otimes2.-45);
      \draw (otimes2) to[in=-90,out=90] (circ.-45);
      \draw (circ) to (top);
    \end{tikzpicture}
    \displaybreak[1]\\
    &\underset{\eqref{eq:BraidedInterchange}}{=}
    \begin{tikzpicture}[smallstring, baseline=60]
      \node (v) at (0.5,0) {$u$};
      \node (v-bar) at (3.25,0) {$v$};
      \node (top) at (2.25,5) {$\cA(a \to \sF(u) \sF(v) a)$};
      \node[box] (eta1) at (0.5,2) {$\eta_u$};
      \node[box] (j1) at (1.5,2) {$j_{\sF(u)}$};
      \node[box] (j2) at (2.25,1) {$j_a$};
      \node[box] (eta2) at (3.25,1) {$\eta_{v}$};
      \node[box] (e2) at (4.25,2) {$e_{a, \sF(v)}$};
      \node[box] (otimes1) at (2.75,2) {$- \otimes -$};
      \node[box] (circ1) at (1,3) {$- \circ -$};
      \node[box] (circ2) at (3.5,3) {$- \circ -$};
      \node[box] (otimes2) at (2.25,4.1) {$- \otimes -$};
      \draw (v) to (eta1);
      \draw (eta1) to[in=-90,out=90] (circ1.225);
      \draw (j1) to[in=-90,out=90] (circ1.-45);
      \draw (v-bar) to (eta2);
      \draw (eta2) to[in=-90,out=90] (otimes1.-45);
      \draw (circ1) to[in=-90,out=90] (otimes2.225);
      \draw (otimes1) to[in=-90,out=90] (circ2.225);
      \draw (j2) to[in=-90,out=90] (otimes1.225);
      \draw (e2) to[in=-90,out=90] (circ2.-45);
      \draw (circ2) to[in=-90,out=90] (otimes2.-45);
      \draw (otimes2) to (top);
    \end{tikzpicture}
    \underset{\text{\cite[Cor.~4.7]{MR3961709}}}{=}
    \begin{tikzpicture}[smallstring, baseline=40]
      \node (v) at (0,0) {$u$};
      \node (v-bar) at (1,0) {$v$};
      \node (top) at (0.75,4) {$\cA(a \to \sF(u) \sF(v) a)$};
      \node[box] (eta1) at (0,1) {$\eta_u$};
      \node[box] (eta2) at (1,1) {$\eta_{v}$};
      \node[box] (j) at (2,1) {$j_a$};
      \node[box] (otimes1) at (1.5,2) {$- \otimes -$};
      \node[box] (otimes2) at (0.75,3) {$- \otimes -$};
      \draw (v) to (eta1);
      \draw (v-bar) to (eta2);
      \draw (eta2) to[in=-90,out=90] (otimes1.225);
      \draw (j) to[in=-90,out=90] (otimes1.-45);
      \draw (eta1) to[in=-90,out=90] (otimes2.225);
      \draw (otimes1) to[in=-90,out=90] (otimes2.-45);
      \draw (otimes2) to (top);
    \end{tikzpicture}
    =
    \begin{tikzpicture}[smallstring, baseline=35]
      \node (bot) at (0,0) {$u v$};
      \node (top) at (1,4) {$\cA(a \to \sF(u) \sF(v) a)$};
      \node[box] (eta) at (0,1) {$\eta_{v v}$};
      \node[box] (mu) at (1,1) {$\mu_{u v}$};
      \node[box] (circ) at (0.5,2) {$- \circ -$};
      \node[box] (j) at (1.5,2) {$j_a$};
      \node[box] (otimes) at (1,3) {$- \otimes -$};
      \draw (bot) to (eta);
      \draw (eta) to[in=-90,out=90] (circ.225);
      \draw (mu) to[in=-90,out=90] (circ.-45);
      \draw (circ) to[in=-90,out=90] (otimes.225);
      \draw (j) to[in=-90,out=90] (otimes.-45);
      \draw (otimes) to (top);
    \end{tikzpicture}
    \underset{\eqref{eq:BraidedInterchange}}{=}
    \begin{tikzpicture}[smallstring, baseline=35]
      \node (bot) at (0,0) {$u v$};
      \node (top) at (1.5,4) {$\cA(a \to \sF(u) \sF(v) a)$};
      \node[box] (eta) at (0,1) {$\eta_{u v}$};
      \node[box] (j1) at (1,1) {$j_a$};
      \node[box] (otimes1) at (0.5,2) {$- \otimes -$};
      \node[box] (mu) at (2,1) {$\mu_{u, v}$};
      \node[box] (j2) at (3,1) {$j_a$};
      \node[box] (otimes2) at (2.5,2) {$- \otimes -$};
      \node[box] (circ) at (1.5,3) {$- \circ -$};
      \draw (bot) to (eta);
      \draw (eta) to[in=-90,out=90] (otimes1.225);
      \draw (j1) to[in=-90,out=90] (otimes1.-45);
      \draw (otimes1) to[in=-90,out=90] (circ.225);
      \draw (mu) to[in=-90,out=90] (otimes2.225);
      \draw (j2) to[in=-90,out=90] (otimes2.-45);
      \draw (otimes2) to[in=-90,out=90] (circ.-45);
      \draw (circ) to (top);
    \end{tikzpicture}
  \end{align*}
  which is exactly the mate of the right hand side.
\end{proof}

The goal of the remainder of this section is to construct a tensored dagger $\sV$-module monoidal category $\sA$ from a dagger $\sV$-monoidal category $\cA$, extending the constructions of \cite{MR3961709,1809.09782} to the dagger setting.
(This can also be seen as extending Construction \ref{const:DaggerFromInvolutive}
for dagger enriched categories to the dagger $\sV$-monoidal setting.)
This extension amounts to proving:
\begin{enumerate}[label=(Z$^*$\arabic*)]
\item
  \label{Z:DaggerFunctor}
  the left adjoint $\sF$ is a dagger monoidal functor, and
\item
  \label{Z:HalfBraidingsUnitary}
  the half-braidings $e_{a,\sF(v)}$ are unitary.
\end{enumerate}
To prove these, we further assume that $(\cV, \kappa^\cV)$ is a dagger $\sV$-category so that \ref{eq:daglhd3} holds for $\sA$ as a right $\sV$-module by Proposition \ref{prop:daglhd3}.

\begin{lem}[\ref{Z:DaggerFunctor}]
  If $\cA$ is a dagger $\cV$-monoidal category,
  then the left adjoint $\sF$ is a dagger monoidal functor.
\end{lem}
\begin{proof}
  Since $\cA$ is a dagger $\sV$-category, and the right $\sV$-action on $\sA$ is given by $a\lhd u = a\sF(u)$, the result follows by taking $a=1_\sV$ in \ref{eq:daglhd2}.
  Since $\mu_{u,v} = \alpha_{1_{\cA}, u, v}$ by \cite[Rem.~6.13]{MR3270794}, unitarity of $\mu$ follows from unitarity of $\alpha$ from \ref{eq:daglhd3}.
\end{proof}

\begin{prop}[\ref{Z:HalfBraidingsUnitary}]
  Suppose $\cA$ is a dagger $\cV$-monoidal category.
  For all $a\in \sA$ and $v\in \sV$, $e_{a, \sF(v)}$ is unitary.
\end{prop}
\begin{proof}
  By Proposition \ref{prop:modulator-from-V-functor}, it suffices to find a dagger $\sV$ functor $\cG^a : \cA \to \cA$ for each object $a$ in $\sA$ such that $e_{a, \sF(v)}$ is the mate of $\eta_v \circ \cG_{1_{\cA} \to \sF(v)}^a$
  under the adjunction
  \begin{align*}
    \sA(a \sF(v) \to \sF(v) a) \cong \sV(v \to \cA(a \to \sF(v)a)).
  \end{align*}
  Define $\cG^a(b) := b a$ and
  \begin{align*}
    \cG_{b \to c}^a :=
    \begin{tikzpicture}[smallstring, baseline=20]
      \node(bot) at (0,0) {$\cA(b \to c)$};
      \node[box] (j) at (2,1) {$j_a$};
      \node[box] (otimes) at (1,2) {$- \otimes -$};
      \node (top) at (1,3) {$\cA(b a \to c a)$};
      \draw (bot) to[in=-90,out=90] (otimes.225);
      \draw (j) to[in=-90,out=90] (otimes.-45);
      \draw (otimes) to[in=-90,out=90] (top);
    \end{tikzpicture}\, .
  \end{align*}
  Clearly the mate of $e_{a, \sF(v)}$ is given by $\eta_{1_{\cA}, v} \circ \cG_{1_{\cA} \to 1_{\cA} \lhd u}^a$. That $\cG^a$ is a $\sV$-functor follows immediately from the braided interchange relation, and that $\cG^a$ is dagger follows immediately from \ref{kappa:5}.
\end{proof}

\begin{construction}
  \label{const:mod-to-mon-cat-1-cells}
  Let $(\cG, \mu^{\cG}) : \cA \to \cB$ be a dagger $\sV$-monoidal functor.
  Then as in \cite[\S5.2]{2104.07747} we get a morphism
  $(\sG, \mu^{\sG}, g) : (\sA, \sF_\sA) \to (\sB, \sF_\sB)$
  of the underlying $\sV$-module monoidal categories.
 First, we take $\sG$ to be the underlying functor of $\cG$; it follows as in Construction \ref{const:DaggerFromInvolutiveFunctor} that $\sG$ is a dagger functor.
 Second, we take $\mu^{\sG}$ to be componentwise equal to $\mu^{\cG}$ under
  \begin{align*}
    \sB(\sG(a) \sG(b) \to \sG(a b)) = \sV(1_{\sV} \to \cB(\cG(a) \cG(b) \to \cG(a b))).
  \end{align*}
  Observe that the dagger condition for a dagger $\sV$-monoidal functor $(\cG,\mu^\cG)$ is exactly unitarity of $\mu^{\sG}$.
  Finally, we take $g_v$ to be the mate of $\eta_v \circ \cG_{1_{\cA} \to \sF(v)}$ under
  \begin{align*}
    \sF_\sB(v) \to \cG(\sF_\sA(v))
    \cong
    \sV(v \to \cB(1_{\cB} \to \cG(\sF_\sA(v)))).
  \end{align*}
  This $g_v$ is unitary by Proposition \ref{prop:modulator-from-V-functor};
  indeed, defining the right $\sV$-module structure on $\sA,\sB$ by $a \lhd_\sA v := a\sF_\sA(v)$ and $b \lhd_\sB v := b\sF_\sB(v)$ respectively, $g_v$ is exactly equal to the modulator $\omega_{1_\sA, v}$ from Proposition \ref{prop:modulator-from-V-functor}.
  Thus $(\sG, \mu^{\sG}, g)$ is a 1-morphism of dagger $\sV$-module monoidal categories.
\end{construction}

\begin{construction}
  \label{const:mod-to-mon-cat-2-cells}
  Given dagger $\sV$-monoidal functors
  $(\cG, \mu^{\cG}), (\cH, \mu^{\cH}) : \cA \to \cB$
  and a $\sV$-monoidal natural transformation $\theta : \cG \to \cH$,
  we construct a monoidal natural transformation $\Theta : \sG \to \sH$ via
  \begin{align*}
    \Theta_a = \theta_a \in \sB(\sG(a) \to \sH(a)) = \sV(1_\sV \to \cB(\sG(a) \to \sH(a))).
  \end{align*}
\end{construction}

\begin{construction}
  \label{const:mon 2-functor}
  Using Constructions \ref{const:dagger-U-mon-cat-to-dagger-mtc}, \ref{const:mod-to-mon-cat-1-cells}, and \ref{const:mod-to-mon-cat-2-cells}, we can construct a map $\Phi_\otimes : \dagvmoncat \to \dagvmodmon$.
  That $\Phi_\otimes$ is a 2-functor is completely analogous to fact that $\Phi$ is a 2-functor from Construction \ref{const:2-functor}.
\end{construction}

%%%%%%%%%%%%%%%%%%%%%%%%%%%%%%%%%%%%%%%%%%%%%%%
\subsection{Equivalence}

For the remainder of the section, $\sU$ is a braided unitary monoidal category.%, and $(\cU,\kappa^\cU)$ is its dagger self-enrichment.

\begin{construction}
  \label{const:unitary-mtc-to-dagger-mon-U-cat}
  We now construct a rigid dagger $\sU$-monoidal category $\cA$ from a dagger $\sU$-module monoidal category $(\sA, \sF_\sA^Z)$. 
  We first define a dagger $\sU$-category $\cA$ as in Construction \ref{constr:U-mod-to-dagger-U-cat} above.
  We further add the tensor product morphism from \cite[\S6.3]{MR3961709}, which equips $\cA$ with the structure of a rigid tensored $\cU$-monoidal category; $- \otimes_{\cA} -$ is defined as the mate under Adjunction \ref{adj:mon-main} of 
  \begin{align*}
.
  \end{align*}
\end{construction}

\begin{defn}
  Recall that given a monoidal functor $\sF: \sA\to \sB$ between rigid monoidal categories, we have a canonical natural isomorphism
  \begin{align*}
    \delta_a:=
    %
    \,.
  \end{align*}
\end{defn}

\begin{rem}
  In the context of the specific action $a \lhd v := a \cF(v)$, the definition of $\kappa_{a \to b}$ in Construction \ref{eq:MateOfKappa} becomes
  \begin{equation} \label{eq:MonoidalMateOfKappa}
    \mate (\kappa_{a \to b}) :=
    %
    \,.
  \end{equation}
\end{rem}

\begin{lem}
  \label{lem:nu-delta-mu}
  Let $\sU$ be a unitary monoidal category and $(\sA, \sF)$ a dagger $\sU$-module monoidal category. Then
  \begin{align*}
    %
\, .
  \end{align*}
\end{lem}

\begin{proof}
  The right hand side in the statement of the lemma is given by
  \begin{align*}
    %
.
  \end{align*}
  where we added a zigzag in the second to last equality.
  Since $\mu$ is natural, this is exactly the left hand side in the statement of the lemma.
\end{proof}

\begin{prop}
  $\cA$ equipped with $-\otimes_\cA-$ is a tensored rigid dagger $\sU$-monoidal category, i.e., $(\cA,\kappa,\otimes)$ satisfies \ref{kappa:5}.
\end{prop}
\begin{proof}
  Set $u = \cA(c \to a)$ and $v = \cA(d \to b)$ for clarity of notation. Taking mates, we have 
  \begin{align*}
    &\mate \left(
    %
  \end{align*}
  Using Lemma \ref{lem:nu-delta-mu} and unitarity of $\mu$ followed by naturality of the half-braiding, this becomes
  \begin{align*}
    %
  \end{align*}
  which simplifies to exactly the mate of $(\kappa_{a \to c} \kappa_{b \to d}) \circ (- \otimes -)$.
  The last equality above uses the well-known identity
  \[
  e_{a,\sF(u)}^\dag
  \underset{\text{\ref{Z:HalfBraidingsUnitary}}}{=}
  e_{a,\sF(u)}^{-1}
  =
  \tikzmath[smallstring]{
    \node[box] (e) at (1.5,2) {$e_{a, \sF({u^*})}$};
    \node[box] (delta1) at (e.135 |- 0,3) {$\delta_u$};
    \node[box] (delta2) at (e.-45  |- 0,1) {$\delta_u^{-1}$};
    \node (v-bot) at (0,0) {$\sF(u)$};
    \node (a-bot) at (e.225 |- 0,0) {$a$};
    \node (v-top) at (3,4) {$\sF(u)$};
    \node (a-top) at (e.45 |- 0,4) {$a$};
    \draw (v-bot) to (v-bot |- delta1.90);
    \draw (v-bot |- delta1.90) to[in=90,out=90] (delta1);
    \draw (a-bot) to (e.225);
    \draw (e.135) to[in=-90,out=90] (delta1);
    \draw (e.45) to[in=-90,out=90] (a-top);
    \draw (delta2) to[in=-90,out=90] (e.-45);
    \draw (v-top) to (v-top |- delta2.-90);
    \draw (v-top |- delta2.-90) to[in=-90,out=-90] (delta2);
  }\,.
  \qedhere
  \]
\end{proof}

\begin{construction}
  \label{const:dagumod-to-dagumon-1-cells}
  Given a morphism $(\sG, \mu^\sG, g) : (\sA, \sF_\sA) \to (\sB, \sF_\sB)$ of dagger $\sU$-module monoidal categories, we get a dagger $\sU$-functor $\cG$ as in Construction \ref{const:mod-to-cat-1-cell}. 
  This $(\cG, \mu^\cG)$ was shown to be a $\sU$-monoidal functor in \cite[Prop.~6.8]{2104.07747} using the same tensorator $\mu^\cG := \mu^\sG$.
  Finally, it is clear $(\cG, \mu^\cG)$ is dagger as $\mu^\sG$ is unitary.
\end{construction}

\begin{thm}
  The 2-functor $\Phi_\otimes$ from Construction \ref{const:mon 2-functor} is a 2-equivalence.
\end{thm}

\begin{proof}
  As before, we need to show that $\Phi_\otimes$ is essentially surjective on 0-cells and 1-cells, and fully faithful on 2-cells.
  For 0-cells, given a 0-cell $(\sA, \sF_\sA) \in \dagumodmon$, we construct a 0-cell $\cA \in \dagumoncat$ via Construction \ref{const:unitary-mtc-to-dagger-mon-U-cat}.
  It suffices to show that $(\sA', \sF_{\sA'}) := \Phi_\otimes(\cA)$ is isomorphic to $(\sA, \sF_\sA)$ as dagger $\sU$-module monoidal categories.
  By the proof of Proposition \ref{prop:equivalence-of-dagger-U-mods}, we have an isomorphism of dagger categories $(\sG, \omega') : \sA \to \sA'$,
  and in \cite[Thm.~7.3]{MR3961709} it was shown that $\sG$ equipped with an identity tensorator and action coherence morphism is an isomorphism of $\sU$-module monoidal categories.
  Since identities are unitary, it follows that $(\sG, \id, \id)$
  is isomorphism of dagger $\sU$-module monoidal categories between $\Phi_\otimes(\cA)$ and $(\sA, \sF_\sA)$.

  Suppose $\cA, \cB \in \dagumoncat$ are 0-cells, and $(\sG, \mu^\sG,g) \in \dagumodmon(\Phi_\otimes(\cA) \to \Phi_\otimes(\cB))$ is a 1-cell.
  We get a 1-cell $(\cG, \mu^\cG) \in \dagumoncat(\cA \to \cB)$ via Construction \ref{const:dagumod-to-dagumon-1-cells}.
  It is shown in \cite[Prop.~6.9]{2104.07747} that $\Phi_\otimes(\cG, \mu^\cG)$ and $(\sG, \mu^\sG, g)$ are equal as $\sU$-module monoidal functors.
  Being a dagger $\sU$-module monoidal functor is a property of a $\sU$-module monoidal functor, so they are also equal as dagger $\sU$-module monoidal functors.

  The 2-cells in $\dagumoncat$ and $\dagumodmon$ are identical to those in $\vmoncat$ and $\vmodtens$, respectively, so proving that $\Phi_\otimes$ is fully faithful on 2-cells is identical to the proof of Theorem \cite[Prop.~6.12]{2104.07747}.
\end{proof}

%%%%%%%%%%%%%%%%%%%%%%%%%%%%%%%%%%%%%%%%%%%%%%%%%%
\appendix

\section{Modulator construction}
\label{app:modulators}

In this appendix, we construct the previously omitted modulators from \cite[\S3.3]{1809.09782} for the isomorphisms in Proposition \ref{prop:equivalence-of-dagger-U-mods}.
Note that the proofs are similar to those of \cite[Lemma 4.1]{1809.09782},
but there is no clear way to define a $\sV$-functor that allows us to use that result directly.

As shown in Lemma \ref{lem:modulator inverse} below, $\omega_{a, v}^{-1} = \sH(\omega_{a, v}')$.
Thus if $\omega$ is a modulator for $\sH$, then $\omega'$ is a modulator for $\sG$,
and the composites $(\sH, \omega) \circ (\sG, \omega')$ and $(\sG, \omega') \circ (\sH, \omega)$ are identity module functors.
The proof that $\omega$ is unital is straightforward;
proofs for naturality and associativity appear in Lemmas \ref{lem:modulator-natural} and \ref{lem:modulator-associative} below.

\begin{lem}
  \label{lem:modulator inverse}
  The modulator $\omega_{a, v}$ has inverse $\sH(\omega_{a, v}')$.
\end{lem}

\begin{proof}
  Setting $u = \cA(a \to a \lhd v)$ for readability, $\omega_{a, v} \circ \sH(\omega_{a, v}')$ is given by
  \begin{align*}
    &
\, ,
  \end{align*}
  which simplifies to $\id_{a \lhd v}$.
\end{proof}

\begin{lem}
  \label{lem:modulator-natural}
  The modulator $\omega$ is natural, i.e., $\omega_{a, u} \circ \sH(f \blacktriangleleft g) = (\sH(f) \lhd g) \circ \omega_{b, v}$.
\end{lem}

\begin{proof}
  Take $f \in \sA' (a \to b)$ and $g \in \sV(u \to v)$. Then $(\sH(f) \lhd g) \circ \omega_{b, v}$ is equal to
  \begin{align*}
    %
\, .
  \end{align*}
  This is exactly $\omega_{a,v}^{-1} \circ \sH(f \blacktriangleleft g)$.
\end{proof}

\begin{lem} \label{lem:modulator-associative}
  The modulator $\omega$ is associative:
  $
  %
\, ,
  \end{align*}
  which is exactly $\omega_{a, u v} \circ \sH(\alpha_{a, u, v}')$.
\end{proof}

Finally, the following lemma and remark are used in Lemma \ref{lem:modulator unitary} in the main text.

  \begin{lem}
    \label{lem:omega-V-natural}
    The modulator $\omega'$ is a $\sU$-natural transformation between $(a \lhd -)$ and $(a \blacktriangleleft -)$.
  \end{lem}

  \begin{proof}
    We check that
    \begin{align*}
      %
      =
      (a \lhd -)_{u \to v}.
    \end{align*}
    Taking mates into $\sA'$ under Adjunction \ref{eq:CatToModAdjunction}, the left hand side becomes
    \begin{align*}
      %
\, .
    \end{align*}
    Next, it follows from \ref{mate:L-composite} that the mate of
    $g \in \sA(a \lhd u \lhd \cV(u \to v) \to a \lhd v) \cong \sA'(a \lhd u \lhd \cV(u \to v) \to a \lhd v)$
    is
    $\omega'_{a, u} \circ \sG(g)$,
    so the last diagram here is the mate of
    $\alpha_{a, u, \cV(u \to v)}^{-1} \circ (\id_a \lhd \ev_u \id_v)$.
    Since this is exactly the mate of $(a \lhd -)_{u \to v}$ under Adjunction \ref{eq:CatToModAdjunction}, we have proven the lemma.
  \end{proof}

  \begin{rem}
    \label{rem:eta-eta'}
    When we consider $\omega_{a, v}'$ as an element of $\sV(1_{\sV} \to \cA(a \blacktriangleleft v \to a \lhd v)) = \sA'(a \blacktriangleleft v \to a \lhd v)$,
    it is a direct consequence of \eqref{eq:MateOf1lhdg} that $(\eta_{a, v}'  \omega_{a, v}') \circ (- \circ_\cA -) = \eta_{a, v}$.
  \end{rem}

%%%%%%%%%%%%%%%%%%%%%%%%%%%%%%%%%%%%%%%%%%%%%%%%%
\bibliographystyle{alpha}
                  {\footnotesize{
                      %\bibliography{../bibliography}
                      \bibliography{../bibliography/bibliography}
                  }}

\end{document}